\documentclass[11pt]{article}
\usepackage{amsfonts,amsmath,latexsym,amssymb,euscript,graphicx,units,mathrsfs,amsthm,enumerate}

\usepackage{caption,subcaption, url}

\usepackage{hyperref}
\usepackage{xr}

\newcommand{\Colon}{\colon\>}   

\newcommand{\sth}{\mathrel{;}}         
\newcommand{\dl}{\mathrm{d}}           

\DeclareMathOperator{\Cont}{\mathcal C}
\newcommand{\comp}{\mathrm{c}}  

\DeclareMathOperator{\ACL}{ACL}

\DeclareMathOperator{\dom}{dom}

\DeclareMathOperator{\Lip}{Lip}

\newcommand{\Id}{\mathrm{Id}}

\newcommand{\loc}{\mathrm{loc}}

\newcommand{\FFF}{\mathcal{F}}
\newcommand{\HHH}{\mathcal{H}}

\newcommand{\AAA}{\mathcal{A}}

\usepackage{geometry}

\usepackage{graphicx}
\usepackage{xcolor}
\usepackage{amssymb}
\usepackage[OT1]{fontenc}
\usepackage{latexsym}
\usepackage{euscript}
\usepackage{amsfonts,amsmath}
\usepackage{verbatim}
\usepackage{fancyhdr}
\usepackage[english]{babel}
\usepackage{mathrsfs}
\usepackage{units}

\newtheorem{prop}{Proposition}[section]
\newtheorem{cor}[prop]{Corollary}
\newtheorem{lemma}[prop]{Lemma}

\newtheorem{rem}[prop]{Remark}

\newtheorem{thm}[prop]{Theorem}

\newtheorem{theorem}[prop]{Theorem}
\newtheorem{definition}[prop]{Definition}
\newtheorem{example}[prop]{Example}

\newcommand{\A}{\mathcal A}

\newcommand{\e}{\varepsilon}

\renewcommand{\geq}{\geqslant}
\def\leq{\leqslant}

\newcommand{\N}{\mathbb{N}}

\newcommand{\R}{\mathbb{R}}

\newcommand{\E}{\mathbb{E}}
\renewcommand{\P}{\mathbb{P}}
\newcommand{\eq}{\begin{equation}}
\newcommand{\qe}{\end{equation}}
\DeclareMathOperator{\dv}{div}

\newcommand{\grad}{\qopname \relax o\nabla \!\null}

\newcommand{\norm}[1]{{\left\| #1 \right\|}}

\newcommand{\opnorm}[1]{{\left\| #1\right\|_{\mathrm{op}}}}

\begin{document}

\title{{Stein's density method}
for multivariate
  continuous distributions} \author{G. Mijoule\footnote{INRIA Paris,
    Team MoKaPlan, 2 rue Simone Iff, 75012 Paris,
    guillaume.mijoule@gmail.com}, \;
    M. Rai\v{c}\footnote{University of Ljubljana, Faculty of Mathematics and Physics, Jadranska 19, SI--1000 Ljubljana, Slovenia; University of Primorska, Faculty of Mathematics, Natural Sciences and Information Technologies, Glagolja\v{s}ka 8, SI--6000 Koper, Slovenia; Institute of Mathematics, Physics and Mechanics, Jadranska 19, SI-1000 Ljubljana, Slovenia;
    martin.raic@fmf.uni-lj.si}, \; G. Reinert\footnote{University of
    Oxford, Department of Statistics, 24--29 St Giles, Oxford OX1 3LB, and
    The Alan Turing Institute, 96 Euston Rd, London NW1 2DB, UK; reinert@stats.ox.ac.uk} \; and
  Y. Swan\footnote{Universit\'e libre de Bruxelles, D\'epartement de
    Math\'ematique, Campus Plaine, Boulevard du Triomphe CP210, B--1050
    Brussels, yvik.swan@ulb.be} }

\date{}
\maketitle

\abstract{This paper provides a general framework for Stein's density
  method for multivariate continuous distributions. The approach
  associates to any probability density function a canonical operator
  and Stein class, as well as an infinite collection of operators and
  classes which we call standardizations. These in turn spawn an
  entire family of Stein identities and characterizations for any
  continuous distribution on $\R^d$, among which we highlight those
  based on the score function and the Stein kernel. A feature of these
  operators is that they do not depend on normalizing constants. A new
  definition of Stein kernel is introduced and examined; integral
  formulas are obtained through a connection with mass transport, as
  well as ready-to-use explicit formulas for elliptical distributions.
  The flexibility of the kernels is used to compare in Stein
  discrepancy (and therefore 2-Wasserstein distance) between two
  normal distributions, Student and normal distributions, as well as
  two normal-gamma distributions.  Upper and lower bounds on the
  1-Wasserstein distance between continuous distributions are
  provided, and computed for a variety of examples: comparison between
  different normal distributions (improving on existing bounds in some
  regimes), posterior distributions with different priors in a
  Bayesian setting (including logistic regression), centred
  Azzalini--Dalla Valle distributions. Finally the notion of weak Stein
  equation and weak Stein factors is introduced. Bounds for solutions of the weak Stein equation are obtained
  for Lipschitz test functions if the distribution admits a
  Poincar\'e constant. {We use these bounds} 
  to compare different copulas on the unit square {in 1-Wasserstein
  distance}. }

\medskip


\section{Introduction}
\label{sec:introduction}

Stein's method of approximate computation of expectations is a
collection of tools permitting to bound \emph{integral probability
  metrics} {or \emph{Zolotarev metrics}}
$$d_{\mathcal{H}}(X, W) = \sup_{h \in \mathcal{H}} \left|
  \E[h(W)] - \E[h(X)] \right|$$ where $\HHH$ is a
measure-determining class of functions and $X, W$ are two random
quantities of interest with $X$, say, {{distributed according to}  the \emph{target
  distribution} {over which it is assumed that we have some handles}.
In order to implement this method for $X$, one needs first to
find a linear operator ${\AAA}$, called  a \emph{Stein operator},  and a
class of functions
$\FFF(\AAA)$, called a \emph{Stein class}, such that (i)
the \emph{Stein identity} $ \E [ \AAA f(X)] = 0$  holds for all functions $f \in \FFF(\AAA)$ and (ii)
for each $h \in \HHH$, there exists a solution
$f = f_h \in \FFF(\AAA)$ of the {{\it Stein equation}}
$\AAA f(w) = h(w) - \E [h(X)].$ Then, upon noting that
$ \E [h(W)] - \E [h(X)] = \E [\AAA f_h (W)]$ for all $h$,
the problem of bounding $d_{\mathcal{H}}(X, W)$ has been re-expressed
in terms of that of bounding the {\emph{Stein discrepancy}}
$ \sup_{h \in \mathcal{H}} {|} \E [\AAA f_h (W)] {|}$. The {popularity  of
Stein's} method lies {for a large  part} in the fact that%
{, in many important cases,} {Stein discrepancies are}  amenable to a wide variety of approaches. {Moreover} 
the {various equations, operators
    and  discrepancies} {appearing in Stein's method}
  {provide mathematical objects   which 
  {can be} in themselves} of intrinsic and
  independent
  interest}. 

There exist many frameworks in which Stein's method is well
  understood, particularly for univariate distributions.  Comprehensive
  introductions to some of the most important aspects of the theory
  are available from the monographs \cite{stein1986} as well as
  \cite{NP11,ChGoSh11}, with a particular focus on Gaussian
  approximation. We also refer to
  the survey \cite{Ro11} as
  well as the papers \cite{LRS16,ley2015distances}.

  Although the univariate case is the most studied, many references
  also tackle multivariate distributions. The starting point for a
  multivariate Stein's method is a \emph{Stein characterization} for the
  Gaussian law which states that a random vector {$Z$} is a multivariate
  Gaussian $d$-random vector with mean $\nu$ and covariance $\Sigma$
  (short: $Z \sim {\cal{N}}(\nu, \Sigma)$) if and only if
\begin{equation}
  \label{eq:stlemma}
   \E\big[ \left\langle \Sigma,  \grad^2 {g}(Z)
   \right\rangle_{\mathrm{HS}}\big] =  \E \big[\left\langle Z-\nu, \grad {g}(Z) \right\rangle \big],
\end{equation}
for all absolutely continuous function ${g} \Colon \R^d \to \R$ for which
the expectations exist; here $\grad$ denotes the gradient operator
{for functions which act on $\R^d$} (a column vector), {$\left\langle A, B \right\rangle_{\mathrm{HS}} = \sum_{i,
  j}A_{ij}B_{ij}$ is the Hilbert--Schmidt inner product {between $d
  \times d$ matrices} and $\left\langle u, v \right\rangle =
\sum_{i} u_i v_i$ is the usual scalar product between vectors in
$\mathbb{R}^{d}$}. From
\eqref{eq:stlemma}, one infers that the second order differential
operator
$\mathcal{A}{g}(x) =
\left\langle \Sigma, \grad^2 {g}(x) \right\rangle_{\mathrm{HS}} -
\left\langle x - \nu, \grad {g}(x) \right\rangle$  is a
Stein operator for the Gaussian distribution {${\cal{N}}(\nu, \Sigma)$}, with {Stein} class
$\mathcal{F}(\mathcal{A})$ the collection of functions for which
\eqref{eq:stlemma} holds. Classical Markov theory then provides a
solution of the corresponding Stein equation
{$\mathcal{A}g(x) = h(x) - \E[h({Z})]$} via the so-called Mehler
formula, leading to a variety of applications of Stein's method for
multivariate normal approximation as treated e.g.\ in \cite{MR1035659, G91,
  GoRi96, MR2573554, meckes2009stein,   bonis2016rates}.  Classical
Markov theory also provides Stein operators, equations, and solutions
outside the Gaussian case, e.g.\ for log-concave densities as well as
for ergodic {measures} of {stochastic differential equations} 
satisfying regularity conditions, as
studied e.g.\ by \cite{mackey2016multivariate,
  gorham2016measuring,fang2018multivariate}. {Other references
  extending the method beyond the multivariate Gaussian setting
  include \cite{arras2017stein,arras2018stein} 
  {for} infinitely
  divisible distributions with finite first moment, and}
\cite{barbour2015multivariate, barbour2018multivariate} 
{for} a
framework which is applicable 
{to} many discrete multivariate
distributions. In \cite{reinert2019approximating}, stationary
distributions of Glauber Markov chains are characterized.  The
Dirichlet distribution has been treated in \cite{gan2017dirichlet},
and the multinomial distribution is considered in
\cite{loh1992stein}. {This literature review is not intended to be
  exhaustive; {more references can be found e.g.\ at the websites}
  https://sites.google.com/site/steinsmethod and
  https://sites.google.com/site/malliavinstein}.

Several {general approaches} to {finding Stein operators}
exist in the univariate setting, {among which we highlight Barbour's
  \emph{generator approach} (which makes use of Markov generator
  theory, see \cite{G91,MR1035659}) and Stein's \emph{density
    approach} (which makes use of the properties of the underlying
  density, see \cite{stein1986, stein2004} as well as
  \cite{Do14,upadhye2014stein,ley2015distances,LRS16}). The generator
  approach extends naturally to the multivariate setting. The density
  approach has, so far, not been extended to the multivariate
  setting.}  {The {aim} of this paper is to fill this gap and to
  explore implications of  Stein's density method in the multivariate
  setting. In particular this paper highlights the considerable
  flexibility in Stein operator choice which includes operators which
  are not gradient operators, thus complementing the generator
  approach in the multivariate setting by offering an alternative
  which does {not} require the analysis of a Markov operator.}

\subsection{Stein's multivariate density method}
\label{sec:overview-results}

The probability distributions in this paper are assumed to admit a
probability density function (pdf) {$p$} with respect to Lebesgue
measure {on  $\R^d$}. {The expectation under $p$ is denoted by
$\E_p$; for a function $h$,  we write $\E_p h = \E \bigl[ h(X) \bigr] $
with $X$ having pdf $p$.} The building blocks of our construction are the
\emph{canonical directional Stein derivatives}
$$ f \mapsto \mathcal{T}_{e, p} f = \frac{\partial_e( p \, f)}{p};$$ the class
of functions which satisfy the {``}elementary{''}
Stein identity
$\E_p [\mathcal{T}_{e, p}f] = 0$ 
is then {denoted
  $\mathcal{F}_1(p)$} {(Section \ref{sec:standardizations} provides many more Stein identities which are perhaps not as elementary)}.  Such {pairs
  $(\mathcal{T}_{e, p}, \mathcal{F}_1(p))$} can be defined in any direction $e${, and we collect in the
  \emph{canonical Stein class}
  $\mathcal{F}(p)$ all functions for which the elementary Stein
  identity holds for every component in (almost) every direction. {More precise definitions, along with clarifications of the various functional spaces we consider,  will be provided below.} The
  resulting objects} can be combined in virtually endless
possibilities. In this spirit, we introduce the \emph{Stein gradient
  operator}
$\bullet \mapsto \mathcal{T}_{\grad, p} \bullet ={\grad \left( p
    \bullet \right)} / {p}$ which many authors call \emph{the} Stein
operator for multivariate $p$.  Similarly, we define the \emph{Stein
  divergence operator} acting on vector- or matrix-valued functions
with compatible dimensions as
$\bullet \mapsto \mathcal{T}_{\dv, p} \bullet =
{\dv\left( p \bullet \right)}/ {p}.$
{More generally,}  a \emph{standardization} of the canonical operator  is
any operator of the form
$\AAA \bullet = \sum_{i=1}^d{\mathbf{A}}_i \mathcal{T}_{e_i, p}
({\mathbf{T}}_i \bullet)$ where $e_{i}, i=1,\ldots, d$ are the unit
vectors in $\R^d$ and $\mathbf{A}_i, \mathbf{T}_i, i = 1, \ldots, d$
are some linear operators{.}

{
{T}he freedom of choice in the standardizations is only as
  useful as one can find operators $\mathbf{A}_i, \mathbf{T}_i$
  leading to tractable expressions for $\mathcal{A}$. Many
  distributions $p$ have a tractable score function $\grad \log p$
  (including, of course, the Gaussian distribution but also
  \emph{elliptical distributions}, convolutions of independent
  components, etc.). In such cases a direct application of the Stein
  gradient {or} divergence operators leads to tractable first order
  vector-valued operators of the form
  $\mathcal{A}_1 g = \mathcal{T}_{\grad, p} (\Sigma g) = \Sigma \grad
  g + (\Sigma \grad \log p)g $ ({where} $\Sigma \in \R^{d \times d}$)  
  with class $\mathcal{F}(\mathcal{A}_1) = \mathcal{F}_1(p)$,
  as well as second order scalar-valued operators
  $ \tilde{\mathcal{A}}_1g = \mathcal{T}_{\dv, p} (\Sigma
  \grad g) = \left\langle \Sigma, \grad^2 g
  \right\rangle_{\mathrm{HS}} + \left\langle \Sigma \grad \log p,
    \grad g \right\rangle$ acting on
  $\mathcal{F}(\tilde{\mathcal{A}}_1)$ the collection of functions
  such that $\grad g \in \mathcal{F}({\mathcal{A}}_1)$.}
The Gaussian Stein operator {derived from} \eqref{eq:stlemma} is of this
second order form. Much of the more applied literature on
multivariate Stein's method focuses on such second order
  operators,
because $\grad \log p$ (and hence the resulting operator) does not
depend on the normalizing constants of $p$, which is very useful e.g.\
in the study of posterior densities in a Bayesian context. This
independence of the normalizing constants is
   inherited from the
  nature of the canonical directional Stein derivatives and is a
  feature of all operators provided by Stein's density method.

  Another family of standardizations with which much of the paper is
  concerned {relates to the \emph{Stein kernel}.}
First studied in \cite{stein1986}
{(although already earlier used to tackle} a smooth pdf $p$, see
e.g.\ \cite{C82} 
{where it is referred to} as a covariance kernel), {this
  quantity} has become {an important component}
of univariate Stein's method, see e.g.~\cite{NP11,
  ernst2020first,saumard2018weighted} for an overview. In dimension
$d=1$, the Stein kernel for a pdf $p$ is
the unique solution $x \mapsto {\tau_p}(x)$ in $\mathcal{F}(p)$ of the
{ordinary differential equation} $$ \mathcal{T}_p({\tau_p}(x)) {:= (\tau_p(x) p(x))'/p(x)} =\nu-x$$ {with} $\nu$
 the
mean of $p$. It exists and is unique when the distribution has finite
  variance {-- the Stein kernel is the zero bias density from \cite{GoRei97}}.    
  {Moreover,} for $X \sim p$,
  $\E[\tau_p(X) g'(X)] = \E[(X -\nu) g(X)]$ for all
  $g$ such that $\E |(X - \nu) g(X) | <\infty$.
  {Uniqueness} of
the kernel is lost for multivariate distributions and several
concurrent definitions exist.
In  \cite{nourdin2013integration,courtade2017existence},
    a Stein
  kernel $\pmb \tau$ for a pdf  $p$ is defined as any matrix-valued
  function satisfying
 \begin{equation} \label{eq:eqsteinidddddd1}
  \E \left[ \left\langle \pmb{\tau}(X), \grad g(X) \right\rangle_{\mathrm{HS}}
  \right] = \E \left[  \left\langle  X-\nu, g(X) \right\rangle\right]
 \end{equation}
at least for all smooth $g\Colon \R^d \to \R^d $ with compact support.
In \cite{ledoux2015stein}, a Stein
  kernel $\pmb \tau$ for a pdf  $p$ is  required to satisfy \eqref{eq:eqsteinidddddd1}
  only for smooth functions which are gradients, to read
 \begin{equation} \label{eq:eqsteinidddddd}
  \E \left[ \left\langle \pmb{\tau}(X), \grad^2 g(X) \right\rangle_{\mathrm{HS}}
  \right] = \E \left[  \left\langle  X-\nu, \grad g(X) \right\rangle\right] .
 \end{equation}
 Associated with this definition is a second order scalar valued
 operator ${\mathcal{A}}_2g(x) = \left\langle \pmb \tau(x),
   \grad^2g\right\rangle_{\mathrm{HS}} - \left\langle x-\nu, \grad
   g(x)
 \right\rangle$.  Comparing \eqref{eq:eqsteinidddddd} with
 \eqref{eq:stlemma} shows that the covariance matrix
 $\Sigma$ is a Stein kernel for the $\mathcal{N}(\nu,
 \Sigma)$ distribution; {hence the Stein-kernel operator
   ${\mathcal{A}}_2$ also generalizes the Gaussian operator.
   Moreover, the difference $\pmb \tau -
   \Sigma$ has been used as a Gaussian discrepancy metric, see e.g.\
   \cite{nourdin2013integration,courtade2017existence,
     ledoux2015stein}, where it is shown that it captures some
   essential features of the underlying distribution, relating Stein
   kernels with
   log-Sobolev inequalities,
  Poincar\'e constants and moment maps.

  In this paper we adopt a direct approach and define a
  {\it{directional Stein kernel}} for each canonical direction $e_i
  \in \R^d$, as any differentiable function $x \mapsto \tau_{p,i}(x)
  \in \R^d$ such that {$ \mathcal{T}_{\dv, p}\big(\tau_{p,
    i}\big)(x) =\E[\left\langle X, e_i \right\rangle]-
  x_i$} for Lebesgue almost all $x$ in the support of
  $p$.  A \emph{Stein kernel} is then \emph{any} matrix-valued
  function $ \pmb\tau$ such that each row $ \pmb \tau_i = (\tau_{i1},
  \ldots, \tau_{id})$ is a kernel in the direction
  $e_i$, different kernels leading to different associated operators.
  {This} definition
  opens the way for identifying several Stein kernels (and therefore
  Stein operators) for any given pdf,
   even for the  Gaussian
  $\mathcal{N}(\nu, \Sigma)$: Example \ref{sec:stein-oper-Gauss} will
  show that,
aside from the Stein kernel $ \pmb\tau = \Sigma,$
 if $d \neq 1$ then
 the matrix-valued functions
\begin{equation*}
    \pmb \tau_{{\delta}}(x) =  \frac{{1}}{(2 - \delta)(d-1)} \left(\left(
      d-1 +  ({1 - \delta})  x^T x \right) \Sigma - ({1 - \delta}) (x - \nu)(x-\nu)^T
      \right)
\end{equation*}
are Gaussian Stein kernels for all $\delta \neq 2$.  More generally,
  we 
  {find} Stein kernels for elliptical distributions and 
  formulas allowing to obtain Stein kernels for {smooth multivariate}
  distributions {with densities} which {are} available up to a
  normalizing constant.

{As our formalism provides an infinite family of Stein operators and
  classes $ (\mathcal{A}, \mathcal{F}(\mathcal{A}))$ for any
  distribution
 ${P}$ having a pdf (even intractable distributions),
  a variety of versions of Stein's method of distributional
  approximation {can be deployed} by considering quantities of the
  form
  $\mathcal{S}(q, \mathcal{A}, \mathcal{G}) = \sup_{g \in \mathcal{G}}
 {\| \E_q \mathcal{A}g \|}$ for $q$ some distribution
  of interest, $\mathcal{A}$ any Stein operator {for $P$}  and $\mathcal{G}$ a
  well-chosen class of functions. 
  These quantities are called \emph{Stein
    discrepancies} and have proven to be particularly useful when
  studying questions of convergence to equilibrium, see e.g.\
  \cite{gorham2016measuring}.
  The freedom of choice and ease of use of our operator approach
  $\mathcal{A}$ now allows for optimization over  possible Stein
  operators $\mathcal{A}$ for $P$ in
   $\mathcal{S}(q, \mathcal{A}, \mathcal{G})$.

   In a first application, in Section \ref{sec:stein-kern-discr}, we concentrate on
  differences of Stein kernels, which provides a
  general and easy to use discrepancy metric; our examples include
  comparison of Gaussians (Example \ref{ex:comaoighe}), comparison of Student and
  Gaussian (Example \ref{ex:gauvsstudent}), and comparison of normal-gamma
  distributions (Example \ref{ex:normalgamma}). {In the second
    example we also exploit the freedom of choice in the Stein kernels
    to optimize the resulting discrepancy, hereby demonstrating the
    use of disposing of several kernels for a fixed comparison problem}.

In a second application,  in Section \ref{sec:stein-discrepancies} we {estimate}
discrepancies between
distributions $P_1, P_2$ {with  supports $\Omega_1 $ and $ \Omega_2$ which are open subsets of $\R^d$ and nested, so that $\Omega_2 \subseteq \Omega_1$. We express the discrepancies} in terms of the (1-)Wasserstein
distance
\begin{equation*}
  {\mathcal{W}}_1(P_1, P_2)
  =
  \sup_{h \in \Lip({\Omega_1}, 1)} \left|
    \int h \, \dl P_1 - \int h \, \dl P_2
  \right| \, ,
\end{equation*}
{where} $\Lip(\Omega, 1)$ is
the collection of
Lipschitz functions {$\Omega \rightarrow \R$ with} 
Lipschitz constant 1.
  For $P_1$
and $P_2$ with pdfs $p_1, p_2$ and chosen
Stein operator and class
$(\mathcal{A}_{p_1}, \mathcal{F}(\mathcal{A}_{p_1}))$ for $p_1$ and
$(\mathcal{A}_{p_2}, \mathcal{F}(\mathcal{A}_{p_2}))$ for $p_2$,
fixing $P_1$ as the target, we consider the \emph{Stein
  equations}
 \begin{equation}
 \label{eq:steqgen2}
 \mathcal{A}_{p_1}g = h -  \E_{p_1}h; \quad h \in \Lip({\Omega_1,}1).
\end{equation}
Taking
expectations with respect to  $P_2$ on either side of
 \eqref{eq:steqgen2} shows
 that the Wasserstein distance
 satisfies
 \begin{equation}\label{eq:diststeindis}
   \mathcal{W}_1(P_1, P_2) =  \sup_{g \in
   \mathcal{G}(\Lip({\Omega_1,} 1))} \left|\E_{p_2} \left[
     \mathcal{A}_{p_1} \, g \right]\right|
 \end{equation}
 with $\mathcal{G}(\Lip({\Omega_1,}1))$
 the collection of all solutions
 of \eqref{eq:steqgen2} which belong to
 $\mathcal{F}(\mathcal{A}_{p_1})$.
 Moreover,  if
$\mathcal{A}_{p_2}$ for $p_2$ is chosen such that
$\mathcal{G}(\Lip({\Omega_1,}1)) \subseteq \mathcal{F}(\mathcal{A}_{p_2})$
{then} $\E_{p_2} \left[ \mathcal{A}_{p_2} \, g \right] = 0$ and
$\mathcal{W}_1(P_1, P_2) = \sup_{g \in \mathcal{G}(\Lip({\Omega_1,}1))}
\left|\E_{p_2} \left[ (\mathcal{A}_{p_1}-\mathcal{A}_{p_2}) \, g
  \right]\right|.$ The freedom of choice in $\mathcal{A}_{p_i}$,
$i=1, 2$, makes this last expression a good starting point for
comparison in Wasserstein distance{: 
Theorem
  \ref{theo:comap} {provides} a general bound on the Wasserstein distance between
  different pdfs on $\R^d$
  {which} we particularize 
  in Proposition
  \ref{cor:wass-dist-betw-1}  
  to  even obtain lower bounds (depending on the Stein kernel).
  We apply these results to a number of concrete
  applications; in Section~\ref{sec:stein-discrepancies} we compare in
  Wasserstein distance different normal distributions
  (Example~\ref{ex:compafonormales}), posterior distributions with different
  priors in a Bayesian setting (Examples \ref{ex:normvsunifom} and
  \ref{ex:bayeslog}); we also study the effect of the skewness
  parameter on centred Azzalini--Dalla Valle distributions where our
  upper and lower lead to an explicit expression for the  Wasserstein distance (Example \ref{ex:azzal}).}

The above approach to distributional comparisons with Stein's method
relies on a good understanding of Stein equations \eqref{eq:steqgen2}
and their solutions.  Much is already known about the regularity
properties (often referred to as {\it Stein factors}) of these
solutions in several important settings; see Example \ref{ex:steinfac}
for Gaussians, Example \ref{lem:bouddd2} for log-concave pdfs, and
Example \ref{ex:mcoahsoin} for more general results.  {Such
  ready-to-use regularity properties are not always available and,
  f}inally, as a further contribution of the paper,
  Proposition~\ref{prop:l2bound} shows that \eqref{eq:diststeindis} also holds under
a weakened form of the Stein {equation} \eqref{eq:steqgen2}, namely
\eqref{eq:weaksteinequ}, which states
\begin{equation*}
   \E_{p_2} \left[ \A_{p_1} g \right]
   =
   \E_{p_2} [  h - \E_{p_1}h].
\end{equation*}
Solutions of this equation exist under the relatively weak assumption of
existence of a Poincar\'e constant for $P_1$.
This {observation} allows to provide bounds even when regularity
properties of solutions of Stein equations are hard to establish,
e.g.\ for Wasserstein distance between copulas on the unit
square (Example \ref{ex:comparingcopulas}).

{

  \subsection{Overview of the paper}
  \label{sec:overview-paper}

  The paper is structured as follows. In Section \ref{sec:notations-1}
  we provide  notations  that are
  used throughout the paper. The multivariate Stein's density method
  is presented and studied in Section \ref{sec:first-defin-outl},
  first by defining the canonical directional Stein operators and
  classes (Section \ref{sec:general-theory}), then by providing
  sufficient conditions for obtaining Stein identities and
  characterizations (Section \ref{sec:stein-identities}), {and}
  finally by making the connection with current literature through the
  concept of standardizations of the canonical operators (Section
  \ref{sec:standardizations}).  Section \ref{sec:stein-kernel} is
  devoted to the study of Stein kernels, first in general (Section
  \ref{sec:generalities}) then under an assumption of ellipticity of
  the distribution (Section \ref{sec:stein-kern-ellipt-1}); first
  applications towards distributional approximation are also outlined
  (Section \ref{sec:stein-kern-discr}).  Section
  \ref{sec:stein-discrepancies} contains the main applications, namely
  a flexible set of bounds on Wasserstein distance between densities
  admitting a Stein operator {under an
  additional assumption of nested supports} (Section \ref{subsec:gencomp}).
  {{T}he idea of \emph{weak Stein
      equations} and {\it{weak Stein factors}} {is introduced in Section \ref{sec:stein-factors-under}, leading to bounds on
    Wasserstein distance under {the} assumption of existence of a
    Poincar\'{e} constant; {these bounds are illustrated by  comparing copulas.}}

\bigskip

This paper is a complete overhaul
of {a} paper that previously appeared on the arXiv
(https://arxiv.org/abs/1806.03478) {which includes more material,}
such as an excursion into kernelized Stein discrepancies and
goodness-of-fit tests.

}

\section{Notations, gradients and product rules}
\label{sec:notations-1}

We first settle the notation that will be used throughout the paper.
Fix $d \in \N_0$ and let $e_1, \ldots, e_d$ be the canonical basis for
Cartesian coordinates in $\R^d$. Vectors of $\R^d$ are understood as
  column vectors.
Given $x, y \in \R^d$ {the Euclidean scalar product is}
$ \left\langle x, y \right\rangle = x^T y = \sum_{i=1}^d x_i y_i$
(here $\cdot^T$
denotes the
usual
transpose)
with associated norm
$\norm{x} = \sqrt{\left\langle x, x \right\rangle}$.  With
$\mathrm{Tr}(\cdot)$ the trace operator, the {\it{Hilbert--Schmidt
    scalar product}} between matrices ${A}, {B}$ of compatible
dimensions is
$\left\langle {A}, {B} \right\rangle_{\mathrm{HS}} =
\mathrm{Tr}({A}{{^T}} {B})$, with associated norm
$\norm{{A}}_{\mathrm{HS}}^2 := {\mbox{Tr}
  ({A}^T{A})}
$.
For
$W \in \R^{d \times d}$ a matrix,
$\opnorm{W} = \sum_{v \in \R^d, \norm{v} =1} \norm{ W v}$.
{The identity function  on $\R^{d}$ is denoted by $\mathrm{Id}$, so that
$ \Id(x) = x$; for a unit vector $e$,
$ \Id_e = \left\langle \Id, e \right\rangle$
denotes {the marginal projection} in direction $e$; {we abbreviate}
$ \Id_{e_i}{(x)} = \Id_i(x) (= x_i)$.
{T}he identity on function spaces is denote $\mathbf{I}$; the
$d \times d$ identity matrix is denoted by $\mathrm{I}_d$. {Thus, for}
a scalar-valued function $ f $:
\begin{list}{\textbullet}{\topsep 0ex \parsep 0ex \itemsep 0ex}
\item $ \Id f $ is a vector-valued function mapping $ x $ to $ f(x) \, x $;
\item $ \mathbf I f $ is a scalar-valued function mapping $ x $
to $ f(x) $;
\item $ \mathrm I_d f $ is a matrix-valued function mapping $ x $ to
$ f(x) \, \mathrm I_d $.
\end{list}
}

{Let $S^{d-1}$ denote the unit sphere in $\R^d$ and let $e \in S^{d-1}$
be a unit vector in $\R^d$. The directional derivative of a function
$v\Colon\R^d \to\R$ in the direction $e$ is denoted {by the real-valued} function $\partial_e v$.
For $ i=1, \ldots, d$ we write
$\partial_i v$ for the derivative in the direction of the unit vector
$e_i$.}  Higher order derivatives are denoted accordingly.  The
directional derivative of a matrix-valued function}
$ \mathbf{F}\Colon \R^d \rightarrow \R^{m+ {r}} \Colon {x}
\mapsto \mathbf{F}(x)
= \big(f_{ij}(x))_{1\leq i \leq m, 1 \leq j \leq r} $
  is defined component-wise:
$(\partial_e \mathbf{F})_{1\leq i \leq m, 1 \leq j \leq {r}} =
(\partial_e {f}_{ij})_{1\leq i \leq m, 1 \leq j \leq {r}}$ (an
$m\times r$ matrix).

The {gradient} of a smooth function {$v\Colon \R^d \to \R$ is the {$d \times 1$ column} vector
  valued function $\grad v$} with entries
$ (\grad v)_{{i}} = \partial_{i} v$, ${i}=1, \ldots, d$, and the Hessian is the {symmetric $d \times d$}
  matrix-valued function  $\grad^2 v$ with entries
  $(\grad^2v)_{i,j}  = \partial_{ij}v$, $i, j = 1,\ldots, d$;  
\begin{align*}
 \grad v =
  \begin{pmatrix}
    \partial_1 v \\
    \vdots \\
    \partial_d v
  \end{pmatrix}
  \mbox{ and
  }
  \grad^2v =
  \begin{pmatrix}
    \partial_{11} v & \partial_{12} v & \cdots & \partial_{1d}v \\
    \partial_{21} v & \partial_{22} v & \cdots & \partial_{2d}v \\
    \vdots & \vdots & \ddots & \vdots \\
        \partial_{d1} v & \partial_{d2} v & \cdots & \partial_{dd}v
  \end{pmatrix} .
\end{align*}
%

The {divergence} of a
$d$-vector-valued function $\mathbf{v}\Colon \R^d \rightarrow {\R^d}$
with components $v_j, j = 1, \ldots, d$ is the scalar valued function
$$\dv \mathbf{v}\Colon \R^d \rightarrow \R :   x \mapsto
\dv \mathbf{v}(x)
= \sum_{i=1}^d \partial_i v_i(x).$$ 
The divergence of a $m\times d$ matrix-valued function
$\mathbf{F} =  \big(f_{ij}(x))_{1\leq i \leq m, 1 \leq j \leq d}$ 
is
\begin{equation*}
   \dv \mathbf{F} =
  \begin{pmatrix}
    \sum_{j=1}^d \partial_j f_{1j} \\
    \sum_{j=1}^d \partial_j f_{2j} \\
    \vdots \\
        \sum_{j=1}^d \partial_j f_{mj} \\
  \end{pmatrix}
\end{equation*}
a
$m$-column vector; the divergence acts on the rows.
{For}
  $v\Colon \R^d \rightarrow \R$ we also use the Laplacian
  $$\Delta v = \dv (\grad v) = \sum_{i=1}^d \partial_i^2 v.$$ 

Given $v$ and $w$ two sufficiently smooth functions from
$\R^d \to \R$ and $e \in S^{d-1}$,  the directional derivative satisfies the \emph{product
  rule}
\begin{equation}
  \label{eq:46}
 \partial_e (v w) = (\partial_e v) w + v \partial_e w,
\end{equation}
at all points in
$\R^d$ at which all derivatives are defined, implying product rules
for gradients and divergences.
 {We shall mainly consider three instances.}
\begin{enumerate}
\item For
$f \Colon \R^d \to \R$  and $ {g}\Colon \R^d \rightarrow \R$
\begin{align} \label{P1}
  \grad (fg) = { (\grad  f)} {{{g}}} + f (\grad {g}).
\end{align}
\item {For {a matrix-valued function}  $\mathbf{F}\Colon \R^d \rightarrow
    \R^{m \times d}
$ and $g\Colon \R^d \rightarrow \R$,
we have
\begin{equation}\label{eq:39plus}
  \dv( \mathbf{F}g) = (\dv \mathbf{F}) \, g +  \mathbf{F} \grad g
\end{equation}
(an $m$-vector). } {In particular if $ \mathbf F = \mathrm I_d $, then
 $ \dv (\mathrm I_d g) = \grad g $.}
\item
For $\mathbf{F}\Colon \R^d \rightarrow \R^{d \times d}
$ and  $ {g}\Colon \R^d \rightarrow \R$
we have
\begin{align} \label{eq:star}
  \dv ( \mathbf{F}^T \grad g ) =   \left\langle
\dv \mathbf{F}, \grad g\right\rangle + \left\langle
  \mathbf{F}, \grad^2 g \right\rangle_{\mathrm{HS}}
\end{align}
(a scalar).
\end{enumerate}

Throughout this paper, all random vectors are assumed to live on the
same probability space and, unless {explicitly mentioned otherwise},
{are} 
distributed according to a measure $P$ which is (i) absolutely continuous with respect to the Lebesgue measure on $\R^d$, and (ii) with pdf
 $p$ whose support $\Omega_p {= \left\{ x \sth p(x)>0\right\}}$ is an
open subset of $\R^d$; {{unless otherwise stated the pdf is assumed to be continuous on $\Omega_p$}} .}

{T}he collection of all functions whose
components are integrable with respect to $p$ is denoted by $L^1(p)$. {
{For ease of notation, the Lebesgue measure is often left out in
integrals; thus, for} $f \in L^1 (p)$,  we use the notations  $\E_p f = \int_{{\Omega}_p} f p \, = \int_{{\Omega}_p} f(x) \, p(x) \,
\dl x$.}     }
We shall consider several function spaces 
on open sets $ \Omega \subseteq \R^d $.
{
\begin{itemize}
\item {$\Lip_{\loc} (\Omega)$ denotes the set of all locally Lipschitz functions $g \Colon \Omega \to \R$.}
\item $ \Cont^k(\Omega) $ (with possibly $ k = \infty $)
denotes the set of real-valued functions on $ \Omega $
with $k$ continuous derivatives.
\item $ \Cont_\comp^k(\Omega) $ (with possibly $ k = \infty $)
denotes the set of functions belonging to $ \Cont^k(\Omega) $
with a compact support $ K \subseteq \Omega $.
\item $ W^{1, 1}(\Omega)$ denotes the Sobolev space of functions
$g \Colon \Omega \to \R$ such that the Sobolev norm
\[
 \norm{g}_{1,1}
 =
 \int_\Omega |g| + \sum_{i=1}^d \int_\Omega |\partial_i g|
\]
is finite. Here, $ \partial_i $ denotes the \emph{weak} derivative.
{Following \cite{Leo}, p.320, we denote by $ \partial_i $ also the
usual derivative,  the interpretation of $ \partial_i $ depending on the
context. For Sobolev spaces it is understood as the weak derivative.}
\item $ \dot W^{1, 1}(\Omega)$ denotes the homogeneous Sobolev space of
functions $ g \in L^1_\loc(\Omega) $ with
$
 \sum_{i=1}^d \int_\Omega |\partial_i g| < \infty \, .
$
\item $ W^{1, 1}_0(\Omega)$ denotes the space of all functions
$ g \in W^{1, 1}(\Omega) $, such that $ \int_{\Omega} \partial_i g = 0 $
for all $ i = 1, 2, \ldots, d $ or, equivalently, $ \int_{\Omega} \partial_e g = 0 $
for all $ e \in S^{d-1} $. 
\item $ \dot W^{1, 1}_0(\Omega)$ denotes the space
of all functions $ g \in \dot W^{1, 1}(\Omega) $, such that
$ \int_{\Omega} \partial_e g = 0 $ for all $ e \in S^{d-1} $.
\item $ \ACL(\Omega) $ denotes the set of all
Borel measurable functions $ g \Colon \Omega \to \R $, such that for each
$ e \in S^{d-1} $, it is true that
for almost all lines $ L $ parallel to $ e $ with respect
to the $ (d - 1) $-dimensional Lebesgue measure, the restriction of $ g $
to any compact interval $ I $ contained in $ \Omega \cap L $ is absolutely
continuous. Notice that $ \ACL(\Omega) $ is closed under multiplication.
\item $ \ACL^1(\Omega) $ denotes the set of functions
$ g \in \ACL(\Omega) $
with $ \partial_e g \in L^1(\Omega) $ for all $ e \in S^{d-1} $.
\item $ \ACL^1_\loc(\Omega) $ denotes the set of all functions
$ g \Colon \Omega \to \R $ {with the property that}
each $ x \in \Omega $ has an open neighbourhood
$ U \subseteq \Omega $,
such that the restriction of $ g $ to $ U $ belongs to $ \ACL^1(U) $.
\item $ \ACL^1_0(\Omega) $ denotes the set of functions
$ g \in \ACL^1(\Omega) $ with $ \int_{\Omega} \partial_e g = 0 $
for all $ e \in S^{d-1} $.
\end{itemize}
Here 
$\ACL$ stands for `absolutely continuous on lines'.
This concept is closely related with the concept of
\emph{almost differentiability} used by Stein \cite[Definition 1]{S81}.
Indeed, each function in $ \ACL(\R^d) $ is almost differentiable,
but the converse is not true: take a straight line in $ \R^2 $ and let {$g \Colon \R^2 \rightarrow \R$ be defined as}
$ g \equiv 1 $ on that line and $ g \equiv 0 $ elsewhere. Then
$ g $ is almost differentiable, bgt {it} is not in $ \ACL(\R^2) $.}

\section{Stein's multivariate density method}
\label{sec:first-defin-outl}

\subsection{The canonical operator}
\label{sec:general-theory}
{All Stein operators to be constructed in this paper have as building
  blocks the \emph{canonical} {Stein} operator and {the \emph{canonical} Stein} class which we now define.  }

\begin{definition}[Canonical directional Stein
operator]\label{def:steclass} 
{Let
  $e \in S^{d-1}$.}  The \emph{canonical Stein derivative for $p$ in
  the direction $e$} is the differential operator
$\phi \mapsto \mathcal T_{e,p} \phi := {\partial_e( p \phi)}/{p} $
where  $\mathcal{T}_{e, p} \phi \equiv 0$ outside of
$\Omega_p$. The domain of $\mathcal{T}_{e, p}$ is
$\dom(\mathcal{T}_{e, p})$ the collection of all functions
that are differentiable in direction $e$.
  \end{definition}


\begin{definition}[Canonical Stein class] 
  The \emph{canonical scalar Stein class for $p$} is the collection
  $\mathcal{F}_1(p)$ of all functions $f \Colon \Omega_p  \to \R$ {with $fp \in \ACL^1_0(\Omega_p)$}.
  The \emph{canonical Stein
    class for $p$} is the collection $\mathcal{F}(p)$ of all scalar,
  vector, and matrix-valued functions whose components belong to
  $\mathcal{F}_1(p)$.

\end{definition}

{\noindent
To illustrate the Stein class we first give the next result.
{%
\begin{prop}
\label{prop:ACec0}
Let  {$u \in \ACL(\Omega_p)$}
If, in addition, $ \partial_e u \in L^1_\loc(\Omega) $
and $ u $ is compactly supported, then $ \partial_e u \in L^1(\Omega) $ and
$ \int_\Omega \partial_e u = 0 $. {In particular, any scalar function $\phi\in \Cont_{{\comp}}^1(\Omega_p)$
      lies in $\mathcal F_{1}(p)$ {as long as the pdf $p$ is
        continuously differentiable}.}
\end{prop}

\begin{proof}
Since $ u $ has a compact support, so does $ \partial_e u $. Therefore, $ \partial_e u \in L^1(\Omega) $.
Without loss of generality, we may assume that $ e = e_1 $.  For Lebesgue-almost all
$ x' \in \R^{d-1} $, 
the function $ u_{x^\prime}(x_1) := u(x_1, x') $ is
absolutely continuous on all compact intervals contained in $ \Omega_{x^\prime} :=
\{ x_1 \in \R \sth (x_1, x') \in \Omega \} $ and 
$ u'_{x^\prime} \in L^1(\Omega_{x^\prime}) $. Now observe
that $ \Omega_{x^\prime} $ is a union of countably many open intervals. If $ I $ is such an interval,
the restriction of $ u_{x^\prime} $ to $ I $ has a compact support. By the fundamental theorem of
calculus, we have $ \int_I u'_{x^\prime} = 0 $. Since $ u'_{x^\prime} \in L^1(\Omega_{x^\prime}) $,
we can sum over the intervals to obtain $ \int_I u'_{x^\prime} = 0 $. The {main assertion now follows} by
Fubini's theorem. {The last assertion is immediate by taking $u  = p \phi$.}
\end{proof}
}
}
\begin{rem}\label{rem:rem33}
\leavevmode
\begin{enumerate}
\item {The directional Stein operator requires the pdf $p$ to be available only up to a normalizing constant.}
      \item {For a Borel function $ f \Colon \Omega_p \to \R $,
a version of $ f $ belongs to $ \mathcal F_1(p) \cap L^1(p) $ if and only if
$ f p \in W^{1,1}_0(\Omega_p) $}.
    \end{enumerate}
\end{rem}

  \begin{example}[Gaussian directional Stein operators] \label{ex:gauss1}
    {For}  the multivariate Gaussian distribution
    with location $\nu\in \R^d$,  { positive definite covariance matrix}
    $\Sigma\in \R^d\times \R^d$ and
    pdf {$\gamma$}, with $e_i$ a unit  vector,
$$ \mathcal{T}_{e_i,\gamma}f(x) = \frac{\partial_{e_i} (f(x)
  \gamma(x))}{\gamma(x)} = \partial_i f(x) - (\Sigma^{-1} (x-\nu))_i f
(x)$$ for $i = 1, \ldots, d$ and {the scalar Stein class is the class
  of all almost differentiable functions $f \Colon \R^d \to \R$ such that
  $\E_{\gamma}[\norm{\grad f}] < \infty$, see \cite[Lemma
  2]{S81}. Indeed the integrability condition
  $ \int \partial_i (f \gamma) =0 $ is automatically satisfied. }
\end{example}

\begin{example}[Student-$t$ operators]
    \label{ex:studentfirstexample}
{For}  the multivariate Student-$t$ distribution with $k >1$
degrees of freedom, location $\nu\in \R^d$, shape
$\Sigma\in \R^d\times \R^d$ and pdf
\begin{align*}
  t_k(x) = c_{k, d } \mathrm{det}(\Sigma)^{-1/2} \left[
    1+\frac{(x-\nu)^T\Sigma^{-1}(x-\nu)}{k}\right]^{-(k+d)/2}
\end{align*}
with normalizing constant
$c_{k, d} $
and
full support $\R^d$, the directional derivatives are
\begin{align*}
  \mathcal{T}_{e_i, t_k}
  f(x) = \partial_{e_i}f(x)- \frac{k+d}{2}  \left(
  1+\frac{(x-\nu)^T\Sigma^{-1}(x-\nu)}{k}\right)^{-1}(\Sigma^{-1}(x-\nu))_if(x)
\end{align*}
for $i = 1, \ldots, d$. {Using a similar argument as in Example
  \ref{ex:gauss1}, the scalar Stein class contains all functions
  $f \Colon \R^d \to \R$ such that $\E_{t_k}[\norm{\grad f}] < \infty$.
}
\end{example}

We stress the {difference between the}
{domain of the}
canonical directional operator, which simply consists of
differentiable functions, and its Stein class $\mathcal{F}(p)$ which
contains all functions $\mathbf{F}$ such that $\mathcal{T}_{e, p}$ is
integrable with respect to $p$ in all directions and, moreover,
satisfy {that for all $e \in S^{d-1}$,}
\begin{equation}
  \label{eq:6}
  \E_p[\mathcal{T}_{e, p} \mathbf{F}] = 0.
\end{equation}
 Identity \eqref{eq:6} is {our}
  \emph{canonical Stein
  identity}  {in the spirit of \eqref{eq:stlemma}}. It is the identity from which all other Stein identities in this paper will follow.
{In particular,}
  {given any differentiable $\mathbf{F}, g$,
  the product rule \eqref{eq:46}
  {yields}
  the Stein-type
  product rule:
\begin{equation}
  \label{eq:24}
{  \mathcal{T}_{e, p}(\mathbf{F}\, g) =  ( \mathcal{T}_{e, p}
  \mathbf{F}) g    +
\mathbf{F} \, \partial_e g}.
\end{equation}
Plugging \eqref{eq:24}, for well-chosen functions $\mathbf{F}$ {or}
$g$ {(or both)}, into \eqref{eq:6},   then leads to
Stein-type integration by parts identities which are the basis of
{many}
of the forthcoming {results}.
{W}e first 
introduce the appropriate sets of functions.}
{\begin{definition}[Stein adjoint class]\label{def:stadjclass}
Let $ \mathcal F_{1,\loc}(p) $ be the class of all functions
$ f \Colon \Omega_p \to \R $ with
$ f p \in \ACL(\Omega_p) $,
{and denote by $ \mathcal F_\loc(p) $}
the class of all scalar-, vector- and matrix-valued functions $ \mathbf F $
with all components belonging to $ \mathcal F_{1,\loc}(p) $.  To every (scalar, vector, {or} matrix-valued) function
  $\mathbf{F} \in \mathcal F_\loc(p)$ we
  {denote by}
  $ \dom(p, \mathbf{F})$ the collection of all functions
{$ g \in \ACL(\Omega_p) $}
  which
  satisfy
  $ (\mathbf{F}\, g) \in \mathcal{F}(p)$ and $\mathbf{F}
  (\partial_e g) \in L^1(p)$ for all $e\in S^{d-1}$.
\end{definition}
}

 {By definition of $\mathcal{F}_\loc(p)$,  for  all
    $g {\Colon \R^d \to \R} \in \dom(p, \mathbf{F})$, the left
  hand side in \eqref{eq:24} integrates to 0 {under $p$} and the
  integrals of the summands on the right-hand-side (rhs) may be taken
  separately so that
\begin{equation}
  \label{eq:10}
{  \E_p \left[  (\mathcal{T}_{e, p} {\mathbf{F}}) g  +
  {\mathbf{F}} \, \partial_e {g}   \right] = 0}
\end{equation}
for all $e\in S^{d-1}$, generalizing \eqref{eq:6} to a Stein identity
{which is} valid for all $g \in \dom(p, \mathbf{F})$}.
 {Our approach to the Stein operator machinery for $p$ is  to fix some $\mathbf F$ and study the operator (in $g$) that can be obtained from \eqref{eq:10}, viewed through its action on the class  $\dom(p, \mathbf{F})$.  We call this process a \emph{standardization} of the canonical operator, and will dive more deeply into this in the next section.  The quality of the operators obtained in this manner depend on the choice of $\mathbf F$.  If
$\mathbf{F}$ has all components in $\mathcal{F}(p)$ then $\dom(p, \mathbf{F})$}
contains at least the constant
functions $g \equiv 1$.
{
{E}ven if the components of $\mathbf F$ do not belong to $ \mathcal F(p) $, $\dom(p, \bullet)$ remains actually quite large, as we now show.  }
\begin{lemma}\label{lm:measurdeterm}
  {Suppose that all components $ f $ of $ \mathbf{F} $ are such that
  $ f p \in \ACL^1_\loc(\Omega_p) $. Then any 
  function
  $ g \in \Lip_{\loc} (\Omega)$ such that $ \mathbf{F} g p $ is compactly supported belongs to
  $ \dom(p, \mathbf{F}) $.}
\end{lemma}

\begin{proof}
  {By Proposition~\ref{prop:DL1loc}, we have $ f p \in L^1_\loc(\Omega_p) $.}
  {Since $ g \in \Lip_{\loc} (\Omega)$,
  it belongs to $ \ACL(\Omega_p) $.
  Moreover,  {as} $ g $ and $ \partial_e g $ 
  both bounded for any
  $ e \in S^{d-1} $,}
  {$ f p (\partial_e g) \in L^1_\loc(\Omega_p).$
  Since $ f p (\partial_e g) $ {has compact support},
  it 
  belongs to  $ L^1(\Omega_p) $.
  From the product rule for the classical derivatives and
  {local}
  boundedness of $ g $ and
  $ \partial_e g $, it follows that $ f g p \in \ACL^1_\loc(\Omega_p) $.
  {As} $ f g p $ {has compact support,}
  Proposition~\ref{prop:ACec0}
  {yields} that $ f g p \in \ACL^1_0(\Omega_p) $.}
\end{proof}
 An in-depth discussion    {of $\dom(p, \mathbf{F})$} in the
  one-dimensional case, along with simple sufficient conditions, can
  be found in \cite[Section 2.3]{ernst2020first}.

\subsection{Standardizations}
\label{sec:stein-identities}

The canonical operators are our building blocks for a large family of
  Stein operators, as follows.

\begin{definition}
  \label{def:standardizations-1}
A differential operator
$\mathbf{G} \mapsto \mathcal{A} \mathbf{G}$ acting on {scalar-, or vector-, or matrix-valued} valued functions  is a
\emph{standardization} of the canonical operators
$\mathcal{T}_{e_i, p}, i = 1, \ldots, d$ if there exist linear
operators $ \{ \mathbf{A}_1, \ldots, \mathbf{A}_d \} $ and
$\{ \mathbf{T}_1, \ldots, \mathbf{T}_d\} $ such that
  \begin{align*}
\AAA {\bf G}
  =
    \sum_{i=1}^d{\mathbf{A}}_i \mathcal{T}_{e_i, p} ({\mathbf{T}}_i
    {\mathbf{G}} ) .
      \end{align*}
The associated Stein class is the collection
$\mathcal{F}(p, \mathcal{A}) $ of all functions $\mathbf{G}$ of
appropriate dimension  such that
$\AAA {\bf G}$ is integrable with mean 0 under $p$.
\end{definition}

{
\begin{rem}
We have kept the definition vague in our assumptions on the    operators $ \mathbf A_1, \ldots, \mathbf A_d $ and
$ \mathbf T_1, \ldots, \mathbf T_d $ which are allowed to be any linear operators
on the corresponding function spaces. In particular, they are allowed to be
multiplications by fixed functions, shift or differential operators which  typically shrinks the associated Stein class
$ \mathcal F(p, \mathcal A) $ to a very small set: if $ \E_p h = 0 $,
we typically do not have $ \E_p \mathbf A_i h = 0 $.
In practice, however, the operators
$ \mathbf A_1, \ldots, \mathbf A_d $ we consider are in fact left
compositions by linear maps acting between scalars, vectors or matrices.
This preserves the size of the associated Stein class $ \mathcal F(p, \mathcal A) $.  In particular if $\mathbf{A}_i = {\mathbf{T}}_i = \mathbf{I}$ (the identity functional operator) for
$i=1, \ldots, d$, then the canonical Stein class $\mathcal{F} (p)$ is
a subset of $\mathcal{F}(p, \mathcal{A}) $.
\end{rem}
}

\medskip
In the sequel, we will say that a linear operator is \emph{a} Stein
operator for $p$ if it is a standardization of the canonical operators
obtained through Definition \ref{def:standardizations-1}. Such
standardizations were studied in \cite{LRS16} in dimension $d=1$.
{Among the many possible options, two  stand out most naturally:
\begin{itemize}
    \item $\mathbf{A}_i =  e_i$  and $\mathbf{T}_i = 1$ for $i = 1,\ldots, d$ leading to what we call the gradient operator,
    \item $\mathbf{A}_i = \mathbf{T}_i = 1$ for $i = 1,\ldots, d$ leading to what we call the divergence operator.
\end{itemize}
}

  \begin{definition}[Gradient and divergence operator]\label{def:gradiv}
    {The \emph{Stein gradient operator} acting on real-valued
    functions $g \Colon \R^d \to \R$ is
   $$\mathcal{T}_{\grad, p}  g  {= \sum_{i=1}^d e_i\mathcal{T}_{e_i, p}  g}  = \frac{\grad \left( p g
     \right)}{p}.$$
   The \emph{Stein divergence operator} acting on vector-valued functions $\mathbf{g} \Colon \R^d \to \R^d$ is
  $$\mathcal{T}_{\dv , p}  \mathbf{g} = { \sum_{i=1}^d \mathcal{T}_{e_i, p}  \mathbf{g}_i = }
  \frac{\dv \left( p \mathbf{g} \right)}{p}.$$
The Stein divergence operator {acting} 
on matrix-valued functions
  $\mathbf{G}\Colon \R^d \rightarrow \R^{m \times d}$  {is}
  $$(\mathcal{T}_{\dv , p}  \mathbf{G})_i = { \sum_{j=1}^d \mathcal{T}_{e_j, p}  \mathbf{G}_{i} = }
  \frac{\dv \left( p \mathbf{G}_i \right)}{p} \quad \quad  {\mbox{
for } i = 1, \ldots, m.}$$ }
  \end{definition}

  \begin{example}[Gaussian operators]
    \label{ex:Gaussfirstexample} Take $\nu \in \R^d$ and $\Sigma$ an invertible matrix in $\R^{d\times d}$. For $\mathcal{N}(\nu, \Sigma)$ with pdf $\gamma$,
    \begin{equation} \label{eq:tgrad1gau}
      \mathcal{T}_{\grad, \gamma}g(x) =   \grad g(x)  - \Sigma^{-1} (x-\nu) g(x)
    \end{equation}
    for all $g \Colon \R^d \to \R$  ($\mathcal{T}_{\grad, \gamma}g(x)$ is    a vector)   and
    \begin{equation}\label{eq:tdivgau}
      \mathcal{T}_{\dv , p}\mathbf{g}(x) =
\dv \mathbf{g}(x) -  \left\langle \Sigma^{-1} (x-\nu),  \mathbf{g}(x)\right\rangle
\end{equation}
for all $\mathbf{g} \Colon \R^d \to \R^d$ ($\mathcal{T}_{\dv , \gamma}\mathbf{g}(x)$
is a scalar).
  \end{example}

{From \eqref{eq:10} we thus}
inherit an entire collection of Stein identities
{for functions defined on an open set $ \Omega $ and with
components belonging to $ \ACL(\Omega) $.}
In the sequel we
will make use of the following {instances}:
\begin{enumerate}
  \item For all scalar functions $f, g \Colon \Omega \rightarrow \R$ we
  have {from \eqref{P1}}
  \begin{equation}
  \label{eq:prodrulenabla}
    \mathcal{T}_{\grad, p}(f g)
    =
    (\mathcal{T}_{\grad, p} f) \, g + f \grad  g
    \, .
  \end{equation}
  For each
  {$ f \in \mathcal F_{1,\loc}(p) $}
  we obtain the Stein identity
  \begin{equation}
  \label{eq:8}
    \E_p \bigl[ (\mathcal{T}_{\grad, p} f) \, g \bigr]
    =
    - \E_p[f \grad g]
    \, ,
  \end{equation}
  which holds for all $g \in \dom(p, f)$.
  \item For all matrix-valued functions
  $\mathbf{F}\Colon \Omega \rightarrow \R^{m \times d}$ and
  {all ${g}\Colon\R^d\to {\R}$,}
  {using \eqref{eq:39plus},}
  \begin{equation}
  \label{eq:42}
    \mathcal{T}_{\dv, p} (\mathbf{F} \, g)
    =
    (\mathcal{T}_{\dv, p} \mathbf{F}) \, g
      +
    \mathbf{F} \grad g
    \, .
  \end{equation}
  For each
  {$ \mathbf{F} \in \mathcal F_{\loc}(p) $}
  we obtain the Stein identity
  \begin{equation}
  \label{eq:37F}
    \E_p \bigl[ (\mathcal{T}_{\dv, p} \mathbf{F}) \, g \bigr]
    =
    - \E_p[ \mathbf{F} \grad g]
    \, ,
  \end{equation}
  which holds for all ${g} \in \dom(p, \mathbf{F})$.
  \item {For all vector-valued function
  $ \mathbf f \Colon \Omega \to \R^d$ and all $ g \Colon \R^d \to \R $, using \eqref{eq:42},
  \begin{equation}
  \label{eq:39plllllus}
    \mathcal T_{\dv, p} (\mathbf f g)
    =
    (\mathcal T_{\dv, p} \mathbf f) \, g
      +
    \langle \mathbf f, \grad g \rangle
    \, .
  \end{equation}
  For each
  {$ \mathbf{f} \in \mathcal F_{\loc}(p) $}
  we obtain the Stein identity
  \begin{equation}
  \label{eq:37f}
    \E_p \bigl[ (\mathcal{T}_{\dv , p} \mathbf{f}) \, g \bigr]
    =
    - \E_p \bigl[ \langle \mathbf{f}, \grad g \rangle \bigr]
    \, .
  \end{equation}
  which holds for all ${g} \in \dom(p, \mathbf{f})$.}
  \item For all matrix-valued functions
  $\mathbf{F}\Colon \Omega \rightarrow \R^{d \times d}$ and all
  $ {g} \Colon \R^d \to {\R} $, {using \eqref{eq:star}},
  \begin{equation}
  \label{eq:9}
    \mathcal{T}_{\dv, p} \bigl( \mathbf{F}^T \grad g \bigr)
    =
    \left\langle
      \mathcal{T}_{\dv, p} \mathbf{F}, \grad g
    \right\rangle
      +
    \left\langle \mathbf{F}, \grad^2 g \right\rangle_{\mathrm{HS}}.
  \end{equation}
  For each
  {$ \mathbf{F} \in \mathcal F_{\loc}(p) $}
  we obtain the Stein identity
  \begin{equation}
  \label{eq:25bis}
    \E_p \bigl[ \left\langle
      \mathcal{T}_{\dv , p} \mathbf{F}, \grad g
    \right\rangle \bigr]
    =
    - \E_p \bigl[ \left\langle
      \mathbf{F}, \grad^2 g
    \right\rangle_{\mathrm{HS}} \bigr]
  \end{equation}
  which holds for all $g$
  {with}
  $\grad g \in \dom(p, \mathbf{F})$.
\end{enumerate}

\begin{example}[Gaussian Stein {identities}] Following up on Examples {\ref{ex:gauss1} and}
  \ref{ex:Gaussfirstexample},
 take $\nu \in \R^d$ and an invertible matrix $\Sigma \in \R^{d\times d}$  and {let}  $\gamma$  denote the pdf of $\mathcal{N}(\nu, \Sigma)$.  Taking {the constant function} $f = 1$ in
    \eqref{eq:8}, {which is in ${\mathcal F}_1(\gamma)$,}  we reap
    \begin{equation}\label{eq:gaussianstandard}
      \E_{\gamma} [ \Sigma^{-1} ({\rm Id}-\nu) g   ] =
      \E_{\gamma}  [\grad g ]
    \end{equation}
    for all $g \Colon \R^d \to \R$
    belonging to $\dom(\gamma, 1)$.  {We could have also obtained this directly from  operator \eqref{eq:tgrad1gau}.}
    Similarly, now taking {the constant function}
    $\mathbf{F} = \Sigma$ in \eqref{eq:25bis}, which is also in ${\mathcal F}(\gamma)$,
   we
   {recover}
     the
    classical second order Stein identity {\eqref{eq:stlemma}} for the Gaussian. {This could have been obtained directly from operator \eqref{eq:tdivgau} applied to functions $\mathbf{g} = \Sigma \grad g$.}
     Example
    \ref{sec:stein-oper-Gauss} {provides other choices of $\mathbf{F}$}.
  \end{example}

  \subsection{Stein characterizations}\label{sec:stein-char}

  {Identity \eqref{eq:10} (and the corresponding identities from
    Section \ref{sec:stein-identities}) holds when integrating under
    $p$. Our purpose now is to deduce ``reverse'' implications when
    changing the measure $p$ to some other measure $q$.} The starting
  point is the following ``directional'' observation.

\begin{prop}[Stein characterizations]\label{sec:diff-stein-oper-directional}
  Let $p$
  and $q$ be pdf{s}
  {with}
  {$\Omega_q \subseteq \Omega_p$. Suppose that}
  {$\Omega_p  $ is connected and}
  {$ q/p  \in \Lip_\loc {(\Omega_p)}$.}
  {Fix $ f \in \mathcal F_{1,\loc}(p) $ with}
  {$f \neq 0 $ over $\Omega_p$}. Then $p=q$ if and only if
  $ \E_q [g \mathcal{T}_{e, p}f] = - \E_q [f \partial_eg] $
  for all $ e\in S^{d-1} $ and all $ g \in \dom(p, f) $
  {with $ g (\mathcal T_{e,p} f), f (\partial_e g) \in L^1(q)$.}
\end{prop}

\begin{proof}
{If $p=q$, then the identity follows from \eqref{eq:10}.  For the
  {opposite}
  direction,
  {assume that} $
    \E_q [g \mathcal{T}_{e, p} f]
    =
    - \E_q [f \, \partial_e g]
  $. Then
  with \eqref{eq:24}, we have
  \[
   0
   =
   \E_q \bigl[ \mathcal{T}_{e, p} (fg) \bigr]
   =
   \E_p \left[
     \bigl( \mathcal{T}_{e, p} (fg) \bigr) \, \frac{q}{p}
   \right]
   =
   - \E_p \left[
     f g \, \partial_e \! \left( \frac{q}{p} \right)
   \right] ,
  \]
  where the last identity follows by \eqref{eq:10} applied with
  {$ q/p $ in place of $ g $}
  and $ \mathbf{F} = f g $,}
  {provided that $ f g \in \mathcal F_{1,\loc}(p) $ and
  $ q/p \in \dom(p, f g) $. In particular, by Lemma \ref{lm:measurdeterm},
  the latter condition is satisfied by all
  $g \in \Cont_\comp^\infty(\Omega_p)$.
  Therefore, $ f p \, \partial_e (q/p) $ vanishes Lebesgue-almost
  everywhere over
  $ \Omega_p $. The same is true for $ \partial_e (q/p) $ because
  $ f p $ vanishes nowhere over $ \Omega_p $.
  Since this is true for all $e \in S^{p-1}$ and $ \Omega_p $ is
  connected, the ratio $ q/p $ is Lebesgue-almost everywhere
  constant on $ \Omega_p $ by Proposition \ref{prop:D0Connected}.
  However, as $ q/p $ is continuous, it must be constant on the entire
  $ \Omega_p $.}
  {In particular, $ \Omega_p = \Omega_q $ automatically.}
  Since both {$p$ and $q$} integrate to 1 on their support,
  $p=q$ follows.
\end{proof}

{
\begin{rem}
  Proposition \ref{prop:D0Connected}, which was applied in the proof,
  can be regarded as {the extreme case of the}
  \emph{Poincar\'e inequality}
  -- see Definition \ref{def:poincareconstant} -- {in which the inequality
  {turns to}
  an equality}. Similarly, Stein
  characterizations can be regarded as {the extreme case of inequalities for Stein disrepancies in which the Stein discrepancy takes on the value 0.} 
  In Section
  \ref{sec:stein-factors-under}, we derive bounds which are
  based precisely on the Poincar\'e inequality.
\end{rem}
}

{The directional characterizations from Proposition
  \ref{sec:diff-stein-oper-directional} lead} to a wide variety of
characterizations, {among which we} highlight the following
 two.


\begin{prop}[Some Stein characterizations]
  {Let $p$ and $q$ be as in
  Proposition \ref{sec:diff-stein-oper-directional}.}
  \begin{enumerate}
    \item \label{sec:diff-stein-oper}
    Let
    $f\Colon \R^d \rightarrow \R \in  {\mathcal{F}_{1,\loc}(p)}$.
    {Then} $p=q$ if and only if
    \begin{equation*}
     \E_q \bigl[ g \mathcal{T}_{\grad, p} f \bigr]
     = - \E_q \bigl[ f \grad g \bigr]
    \end{equation*}
    for all
    $g\Colon \R^d \rightarrow \R \in \dom(p,f)$.
    \item Let
    $\mathbf{F}\Colon \R^d \rightarrow \R^{m \times d} \in
     {\mathcal{F}_{\loc}(p)}$
    {be such that the matrix $ \mathbf F(x) $ has zero nullity for Lebesgue-almost all $ x \in \Omega_p $.}
    {Then} $p=q$ if and only if
    \begin{equation*}
     \E_q \left[ g \mathcal{T}_{\dv, p}  \mathbf{F} \right]
     =
     - \E_q \bigl[ \mathbf{F} \grad {g}  \bigr]
    \end{equation*}
    for all
    $g \Colon \R^d \rightarrow \R \in \dom(p, {\mathbf{F}})$.
  \end{enumerate}
\end{prop}

\begin{proof}
  {The first
  {part}
  follows directly from Proposition \ref{sec:diff-stein-oper-directional}.}
  {The second one can be proved similarly: if $ p = q $, the desired
  equality follows from \eqref{eq:37F}. For the opposite direction,
  \eqref{eq:42} and \eqref{eq:37F} give
  $ \E_p \bigl[ \mathbf{F} g \grad (q/p) \bigr] = 0 $
  for all $ g \in \Cont_\comp^\infty(\Omega_p) $. Therefore,
  $ \mathbf{F} \grad (q/p) $ vanishes Lebesgue-almost everywhere
  over $ \Omega_p $. Since $ \mathbf F(x) $ has zero nullity
  for Lebesgue-almost all $ x \in \Omega_p $, $ \grad (q/p) $
  also vanishes Lebesgue-almost everywhere. As a result, $ q/p $
  is constant {on} 
  $ \Omega_p $ and the result follows.}
\end{proof}

  \subsection{The score function and the Stein kernel}
\label{sec:standardizations}

We conclude the section with three key examples of Stein operators.

\subsubsection{Gradient based first order operators and the score function}
 \label{sec:grad-based-oper}

 We first focus on operator \eqref{eq:prodrulenabla} and its companion
  Stein identity \eqref{eq:8}.
{We deduce} a family of Stein operators for $p$
  obtained by fixing some {differentiable $f$} and considering the
  first order operator $ \mathcal{A}_p := \mathcal{A}_{f, p}$ given by
  $${g} \mapsto \mathcal{A}_p {g}:=\mathcal{T}_{\grad,
    p}(f {g})$$ with associated Stein class
  ${g}\Colon\R^d\to \R \in \mathcal{F}(\mathcal{A}_p) =
  \dom(p, f)$.  Each particular  a.s. differentiable $f$
  thus
  gives rise to a particular operator, acting on a particular class of
  functions, entailing a particular Stein identity.
One {choice for $f$} stands out: $f = 1$.
\begin{definition}[Score function and operator] \label{def:multiscore}
  Let $p$ be differentiable. The \emph{score function} of $p$ is the
  function
\begin{align*}
  \rho_p = \mathcal{T}_{\grad, p}1 = \grad \log p = \frac{\grad p}{p}
\end{align*}
(still with the convention that $\rho_p \equiv 0$ outside of $\Omega_p$).
The \emph{score-Stein operator} is the vector-valued operator
\begin{equation}
\label{eq:1}
      \mathcal{A}_p = \grad   +  \rho_p \mathbf{I}
    \end{equation}
    acting on differentiable functions $g \Colon \R^d \to \R$.
\end{definition}

Operator \eqref{eq:1} is the most classical Stein operator for
multivariate $p$. It is {particularly useful in the context of
  Stein's method}
when
$\mathcal{F}_{{1}}(p)$ contains all  constant functions. The latter assumption  holds if $p$ is a
differentiable pdf such that $\partial_i p$ is integrable for all
$i = 1, \ldots, d$ and $\int \partial_i p = 0$ (this is not guaranteed
  if we only assume that $p$ is continuously differentiable:
 consider for example $p(x) \propto x^2 \sin (x^{-2})$ on $(0,1)$).  Then we can
take $\mathcal{F}(\mathcal{A}_p) = \dom(p, 1)$;
{the}
 resulting (characterizing) Stein identity
is
 \begin{align*}
\E_p \left[ \rho_p \, g
   \right] = -  \E_p \left[ \grad g \right] \mbox{ for all }g \in
   \mathcal{F}(\mathcal{A}_p).
 \end{align*}
{If $1 \notin \mathcal{F}_1(p)$ then a version of this identity still
   holds, but with integration constants which must be taken into
   account in the various identities, see Theorem \ref{theo:comap}. }

 \begin{example}[Gaussian score operator]\label{ex:scorestinop}
  {For $\mathcal{N}(\nu, \Sigma)$ with pdf $\gamma$, we have}
   $\rho_{\gamma}(x) = \grad \log \gamma(x) = -\Sigma^{-1}
   (x-\nu)$; the corresponding operator is
   \begin{equation*}
     \mathcal{A}_{\gamma}g(x) = \grad g(x) -\Sigma^{-1}
     (x-\nu) g(x).
   \end{equation*}
    \end{example}

    The flexibility of choice of functions $f$ clarifies a
      relationship between Stein operators, as follows.

 \begin{example}[{Gradient based operators and change of measure}]
 \label{sec:grad-based-oper-change}

 Suppose that  {$f$}
is such that $fp$ is a pdf
 which is continuously differentiable
 on its support
 $\Omega_p$. Then, for {any ${g}$  which is {continuously differentiable}
  on
 $\Omega_p$,}
 \begin{align*}
 \mathcal{T}_{\grad, p} \left( f  {g}\right) = \frac{\grad (  {g} f p)}{p}
= f  \frac{\grad (  {g} f p)}{f p}
= f  \mathcal{T}_{\grad, fp} ( {g} ).
\end{align*}
Such a change of measure operation can be useful for finding solutions
of Stein equations.
The   Stein equation
  $$ \mathcal{T}_{\grad, p} \left( f  {g}\right) = h - \E_p h$$
  then translates into
   $$  \mathcal{T}_{\grad, fp} ( {g} )=  1/f ( h  -  \E_p h).$$
For instance, if $p$ is log-concave then, under
additional regularity conditions, solutions of the Stein equation with
bounded derivatives are available (see
\cite{mackey2016multivariate}).
{In Proposition \ref{prop:l2bound} we shall show a similar result under the assumption of a Poincar\'{e} constant.} Such bounds may also be applied even
when $p$ is not log-concave {or does not possess a Poincar\'{e} constant}, but there exists $f$ such that $fp$ is
{has this desired property}  and the change of measure illustrated above can be
applied.  {As an illustration, the
 one-dimensional Beta distribution with pdf
$p(x) = x^{\alpha-1} (1-x)^{\beta-1}/B(\alpha, \beta)$ on $(0,1)$ is
not log-concave if $\min(\alpha, \beta)<1$.  Choosing
$f(x) = B(\alpha, \beta) x(1-x) / B(\alpha+1, \beta+1)$ results in
$fp$ being log-concave.}
{However, as a different Stein equation is being
solved, Stein's method need not bound the error in the same metric. For log-concave densities,
such as $ f p $, one can for instance obtain bounds in the Wasserstein distance, while for $ p $, the bound
is expressed in a different metric.}
{This change of measure leads to $p$ and $fp$ to be nested in the sense of {Section \ref{subsec:gencomp}}, where the focus is on Wasserstein distance.}
 \end{example}

\subsubsection{Divergence based first order operators and Stein kernels}
\label{sec:diverg-based-first}

We now start from the product rule {\eqref{eq:42} and the
  corresponding Stein identity \eqref{eq:37F}, with
  $\mathbf{F}\Colon\R^d\to\R^{m\times d}$ and $g \Colon \R^d\to \R$.}  We
deduce a family of Stein operators for $p$ obtained by fixing
  $\mathbf{F}$ and considering
$$g \mapsto \mathcal{A}_p  g  = \mathcal{T}_{\dv , p} \big(
\mathbf{F} g\big)
= (\mathcal{T}_{\dv , p}
\mathbf{F}) g + \mathbf{F} \grad g $$ with associated Stein class
$\mathcal{F}( \mathcal{A}_p) := \dom(p, \mathbf{F}) $.

Taking {$\mathbf{F} = I_d$ the identity matrix} results in
\eqref{eq:1}.  If $p$ has finite covariance $\Sigma$, a more natural
choice may be $\mathbf{F} \equiv \Sigma$, but the resulting operator
is not intrinsically different from \eqref{eq:1}.
A 
different group of choices for $\mathbf{F}$ stands out: \emph{Stein
  kernels}, which we define as follows.

\begin{definition}[Stein kernel and operator]\label{def:multiteikdef}
Suppose $p$ has finite mean
  $\nu \in \R^d$. For each    {unit vector}
  $e_i \in \R^d$,
  $i=1, \ldots, d$, \emph{a Stein kernel} for $p$ \emph{in
    direction $e_i$} is any vector field
  $x \mapsto \tau_{i}(x) \in \R^d$
such that  {$\tau_{i} p \in \ACL^1_\loc(\Omega_p)$}
and, for {Lebesgue-almost} all $x \in \Omega_p$,
  \begin{equation*}
    \mathcal{T}_{\dv , p}\big(\tau_{i}\big)(x) = \nu_i- x_i.
  \end{equation*}
  A \emph{Stein kernel} is any square {matrix-valued function}
  $ \pmb\tau =(\tau_{i,j})_{1 \leq i,j \leq d}$ such that each {row}
  $ \pmb \tau_i = (\tau_{i1}, \ldots, \tau_{id})$ is a Stein kernel
  for $p$ in direction $e_i$;
  \begin{equation} \label{eq:steinkenrlnelnr} {\mathcal T_{\dv , p} \boldsymbol \tau = \nu - \Id  \mbox{ a.e.\,on } \Omega_p. }
  \end{equation}
  {For a given Stein kernel
    $\pmb \tau$} the $\pmb\tau$-\emph{kernel-Stein operator} {acting
    on scalar functions $g$} is the $\R^d$-valued operator
\begin{align*}
  \mathcal{A}_pg(x) =  \mathcal{T}_{\dv , p}\big({\pmb
    \tau} g\big) (x) =   \pmb\tau(x) \grad g(x)- (x-\nu)g(x)
\end{align*}
with domain {$\mathcal{F}(\mathcal{A}_p)$}.
  \end{definition}


  \begin{example}[Gaussian Stein kernel] {For $\mathcal{N}(\nu, \Sigma)$ with pdf $\gamma$,}
 as we shall see in Section
    \ref{sec:stein-kernel}, the constant function with value $\Sigma$
    is a Stein kernel for $\gamma$; the corresponding operator is
   \begin{equation*}
     \mathcal{A}_{\gamma}g(x) = \Sigma \grad g(x) -
     (x-\nu) g(x).
   \end{equation*}
   The only difference to the score-Stein operator from Example
   \ref{ex:scorestinop} is the position of the covariance matrix
   $\Sigma$.
  \end{example}

  {We note that the definition of a Stein kernel does not make any assumption about its nullity, but in view of Proposition \ref{sec:diff-stein-oper-directional}, often a Stein kernel with zero nullity may be desirable.}

  As argued in the Introduction, Stein kernels play an important role
  in the study of probability distributions. We defer a detailed study
  (existence, construction, examples, applications) to Section
  \ref{sec:stein-kernel}.

\subsubsection{Divergence based second order operators}
  \label{sec:diverg-based-second}

Finally we consider the product rule  \eqref{eq:9} and
corresponding Stein identity \eqref{eq:25bis}, for some
differentiable matrix valued function $\mathbf{F}$.  {We deduce}  the
family of operators
\begin{align}
\label{eq:4}
  g \mapsto \mathcal{A}_p g
  =
  \mathcal{T}_{\dv, p} \bigl( \mathbf{F}^T \grad g \bigr)
  =
  \left\langle
    \mathcal{T}_{\dv, p} \mathbf{F}, \grad g
  \right\rangle + \left\langle
    \mathbf{F}, \grad^2 g
  \right\rangle_{\mathrm{HS}}
\end{align}
with corresponding class $\mathcal{F}(\mathcal{A}_p)$ the collection of
$g$ such that $\grad g \in \dom(p, \mathbf{F})$.

{In line with the previous considerations, two
   choices for $\mathbf{F}$ stand out:}
\begin{enumerate}[(i)]
\item $\mathbf{F} = I_d$, so that
\begin{equation}
  \label{eq:Acla}
  \A_{p}  g  =  \left\langle \grad \log p , \grad g \right\rangle
  +  \Delta g
\end{equation}
with class $\mathcal{F}(\mathcal{A}_p)$ the collection of $g$ such
that $\grad g \in \dom(p, I_d)$;
\item $\mathbf{F} = \pmb \tau$ a  Stein kernel for  $p$ so that
\begin{equation}\label{eq:courtadekernel}
  \mathcal{A}_p g =    \langle \nu - {\mathrm{Id}}, \grad g   \rangle
  + \langle \pmb\tau, \grad^2 g\rangle_{\mathrm{HS}}
\end{equation}
with domain $\mathcal{F}(\mathcal{A}_p)$ the collection of $g$ such
that $\grad g \in \dom(p, \pmb\tau)$.
\end{enumerate}

\begin{example}
[Gaussian second order operator]  \label{ex:steinfiffufi} {For
  $\mathcal{N}(\nu, \Sigma)$ with pdf $\gamma$,}
   \eqref{eq:Acla}  yields
   \begin{equation*}
     \mathcal{A}_{\gamma}g(x) = \Delta g(x) - \left\langle \Sigma^{-1}(x
       - \nu), \grad g(x)  \right\rangle
   \end{equation*}
   whereas \eqref{eq:courtadekernel} with $\pmb \tau = \Sigma$ yields
      \begin{equation*}
     \mathcal{A}_{\gamma}g(x) = \left\langle \Sigma, \grad^2 g
     \right\rangle_{\mathrm{HS}} - \left\langle  x
       - \nu, \grad g(x)  \right\rangle.
   \end{equation*}

{
}
 \end{example}

\begin{example}[Generators of diffusions]
Infinitesimal generators of
    multivariate diffusions as studied e.g.\ in
    \cite{gorham2016measuring, fang2018multivariate} are of the form
    \eqref{eq:4}, with $\mathbf{F}$ imposed by the properties of the
    underlying process.  Indeed consider
    a measure $\mu$ with pdf  $p$ in $\Cont^2(\R^d)$ which is the
    ergodic measure of the It\^o stochastic differential equation
\begin{align*}
  \dl Z_t  =  b(Z_t) \,\dl t  + \sigma(Z_t)
  \,\dl  B_t,
  \quad Z_0 = x,
\end{align*}
where $B_t$ is a standard $d$-dimensional Brownian motion,
$b \Colon \R^d \to \R^d$ is a sufficiently regular (typically Lipschitz)
drift coefficient and $\sigma \Colon \R^d \to \R^{d \times m}$ is a
sufficiently regular (again, typically Lipschitz) diffusion
coefficient.  Let  $\mathbf{a}  = \sigma \sigma^T$ denote the
covariance
coefficient of the process.  The Stein operators  from \cite[Theorem
2]{gorham2016measuring})  are of the form
\begin{align*}
  \mathcal{A}_{p} g & =  \frac12 \mathcal{T}_{\dv , p} \left( (\mathbf{a} +
                         \mathbf{c})\grad g \right)
\end{align*}
where $\mathbf{c}$ is a differentiable skew-symmetric
matrix-valued function such that $\mathcal{T}_{\dv , p}\mathbf{c}(x)$ is
nonreversible.
\end{example}

\section{More about Stein kernels}
\label{sec:stein-kernel}
In this section we study the Stein kernels: their definition,
construction and some applications towards distributional comparisons.

\subsection{Existence and  construction}
\label{sec:generalities}

{From Definition \ref{def:multiteikdef}, a matrix-valued function
  $\pmb \tau$ is a Stein kernel for a continuously differentiable
pdf  $p$ if and only if $\pmb \tau \, p$ is continuously
  differentiable on $\Omega_p$ and {\eqref{eq:steinkenrlnelnr} holds}.
By the product rule
  \eqref{eq:39plus}, the score function and the divergence operator
  are linked through
 \begin{equation} \label{scorelink} \mathcal{T}_{\dv , p} \mathbf{F} = \mathbf{F} \rho_{p} +
  \dv (\mathbf{F}).
  \end{equation}
  Hence for a pdf $p$ with
  mean $\nu$, any  continuously differentiable matrix-valued function
  $\pmb \tau$
  satisfying  Lebesgue-almost
surely on  $\Omega_p$,
  $$ \pmb \tau\,  \rho_p  + {\rm div} ( \pmb \tau ) = \nu -
  \mathrm{Id} $$ is a Stein kernel for $p$.  This leads to the
  following simple explicit construction of a {family of} Stein
  kernels.

 \begin{lemma}\label{lem:de}
   A {Stein kernel can be constructed as}
  $$ {\pmb \tau} =  \frac{1}{\alpha + \beta} \mathbf{F}$$
   {where $ \mathbf{F}$ is}
   a
   continuously
differentiable matrix 
such that for some constant
   $\alpha, \beta \in \R$, $\alpha + \beta \ne 0$ and a function
   $\mathbf{r}\Colon \R^d \rightarrow \R^d$,
  \begin{align}
  \mathbf{F}(x)  \rho_{p} (x) &= \alpha (\nu - x ) + \mathbf{r}(x) \label{constr1} \\
  \dv ( \mathbf{F}(x) ) &= \beta (\nu - x) - \mathbf{r}(x) \label{constr2} .
  \end{align}
 \end{lemma}

 \begin{proof}
  Inserting \eqref{constr1} and \eqref{constr2} into \eqref{scorelink} gives the assertion.
  \end{proof}

 \begin{rem}\label{rem:impo}
     Using product rule \eqref{eq:42} and
a density argument, it is straightforward to see that
showing some differentiable matrix valued function is a Stein kernel
can be done equivalently by checking that {the Stein identity}
 \begin{equation}
  \label{eq:defstek2}
  \E_{{p}} \left[ \pmb{\tau} \grad g
  \right] = \E_{{p}} \left[    ( {{\rm Id}}-\nu) g \right]
\end{equation}
holds at least for $g \in \Cont_\comp^\infty(\Omega_p)$;
this concurs with the original
definitions of {multivariate} Stein kernels as proposed e.g.\ in
\cite{nourdin2013integration}.

 \end{rem}

\begin{rem}
{Equation \eqref{eq:defstek2} can be used to relate {the moments of
    ${p}$ with the moments of a Stein kernel}. For instance, if
  $\pmb \tau$ is a Stein kernel for $p$, if $p$ has finite variance
  {$\Sigma$},
    and if ${\mathrm{Id}}-\nu \in \dom(p,\pmb \tau)$ {(a
    mild condition)}, it follows
  directly from \eqref{eq:defstek2}  that}
$ \E \left[ \pmb \tau \right] = {\Sigma}.$
\end{rem}

\begin{example}[Gaussian Stein kernel]
  {
  For {$\mathcal{N}(\nu, \Sigma)$},
\eqref{eq:gaussianstandard}}
entails that $\pmb \tau(x) = \Sigma$ is a Gaussian Stein kernel.
Example \ref{sec:stein-oper-Gauss} will illustrate
 that it is not the only one.
\end{example}

{In the one-dimensional case, the first order {differential equation}
{corresponding to
  \eqref{eq:defstek2}} is easy to solve.
  {If $ \Omega_p $ is an interval (possible infinite), then
  }
  \[
   \tau(x)
   =
   \frac{1}{p(x)} \int_{-\infty}^x (\nu-u) \, p(u) \,\dl u
   = - \frac{1}{p(x)} \int_{x}^\infty (\nu-u) \, p(u) \,\dl u
  \]
  {is the unique Stein kernel with $ \tau(x) \, p(x) $ tending to zero
  as $ x $ approaches the boundary of $ \Omega_p $ or infinity.} 
  We
  refer to \cite{ernst2020first,saumard2018weighted} for an overview
  and
  bibliography on one-dimensional Stein kernels. In
  higher dimensions, given that there exist infinitely many functions
  which share a divergence, the Stein kernel is by no means uniquely
  defined. Aside from the construction in Lemma \ref{lem:de}, we also
  have {a}}
result, due to \cite{courtade2017existence},
  which
  provides a characterizing Stein
  kernel under the assumption that the distribution admits a
    Poincar\'e constant, as follows.

\begin{definition}
  \label{def:poincareconstant} {A probability distribution} $P$
  satisfies a Poincar\'e inequality if there exists a constant
  $C<\infty$ such that for every locally Lipschitz function
  $\varphi \in L^2(P)$ with expectation $\E_P \varphi = 0$, {we have}
  $ \E_P \varphi^2 \leq C \E_P \norm{\grad \varphi}^2.$ The
  smallest constant for which this inequality holds is the Poincar\'e
  constant of $P$, denoted $C_{{P}}$; it is also referred to as $P$'s
  spectral gap.
\end{definition}

\begin{example}[Stein kernels under a Poincar\'e inequality] \label{ex:courtstek}
In   \cite{courtade2017existence}
a Stein kernel is defined as any matrix $\pmb\tau$ such that
\eqref{eq:eqsteinidddddd1} holds for all $g\Colon \R^d \to \R^d $ in the
Sobolev space ${W}_p^{1,2}$.  With this definition it is shown
that if $p$ satisfies a Poincar\'e-type inequality then it possesses a
Stein kernel such that each {row} is the gradient of a vector field.
 \end{example}

 Our next main result, Theorem \ref{th:artstein}, 
 provide{s another}
  explicit construction 
  which is valid under very weak
   assumptions. We start with a formula for bivariate
  {For a bivariate density $ p $, denote by $ p_1 $ the marginal density of the first component (i.~e., in direction $ e_1 $) and by $ p_{2 \mid 1}(x_2 \mid x_1) := p(x_1, x_2)/p_1(x_1) $ the conditional density of the second component given the first one. Next, define $ \partial_1 p_{2 \mid 1} $ and $ \partial_2 p_{2 \mid 1} $ as $ \partial_1 p_{2 \mid 1}(x_2 \mid x_1) := \partial_{x_1}(x_2 \mid x_1) $ and $ \partial_2 p_{2 \mid 1}(x_2 \mid x_1) := \partial_{x_2}(x_2 \mid x_1) $.
   }

\begin{lemma}[Bivariate Stein kernels]\label{prop:transport-definition-1}
  Let $p$ be a continuous pdf on $\R^2$ {which is $\mathcal{C}^1$ on its
  support $\Omega_p$}.
  {Suppose that each $ x_1 \in \Omega_{p_1} $ has a neighbourhood $ U $ such that $ \int_{- \infty}^\infty \sup_{u \in U}
  \bigl| \partial_1 p_{2 \mid 1}(v \mid u) \bigr| \,\dl v < \infty $.}
  Moreover, suppose that $p_1$ has finite mean $\nu_1$, and let
  $\tau_1$ be the corresponding univariate kernel. Set
  $\tau_{11}(x_1, x_2) = \tau_1 (x_1)$ and
{%
\begin{equation}
\label{eq:67}
  \begin{split}
  \tau_{12}(x_1, x_2)
  &=
  - \frac{\tau_1(x_1)}{p_{2 \mid 1}(x_2 \mid x_1)}
  \int_{- \infty}^{x_2} \partial_1 p_{2 \mid 1}(v \mid x_1) \,\dl v
  \\ &=
  \frac{\tau_1(x_1)}{p_{2 \mid 1}(x_2 \mid x_1)}
  \int_{x_2}^\infty \partial_1 p_{2 \mid 1}(v \mid x_1) \,\dl v
  \, .
  \end{split}
\end{equation}
}%
  Then the vector
  $(x_1, x_2) \mapsto
  \left(
    \tau_{11}(x_1, x_2), \tau_{12}(x_1, x_2) \right)
  $ is a Stein kernel for $p$ in the direction $e_1$. {A Stein kernel
  for $p$ in the direction $e_2$ is defined similarly, by reversing
  the roles of $x_1$ and $x_2$.}
\end{lemma}

\begin{proof}
  {First, observe that both expressions in the right-hand side of
  \eqref{eq:67} agree because $
    \int_{- \infty}^\infty \partial_1 p_{2 \mid 1}(v \mid x_1) \,\dl v
    =
    \partial_{x_1} \int_{- \infty}^\infty p_{2 \mid 1}(v \mid x_1) \,\dl v
    = 0 $: the stated conditions allow us to differentiate under
    the integral sign.
  }
  The fact that
  $ \sum_{j=1}^2 {\partial_{x_j}} \bigl[ \tau_{1j}(x_1, x_2) \, p(x_1, x_2) \bigr] =
  (\nu_1 - x_1) \, p(x_1, x_2) $ for all $(x_1, x_2) \in {\Omega_p} $ follows easily
  from the definitions of $\tau_{11}$ and $\tau_{12}$ and some
  straightforward manipulations.
  {Finally, to check differentiability of $ (\tau_{11}, \tau_{12}) p $, rewrite \eqref{eq:67} as
  \[
   \tau_{12}(x_1, x_2) \, p(x_1, x_2)
   =
   \tau_1(x_1) \, p_1(x_1)
   \int_{x_2}^\infty \partial_1 p_{2 \mid 1}(v \mid x_1) \,\dl v
   \, .
  \]}%
\end{proof}

\begin{rem}
  The inspiration for formula \eqref{eq:67} is \cite[Equation
  (9)]{BaBaNa03}, where a similar quantity is introduced via a
  transport argument. To see the connection, assume that
  $\Omega_p =\Omega \times \Omega$ for an open set
  $\Omega \subseteq \R$.
  Fix $i=1$ and, for each
  $t, t', x_2 \in \Omega$ let $x_2 \mapsto T_{t, t'}(x_2)$ be the
  {map $\Omega \to \Omega$}
  transporting the conditional pdf at $x_1=t$ to that at $x_1=t'$.
  By the transformation formula, we have
  {
  \begin{equation}
  \label{eq:56}
    p_{2 \mid 1}(x_2 \mid t)
    =
    p_{2 \mid 1} \bigl(T_{t, t'}(x_2 \mid t') \bigr)
    \, \partial_{x_2} T_{t, t'}({x_2}).
  \end{equation}%
  }%
  {In particular $T_{t, t}({x_2}	) = {x_2}$ and
  $ \partial_{x_2} T_{t, t}({x_2})=1$.}
  {Assume that the function $ (t, t', x_2) \mapsto T_{t, t'}(x_2) $
  is twice continuously differentiable.}
  Taking derivatives in \eqref{eq:56} with respect to $t'$ and setting
  $t'=t=x_1$ we deduce
  that
  {
  \begin{align*}
   \lefteqn{
     \partial_2 p_{2 \mid 1}(x_2 \mid x_1)
     \, \partial_{t'} T_{t, t'}(x_2) \big|_{t'=t=x_1}
       +
     \partial_1 p_{2 \mid 1}(x_2 \mid x_1)
   } \qquad
     \\ & \kern 5em \null +
   p_{2 \mid 1}(x_2 \mid x_1)
   \, \partial_{t'} \partial_{x_2} T_{t,t'}(x_2) \big|_{t'=t=x_1}
   =
   0
   \, .
  \end{align*}
  Interchanging the differentiation in the last term and applying the
  product rule, we obtain
  \[
   \partial_{x_2} \Bigl[
     p_{2 \mid 1}(x_2 \mid x_1)
     \, \partial_{t'} T_{t,t'}(x_2) \big|_{t'=t=x_1}
   \Bigr]
     +
   \partial_1 p_{2 \mid 1}(x_2 \mid x_1)
   =
   0
   \, .
  \]
  Integrating by $ x_2 $, applying \eqref{eq:67} and multiplying by
  $ p_1(x_1) $, we find that the functions
  \[
   p(x_1, x_2) \, \partial_{t'} T_{t,t'}(x_2) \big|_{t'=t=x_1}
   \kern 1.5em \text{and} \kern 1.5em
   p(x_1, x_2) \, \frac{\tau_{12}(x_1, x_2)}{\tau_1(x_1)}
  \]
  differ only by a function of $ x_1 $.}
  {Thus, the ratio of the Stein kernel components corresponds to the
  direction of
  {the} transport from $x_1$ to $x_2$ at $t$.}  Another
  construction connected to
  optimal transport considerations is provided in \cite{1i2018stein}.
 \end{rem}

\noindent
{Next we} compute {this bivariate kernel} for several examples.
\begin{example}[Bivariate Gaussian] {For the bivariate Gaussian distribution}
  $\mathcal{N}(\nu, \Sigma)$
  direct computations of the kernel in Lemma
  \ref{prop:transport-definition-1} lead to
  $ \pmb \tau(x) = \Sigma.$
\end{example}
\begin{example}[Bivariate Student]\label{sec:exist-constr} {For the bivariate Student distribution $ t_k(\nu, \Sigma)$}
  direct computations of the kernel in Lemma  \ref{prop:transport-definition-1} {give}
  $ \pmb \tau(x) = \frac{1}{k-1} \left( (x - \nu) (x - \nu)^T + k
    \Sigma \right)$ which we shall encounter again in Example
  \ref{sec:mult-stud-t}. See also example \ref{ex:studentfirstexample}. Note that this Stein kernel can{not} be
  written as a gradient.
\end{example}

\begin{example}[Bivariate normal-gamma] \label{ex:norm-gamm}
 Direct computations of the kernel in Lemma
 \ref{prop:transport-definition-1} for  the bivariate normal-gamma
  distribution {$NG(\mu, \lambda, \alpha, \beta) $}  with pdf
\begin{equation*}
  p(x_1, x_2) = \frac{\beta^{\alpha} \sqrt
    \lambda}{ \Gamma(\alpha) \sqrt{2\pi}} x_2^{\alpha-\frac{1}{2}} e^{-x_2
    \beta - \frac{1}{2} x_2 \lambda(x_1-\mu)^2}
\end{equation*}
on $(x_1, x_2) \in \R \times \R^+$
 give
\begin{equation*}
  \pmb \tau(x_1, x_2) =
  \begin{pmatrix}
    \frac{\kappa(x_1-\mu)^2 + 2 \beta}{\kappa(2\alpha-1)} & \frac{2
      (x_1-\mu) x_2}{2\alpha - 1} \\
    \frac{x_1-\mu}{2\beta} & \frac{x_2}{\beta}
  \end{pmatrix}
\end{equation*}
where, for the second line, we simply exchange  the roles of $x_1$
  and $x_2$ in Lemma \ref{prop:transport-definition-1}.
Notably, this
   Stein kernel is not  symmetric.
\end{example}

Next, we state and prove Theorem~\ref{th:artstein}, which,
inspired by \cite{artstein2004solution}, gives a mechanism
{extending}
the bivariate construction from Lemma
\ref{prop:transport-definition-1} to arbitrary dimensions.

\begin{thm}\label{th:artstein}
Let $p \Colon \R^d \to (0,  \infty)$ be a
  continuously twice differentiable pdf  on $\R^d$ with
  \begin{equation*}
    \int \frac{\norm{ \grad p }^2}{p} {< \infty} \quad \mbox{ and } \quad \int \norm{ \grad^2(p) }
    < \infty
  \end{equation*}
{and such that $p$ has finite variance. Denote  by $\tau_i^{(1)}, i = 1, \ldots, d$ the marginal Stein
  kernels.}
Then, for any direction $e_i, i=1, \ldots, d$ there exists a Stein
kernel {$\tau_{p,i}^{(d)}(x)$} for $p$ in direction $e_i$,
\begin{equation*}
  \tau_{p,i}^{(d)}(x) =\tau_i^{(1)}(x_i)
  \begin{pmatrix}
    \tau_{i,1}^{(d)}(x \mid x_i)\, \cdots \,
    \tau_{i,i-1}^{(d)}(x \mid x_i) \quad 1 \quad \tau_{i,i+1}^{(d)}(x \mid x_i)  \, \cdots  \,
    \tau_{i,d}^{(d)}(x \mid x_i)
  \end{pmatrix}{^T}
\end{equation*}
{such that}
$\tau_{i}^{(d)}(x \mid x_i) {= (\tau_{i,1}^{(d)} (x \mid x_i), \ldots,
  \tau_{i,d}^{(d)} (x \mid x_i))^T}$ {as a function of $x$}
solves the
equation
  \begin{equation}\label{eq:60}
    \mathcal{T}_{\dv , p}(\tau_{i }^{(d)}(x \mid x_i)) = \rho_i(x_i).
  \end{equation}
Here $\rho_i(x_i) = p_i'(x_i)/p_i(x_i)$  is the score function of the
marginal {of $p$ in direction $e_i$} and $x=(x_1, \ldots, x_d)$.
{Moreover,   $ \tau_{i,i}^{(d)} (x  \mid x_i) = \tau_i^{(1)}(x_i)$.}
\end{thm}
\begin{proof}
  {Let $e_i$ be a unit vector and $p_i$ the marginal of
    $p$ in direction
    $e_i$.}  The result is almost immediate from \cite[Theorem
  4]{artstein2004solution},
  where it is proved (see middle of page 978) that, under the stated
  conditions, there exist continuously differentiable vector fields
  ${\tau_i^{(d)} (x \mid x_i)} $ {as functions of $x$} such that
   \begin{align*}
    \frac{\dv_x(\tau_i^{(d)}(x \mid x_i)(x) \, p(x))}{p(x)} = \frac{p_i'({x_i})}{p_i({x_i})}
  \end{align*}
  {and such that the component of $\tau_i^{(d)} (x  \mid  x_i)$ in direction $e_i$ is 1;  $\langle\tau_i^{(d)} (x  \mid  x_i), e_i \rangle = 1$ for all $x$.  Thus,  \eqref{eq:60} holds.}
  To see the connection with
Stein
  kernels, write
$    \tau_{ij}^{(d)}(x) = \tau_i^{(1)}(x_i) \tau_{ij}^{(d)}(x \mid x_i).  $
Then
\begin{align*}
  \sum_{j=1}^d \partial_{x_j} \Bigl( \tau_{ij}^{(d)}(x) \, p(x) \Bigr)
  & =
  \sum_{j=1}^d \partial_{x_j} \Bigl( \tau_{ij}^{(d)}(x \mid x_i) \, p(x) \,
  \tau_i^{(1)}(x_i) \Bigr)
\\ & =
  \sum_{j=1}^d \partial_{x_j} \Bigl( \tau_{ij}^{(d)}(x \mid x_i) \,
  p(x) \Bigr) \, \tau_i^{(1)}(x_i)
    \\ & \kern 3em \null +
    \sum_{j=1}^d \tau_{ij}^{(d)}(x \mid x_i) \, p(x) \,
    \partial_{x_j} \bigl( \tau_i^{(1)}(x_i) \bigr)
\\ & =
  \rho_i(x_i) \, p(x) \, \tau_i^{(1)}(x_i) +
    p(x) \, \partial_i \tau_i^{(1)}(x_i)
  \, ,
\end{align*}
where in the last line we use \eqref{eq:60} in the first sum and
$\partial_{x_j}(\tau_i^{(1)}(x_i)) = 0$ for all $j \neq i$ in the
second sum. By the definition of the univariate Stein kernel,
\begin{align*}
 \partial_i \tau_i{^{(1)}}(x_i)
 =
 - \rho_i(x_i) \tau_i{^{(1)}}(x_{i}) + \E[X_i] - x_i
 \, .
\end{align*}
{The claim follows.}
\end{proof}

\begin{rem}
  The proof of \cite[Theorem 4]{artstein2004solution} provides an explicit
  solution of \eqref{eq:60}, allowing us to generalize the bivariate
  construction from Lemma \ref{prop:transport-definition-1}
  {to the $ d $-variate case under the same conditions.}
  {
  {A} Stein kernel in direction $ e_1 $ can be {constructed} 
  in terms
  of the conditional densities
  \[
   p_{j, j+1, \ldots, d \mid 1}(x_j, x_{j+1}, \ldots, x_d \mid x_1)
   =
   \frac{p_{1, j, j+1, \ldots, d}(x_1, x_j, x_{j+1}, \ldots, x_d)}%
        {p_1(x_1)}
  \]
  and marginal cumulative distribution functions
  $ P_i(x_i) = \int_{- \infty}^{x_i} p_i(v) \,\dl v $
  as follows: firstly, set $ \tau_{1,1}^{(d)}(x_1, x_2, \ldots, x_d)
  = \tau_1(x_1) $. For $ j = 2, 3, \ldots, d - 1 $, set
  \begin{align*}
    \tau_{1,j}^{(d)}(x_1, x_2, \ldots, x_d)
   &=
    \tau_1(x_1) \,
    \frac{p_1(x_1) \, p_2(x_2) \cdots p_{j-1}(x_{j-1})}%
         {p(x_1, x_2, \ldots, x_d)}
        \\ & \kern 5em \null \times
    \int_{x_j}^\infty
    \partial_1
    p_{j, j+1, \ldots, d \mid 1}(v, x_{j+1}, x_{j+2}, \ldots, x_d \mid x_1)
    \,\dl v
      \\ & \kern 3em \null +
    \tau_1(x_1) \,
    \frac{p_1(x_1) \, p_2(x_2) \cdots p_{j-1}(x_{j-1}) \, P_j(x_j)}%
         {p(x_1, x_2, \ldots, x_d)}
        \\ & \kern 5em \null \times
    \partial_1
    p_{j+1, j+2, \ldots, d \mid 1}(x_{j+1}, x_{j+2}, \ldots, x_d \mid x_1)
    \, .
  \end{align*}
  Finally, set
  \[
    \tau_{1,d}^{(d)}(x_1, x_2, \ldots, x_d)
    =
    \tau_1(x_1) \,
    \frac{p_1(x_1) \, p_2(x_2) \cdots p_{d-1}(x_{d-1})}%
         {p(x_1, x_2, \ldots, x_d)}
    \int_{x_d}^\infty
    \partial_1
    p_{d \mid 1}(v \mid x_1) \,\dl v
  \]
  (the partial derivative $ \partial_1 $ is defined as in the
  bivariate case).
  A straightforward, though somewhat involved calculation shows
  that $ \bigl( \tau_{1,1}^{(d)}, \ldots, \tau_{1,d}^{(d)} \bigr)^T $
  is indeed a Stein kernel in direction $ e_1 $.
  Stein kernels in other directions can be {obtained}
  analogously
  by rotating the indices.}
\end{rem}

\subsection{Stein kernels for elliptical distributions}

\label{sec:stein-kern-ellipt-1}


{In this subsection we} construct {a family of} Stein kernels for
 {any member of the family of}
elliptical distributions.

\begin{definition}
{The}
multivariate elliptical
distribution $E_d(\nu, \Sigma, \phi)$ {on $\R^d$ has pdf}
\begin{equation}
  \label{eq:ellipticdistr}
  p(x) = \kappa  \, [{\rm{det}}( \Sigma) ]^{-1/2} \phi \left( \frac{1}{2}
    (x-\nu)^T\Sigma^{-1}(x-\nu) \right), \, \quad  x \in \R^d,
\end{equation}
for $\phi \Colon \R^+ \to \R^+$ a measurable function,
 $\nu \in \R^d$, $\Sigma = (\sigma_{ij})$ a symmetric positive definite
$d \times d$ matrix, and
  $\kappa$ the
normalizing constant.
\end{definition}

Note that the matrix $\Sigma$ in
 \eqref{eq:ellipticdistr}
is not necessarily the covariance matrix; also not all choices of
$\phi$ lead to well-defined densities, see \cite{landsman2008stein}
for a discussion and references.

Some prominent members of the elliptical
family are
the
 Gaussian distribution $\mathcal{N}_d({\nu},  \Sigma)$, with
 $\phi(t) = e^{-t}$;
the  power exponential distribution,
with $\phi(t) = \mathrm{exp}(- b_{p, \zeta} t^{\zeta})$ for $\zeta>0$
and $ b_{p, \zeta}$ a scale factor; the multivariate Student-$t$
distribution,
with
  $\phi(t) = \left( 1+2t/k \right)^{-(k+d)/2}$;  and
the spherical
  distributions $E_d(0, \mathrm{I}_d, \phi)$.
For simplicity, we assume that $\phi(t)>0$ for
    all $t\ge0$ so that $\Omega_p = \R^d$.
{Example 2.1 in \cite{fathi2022relaxing}}
shows that in order to find Stein
kernels for elliptical distributions, it suffices to consider {spherical distributions, so that}
$\Sigma = \mathrm{I}_d$ and $\nu=0$. {However as the notion of Stein kernel is not as broad in \cite{fathi2022relaxing} we provide a proof here.}
 \begin{prop}
 \label{prop:sigmaid}
 The application
$ \pmb \tau \mapsto \left[ x\mapsto \Sigma^{1/2} \pmb \tau(\Sigma^{-1/2} (x-\nu) ) \Sigma^{1/2} \right] $
{maps Stein kernels of $E_d(0, \mathrm{I}_d, \phi)$ to Stein kernels of  $E_d(\nu, \Sigma, \phi)$ and}
 is a bijection.

 \end{prop}
 \begin{proof}
 {Fix  a matrix $ \mathrm A $ and a
vector $ b $ of a proper dimension, define $ \tilde{\mathrm A}(x) := \mathrm A x + b $.
By the chain rule, we have
$
 \dv(\mathbf F \circ \tilde{\mathrm A})
 =
 \bigl( \dv(\mathbf F \mathrm A^T) \bigr) \circ \tilde{\mathrm A}
$
for any suitable matrix-valued function $ \mathbf F $. Alternatively, one can write
$
 \dv \bigl( (\mathbf F \mathrm A^{- T}) \circ \tilde{\mathrm A} \bigr)
 =
 (\dv \mathbf F) \circ \tilde{\mathrm A}
$.
Letting $ \mathbf F = p \boldsymbol \tau $, where $ p $ is a scalar-valued function, we obtain
$
 \mathcal T_{\dv , p \circ \tilde{\mathrm A}}
 \bigl( (\boldsymbol \tau \mathrm A^{- T}) \circ \tilde{\mathrm A} \bigr)
 =
 (\mathcal T_{\dv , p} \boldsymbol \tau) \circ \tilde{\mathrm A}
$.
Now let $ \tilde{\mathrm A}(x) := \Sigma^{-1/2}(x - \nu) $ and let $ p $ be the density of
$ E_d(0, \mathrm I_d, \phi) $. Clearly,
$ q := [ \det(\Sigma) ]^{-1/2} (p \circ \tilde{\mathrm A}) $
is the density of $ E_d(\nu, \Sigma, \phi) $. If $ \boldsymbol \tau $ is a Stein kernel
for $ E_d(0, \mathrm I_d, \phi) $, then $ \mathcal T_{\dv , p} \boldsymbol \tau(z) = - z $.
Letting $ z = \tilde{\mathrm A}(x) $, we then have
$
 \mathcal T_{\dv , q} \bigl( (\boldsymbol \tau \Sigma^{1/2}) \circ \tilde{\mathrm A} \bigr)(x)
 =
$
$
 \mathcal T_{\dv , p \circ \tilde{\mathrm A}} \bigl( (\boldsymbol \tau \Sigma^{1/2}) \circ \tilde{\mathrm A} \bigr)(x)
 =
 (\mathcal T_{\dv , p} \boldsymbol \tau) \bigl( \tilde{\mathrm A}(x) \bigr)
 =
 \Sigma^{-1/2}(\nu - x)
$. Multiplying by $ \Sigma^{1/2} $ from the left, we conclude that
$ x \mapsto \Sigma^{1/2} \boldsymbol \tau \bigl( \Sigma^{-1/2}(x - \nu) \bigr) \Sigma^{1/2} $
is a Stein kernel for $ E_d(\nu, \Sigma, \phi) $.}
 \end{proof}

The  score function for $E_d(\nu,\Sigma, \phi)$ is
\begin{equation*}
  \rho_p(x) = \Sigma^{-1}(x-\nu)
  \frac{\phi'((x-\nu)^T\Sigma^{-1}(x-\nu)/2)}{\phi((x-\nu)^T\Sigma^{-1}(x-\nu)/2)},  \, \quad x \in \R^d.
\end{equation*}
Combining this special form with Lemma \ref{lem:de} yields
a family
of Stein kernels for members of the elliptical distributions.

  \begin{prop}\label{lem:steikellipt2prop}
    For $d\ge2$, letting $t = (x-\nu)^T\Sigma^{-1}(x-\nu)/2$, the
    matrix-valued functions
  \begin{equation}
    \label{lem:steikellipt2}
    \pmb \tau_{\delta}  (x) =
    \frac{\frac{\phi''(t)/\phi'(t)}{\phi'(t)/\phi(t)} - \delta }{(2- \delta)(d-1)}\left(  \left(
        \frac{d-1}{\delta \frac{\phi'(t)}{\phi(t)} -
          \frac{\phi''(t)}{\phi'(t)}} + 2 t \right)\Sigma  - (x-\nu) (x-\nu)^T   \right)
  \end{equation}
are  Stein kernels for  $E_d(\nu,\Sigma, \phi)$   for all $\delta \ne
2$, as long as {$\pmb \tau_{\delta} \in \ACL_{\mathrm{loc}}^1(\R^d)$}.
\end{prop}

\begin{proof}

  Without loss of generality we set $\nu=0$ and
  $\Sigma= \mathrm{I}_d$.  Using the temporary notation
  $\psi(t) = \phi(t)/\phi'(t)$, the score function is
  $ \rho_p(x) ={x}/{\psi(t)}$ with $t = x^T x/2$.  To use
  \eqref{constr1} with $r=0$, the equation
$$\mathbf{F}(x) x  =  -\alpha  \psi(t) x $$
is solved for example by $\mathbf{F}(x)  = - \alpha \frac{ \psi(t)}{2 t} x x^T$.
More generally a family of solutions of  \eqref{constr1}  with $r=0$
is given by
 matrix-valued functions  of the form
  \begin{align*}
\mathbf{F}(x) =  -\alpha \frac{\psi(t)}{ 2( t+f(t)) } \left(  xx^T + 2f(t)  \right)
  \end{align*}
  for some $f\Colon\R\to\R$; it is easy to check that
  $\mathbf{F}(x)\rho_p(x) =- \alpha x$.  For \eqref{constr2} with
  $r=0$ the flexibility in the choice of $f$ enters; we introduce
$b(t) = -\alpha \psi(t)/( 2( t+f(t)) )$ so that
$\mathbf{F}(x) = b(t) \left( x x^T + 2 f(t)
  \right)$.  For
  all $1 \leq i \leq d$, we have
\begin{align*}
&\dv
 \left(b(t) \left( xx^T + 2f(t) \right) \right)   =\left( 2t b'(t) +  2(b'(t) f(t)+ b(t)  f'(t))+ (d+1) b(t)  \right) x .
\end{align*}
For \eqref{constr2} with $r=0$ to hold, it suffices to choose $f$ such
that, for all $t \in \R$,
\begin{equation}\label{b-equation}  2(t+f(t)) b'(t) +  (2f'(t)+  (d+1)) b(t)
= -\beta
\end{equation}
for some $\beta$. Since $b(t) = -\alpha\psi(t)/( 2( t+f(t)) )$, simple
calculations  lead to the requirement that
\begin{align*}
t+f(t) = - \frac{(d-1)\psi(t)}{ 2 \psi'(t) - 2\beta / \alpha }
\end{align*}
at all $t$. With this in hand, we easily obtain
\begin{equation*}
  \mathbf{F}(x) = \frac{-\beta  + \alpha \psi'(t)}{d-1}\left( x x^T - 2 \left(
      \frac{(d-1)\psi(t)}{ 2 \psi'(t) - 2\beta  / \alpha} +
      t \right)\right)
\end{equation*}
Plugging in $\psi(t) = \phi(t)/\phi'(t)$ whose derivative is
$\psi'(t) = 1 - \phi(t) \phi''(t)/(\phi'(t))^2$, and dividing by
$\alpha + \beta$, then setting $\delta = 1 - \beta / \alpha$,
Proposition \ref{lem:steikellipt2prop} ensues for
$\nu=0, \Sigma = \mathrm{I}_d$; the general formula follows from Lemma
\ref{prop:sigmaid}.
\end{proof}

%

The equation \eqref{b-equation} can be reparameterized as follows.

\begin{cor}
\label{prop:stkernelell}
{Let $a, b \Colon \R^d \rightarrow \R$ be two  continuously differentiable
functions such that, for all $t \geq 0$,}
\begin{equation}
\label{eq:linkab}
\frac{(a(t) \phi(t))'}{\phi(t)}   + 2t \frac{(b(t) \phi(t))'}{\phi(t)}
+ (d+1)b(t) + 1 = 0.
\end{equation}
Then
$$\pmb \tau_{a, b}(x) = a(t) \Sigma + b(t) \,(x-\nu)(x-\nu)^T$$
with $t = \frac{1}{2}(x-\nu)^T \Sigma^{-1} (x-\nu)$  is a Stein
kernel for $E_d(\nu, \Sigma, \phi)$.
\end{cor}
\begin{proof}
In  \eqref{b-equation} take $\beta = -1$ and require that
$a(t)$ satisfies
$$ ( a \phi)'(t) = 2 \phi(t) (b f)'(t) + \frac{t \phi(t)}{t + f(t)}.$$
Then the assertion follows from Proposition \ref{lem:steikellipt2prop}.
\end{proof}

{{Many} options {in \eqref{eq:linkab}} are possible. For instance,}
setting $b \equiv 0$ {gives}
$a(t) =\frac{1}{\phi(t)} \int_t^{+\infty} \phi(u) \,\dl u,$
and the
  matrix-valued function
\begin{align}\label{lemm:lemmasteik}
\pmb \tau (x)  = \left( \frac{1}{\phi((x-\nu)^T\Sigma^{-1}(x-\nu)/2)}
\int_{(x-\nu)^T\Sigma^{-1}(x-\nu)/2}^{+\infty} \phi(u) \,\dl u \right) \Sigma
\end{align}
is a Stein kernel for $E_d(\nu, \Sigma, \phi)$ (provided it is
continuously differentiable).
 This recovers
\cite[Theorem 2]{landsman2008stein}.  Setting $a \equiv 0$ leads to:
\begin{align*}
{{\pmb \tau}(x)} = \left( \frac{t^{-\frac{d+1}{2}}}{2\phi(t)} \int_t^{+\infty} u^{\frac{d-1}{2}} \phi(u) \,\dl u \right)(x-\nu)(x-\nu)^T,
\end{align*}
is a Stein kernel for $E_d(\nu, \Sigma, \phi)$; we have so far not
found any use for this formula.

\begin{example}[Stein kernels for the multivariate Gaussian
  distribution]
   \label{sec:stein-oper-Gauss} {As} the multivariate Gaussian has  $\phi(t) = e^{-t}$ {and} $\phi'(t)/\phi(t) = -1$, we {recover that}
   $$ \rho_{\gamma}(x) = - \Sigma^{-1} (x-\nu)$$ is the score
   function of $\gamma$. Since
   $\frac{1}{\phi(t)} \int_t^{\infty} \phi(u) \,\dl u = 1$ for all
   $t$,  \eqref{lemm:lemmasteik} shows that
   $ \pmb \tau_{1} = \Sigma$ is a Stein kernel for $\gamma$. 
Moreover, \eqref{lem:steikellipt2}
 yields, after some
simplifications, the following family $  \pmb \tau_{\delta}(x)$
{given for $\delta \ne 2$ by
    \begin{equation}\label{eq:stkernguau}
      \pmb \tau_{\delta}(x) =  \frac{1}{(2 - \delta)(d-1)} \left(\left(
      d-1 + 2t (1 - \delta)  \right) \Sigma - (1 - \delta) (x - \nu)(x-\nu)^T
      \right)
\end{equation}
are all Stein kernels for $\gamma$.  In particular, the choice
$\delta =1$ recovers $\pmb \tau_{0}(x) = \Sigma,$ although many other
choices are possible.}  {First explorations indicate that in this
example the freedom of choice in $\pmb \tau_{\delta}$ does not
provide improvement over the most natural choice
$\pmb \tau(x) = \Sigma$,
      for comparison of
      normal distributions in Wasserstein distance,
      see Example~\ref{ex:comaoighe}. }
\end{example}

\begin{example}[Stein kernels for the multivariate Student $t$-distribution]
\label{sec:mult-stud-t}

This distribution is an elliptical distribution with
$\phi(t) = (1+2t/k)^{-(k+d)/2}$ and hence
$\phi'(t)/\phi(t) = -(d+k)/(k+2t)$.
 {F}rom $k > 1$ it follows that $d+k>2$ and
  \begin{equation*}
      \frac{1}{\phi(t)} \int_t^{+\infty}  \phi(u) \,\dl u = \frac{k+2t}{d+k-2}.
  \end{equation*}
  Hence  \eqref{lemm:lemmasteik} gives that
  \begin{align}
    \label{eq:steikstud1}
  \pmb \tau_1(x) =  \frac{(x-\nu)^T\Sigma^{-1}(x-\nu)+k}{d+k-2} \Sigma
\end{align}
is a Stein kernel for the multivariate Student distribution for
$k > {2}$; here $k>2$ guarantees that $t_k$ has finite variance.
Also, we note that $ \pmb \tau_1 \in \mathcal{F} (t_k)$.  Similarly,
using that $\phi''(t)/\phi'(t) = -(d+k+2)/(k+2t)$,
\eqref{lem:steikellipt2} gives a family of Stein kernels which
are indexed by $\delta \ne 2$:
\begin{equation}
\label{eq:steintsubet}
\begin{split}
 \pmb \tau_{\delta} (x)
 &=
 \frac{1}{(2 - \delta )(d+k)(d-1)} \Bigl[ \left\{
     (d-1)(k+2t) + 2t ( (1 - \delta) (d+k) + 2 )
 \right\} \Sigma
 \\ & \kern 3em \null
 - \left\{ (1 - \delta) (d+k) + 2\right\} (x - \nu) (x - \nu)^T
 \Bigr]
 \, .
\end{split}
\end{equation}
 We note that $ \pmb \tau_{\delta} \in \mathcal{F}(t_k)$ {for
   $k>2$.} The particular choice of
 $\delta$ such that
 $d-1+( 1- \delta) (d+k) + 2 = 0$, i.e.\
 $\delta  =1  - (d+1)/(d+k) $, eliminates $t$ from  \eqref{eq:steintsubet} and,
 after simplifications, we obtain
 for $k > 2$
 \begin{equation}
  \label{eq:steikstu2}
  \pmb \tau_2(x) = \frac{1}{k-1} \left(  (x - \nu) (x - \nu)^T  + k
    \Sigma \right).
\end{equation} This agrees with the kernel already identified in
Example \ref{sec:exist-constr}.
When $d=1$, then
both $\pmb \tau_1$ and
$\pmb \tau_2$ simplify to $\tau(x)= (x^2+k\sigma^2)/(k-1)$,
the univariate kernel for the Student-$t$ distribution with $k$
degrees of freedom and centrality parameter $\sigma^2$, see e.g.\
\cite[page 30]{LRS16}. We will see in Example \ref{ex:gauvsstudent}
that, when comparing with a Gaussian pdf, neither $\pmb \tau_1$
nor $\pmb \tau_2$ are ``better'' choices; in fact optimizing over
$\delta$ in \eqref{eq:steintsubet} provides possibilities for strict
improvement.

\end{example}

\subsection{Stein kernel discrepancies}
\label{sec:stein-kern-discr}

Stein kernels  provide a natural means for comparing
  distributions  through the discrepancy
$$S(p_2\mid p_1) =  \inf_{\pmb \tau_1, \pmb \tau_2} \E_{p_2}[\norm{
  \pmb \tau_2 - \pmb \tau_1}^2_{\mathrm{HS}}]^{1/2}$$ where the
infimum is taken over all Stein kernels for $p_1$ and for $p_2$.  In,
fact, the specific case $p_1 = \gamma$ the standard normal
distribution has been studied in detail e.g.\ in
\cite{nourdin2013integration,ledoux2015stein,courtade2017existence}
where connections with various classical probability metrics
 as well as information-type
discrepancies are identified. In particular  it is shown in
\cite{ledoux2015stein} that
$$\mathcal{W}_2(\gamma, p_2) \leq S(p_2\mid \gamma)$$ (here
$\mathcal{W}_2$ denotes the classical 2-Wasserstein distance). We
shall
see in Example~\ref{rem:steinker} in the next section
that, outside a Gaussian context, $S(p_2\mid p_1)$ also bounds
1-Wasserstein distance, under some additional assumptions on
$p_1$.

\begin{example}[Comparison between Gaussians]
  \label{ex:comaoighe}
{For $i=1,2$} let $p_i$ be a centred Gaussian pdf
with covariance
  $\Sigma_i$. Then
  \begin{equation}\label{eq:boundGauss}
    S(p_2 \mid p_1)^2 \leq  \norm{ \Sigma_2 - \Sigma_1}_{\mathrm{HS}}.
  \end{equation}
 If the $\Sigma_i$ are of the form $\Sigma_i = \begin{pmatrix}
    1 & \rho_i  \\ \rho_i &  1
  \end{pmatrix}$ for some $\rho_i \in [0, 1]$, $i=1, 2$, then
 $$\norm{ \Sigma_2 -
   \Sigma_1}_{\mathrm{HS}} = \sqrt{2}|\rho_1-\rho_2|$$ which is
 exactly the value of the 2-Wasserstein distance in this case, see
 \cite[Theorem~2.4]{T11}.  {A legitimate question in this context is
   whether there is some optimization to be reaped from the freedom of
   choice in the Gaussian Stein kernels from \eqref{eq:stkernguau},
   and considering
   $\E_{p_j}\norm{\pmb\tau_{1, \delta} - \pmb\tau_{2,
      \delta}}_{\mathrm{HS}}$ with $j = 1, 2$ and optimizing
   over $\delta$.
   Explicit numeric computations with the various kernels  indicate
   that the answer is negative; they show that
   the smallest bound is attained at $\delta = 1$ and
   $\pmb \tau_{i, 1} = \Sigma_i$, $i=1, 2$. The covariance matrix is
   therefore, in this sense, the ``best'' Stein kernel for the
   Gaussian pdf.
   In the
   next example we will exhibit a situation where optimization is in
   fact possible. }
  \end{example}

{
\begin{example}[Student vs Gaussian] \label{ex:gauvsstudent}
  Let $p_1 = \gamma$ be the standard Gaussian pdf for which we fix
  $\pmb \tau_1 = \mathrm{I}_d$,  and $p_2{=t_k}$ the
  centred Student pdf with $k$ degrees of freedom and shape
  $\Sigma = \mathrm{I}_d$. {The} Stein kernels for this distribution {from}
  Example~\ref{sec:mult-stud-t}
  {provide a}
  variety of possible differences $   \pmb  \tau_1 - \pmb \tau_2
  $. For instance we get
  \begin{equation*}
\pmb     \tau_1(x) - \pmb \tau_2(x) =  \left( 1-
  \frac{x^Tx+k}{d+k-2}\right)\mathrm{I}_d \mbox{ and } \pmb\tau_1(x) -\pmb \tau_2(x) =   - \frac{1}{k-1} \left( x x^T  +
      \mathrm{I}_d \right).
\end{equation*}
where the  first is obtained by
  taking  $\pmb  \tau_2$  as given in
  \eqref{eq:steikstud1}   and the second by taking $\pmb \tau_2$ as given in
\eqref{eq:steikstu2}.
These expressions  lead to
\begin{equation*}
  S_1(\gamma\mid t_k) = \frac{\sqrt{2d(d+2)}}{d+k-2} \mbox{ and }   S_2(\gamma \mid t_k) = \frac{\sqrt{d(5+d)}}{k-1},
\end{equation*}
respectively. When $d = 1$ we get
$S_1(\gamma \mid t_k) = S_2(\gamma \mid t_k) = \sqrt{6}/(k-1)$
which concurs with \cite[Section 6.3]{LRS16}. In dimension $d\ge2$,
both bounds tell a similar story, although for fixed $d$ the bound
$S_2(\gamma \mid t_k)$ is slightly smaller
for large $k$ than $S_1(\gamma \mid t_k)$. However dependence of
$S_1(\gamma \mid t_k)$ on the dimension is more informative, as
this last bound does not explode as $d$ goes to infinity. For the sake
of illustration, in the case $d=2$, we also computed the discrepancy
provided by comparing $\pmb \tau_1 = \mathrm{I}_d$ with the kernel
  $\pmb \tau_{\delta}$ given in \eqref{eq:steintsubet}. Then
  one can see that an optimal choice of parameters is
  $\delta = 1 -4 (2k-3)/(3k^2+4k-4)$ leading to a third discrepancy
  given by
\begin{equation*}
  S_3(\gamma \mid t_k) = \sqrt{\frac{40}{8 + k(3k-4)}}
\end{equation*}
which improves $S_1(\gamma \mid t_k)$ and  $S_2(\gamma \mid t_k)$
 for all $k>1$. Similar manipulations
will be possible in higher dimensions; this may be of independent
interest. Finally we remark that we have also computed
$ S_j(t_k \mid \gamma)$, $j = 1, 2, 3$ by exchanging the roles of
target and approximating distribution; however this results in bounds
which are hard to read and which we do not reproduce here.
\end{example}
}

{Thanks to the results in Sections \ref{sec:generalities} and
  \ref{sec:stein-kern-ellipt-1}, manipulation of Stein kernels
  can be surprisingly easy even for densities which {are not
    elliptical} and do not satisfy the usual regularity assumptions
  (such as being log-concave, stationary distributions of diffusions,
  having a spectral gap, etc.). This final example serves as an
  illustration.}

\begin{example}[Normal-gamma prior] \label{ex:normalgamma}
  {Now
  consider a Bayesian model where the prior distribution of $\theta$ is
  bivariate normal-gamma
  $\mathrm{NG}(\mu_0, \lambda_0, \alpha_0, \beta_0)$
  (see Example \ref{ex:norm-gamm}).
  Let ${{\xi}} =1/\sigma^2$ denote the precision and
  $\pmb\theta = (\nu, {\xi}) \in \R \times \R^+$ be the parameter of
  interest.  For $\theta_2 $ following the posterior distribution given
  $X_1 = x_1, X_2 = x_2, \dots ,X_n = x_n$ (independently sampled from a
  univariate Gaussian
  {$\mathcal{N}(\nu, \sigma^2)$} distribution)  it is known that}
  the {resulting}  posterior distribution {$P_2$} of $\pmb \theta_2$ is
  \begin{equation*}
\pmb \theta_2 \sim \mathrm{NG}\left(\frac{\lambda_0  \mu_0 + n \bar{x}}{n + \lambda_0}, n +
\lambda_0, \alpha_0 + \frac{n}{2}, \beta_0 + \frac{n}{2}s^2 + \frac{1}{2}\frac{n
  \lambda_0}{n+\lambda_0}(\bar  x - \mu_0)^2\right)
  \end{equation*}
  {where $\bar{x}= \frac1n \sum_{i=1}^n x_i$ denotes the sample mean and $s^2 = \frac1n \sum_{i=1}^n (x_i - \bar {x})^2$.}
  Similarly if $\pmb \theta$ has (improper) prior the uniform
  distribution then it is easy to see that the corresponding posterior
  distribution $\pmb \theta_1$ is
  \begin{equation*}
    \pmb \theta_1 \sim \mathrm{NG}\left(\bar x, n, \frac{n+1}{2},
    \frac{n}{2}s^2\right).
\end{equation*}
{Example \ref{ex:norm-gamm} gives a family of Stein kernels, each of which results in a}
bound on the Stein
discrepancy {which we do not detail here but but illustrate its
  behavior numerically for certain choices of parameter {in}
  Figure~\ref{fig:ploppp}.}  {The Stein discrepancy to the normal
  decreases superlinearly with increasing $n$, and with increasing
  distance $| \bar x - \mu_0|$.}

  \begin{figure}
    \centering
    \captionsetup{width=0.88\textwidth}
\includegraphics[width=0.70\textwidth]{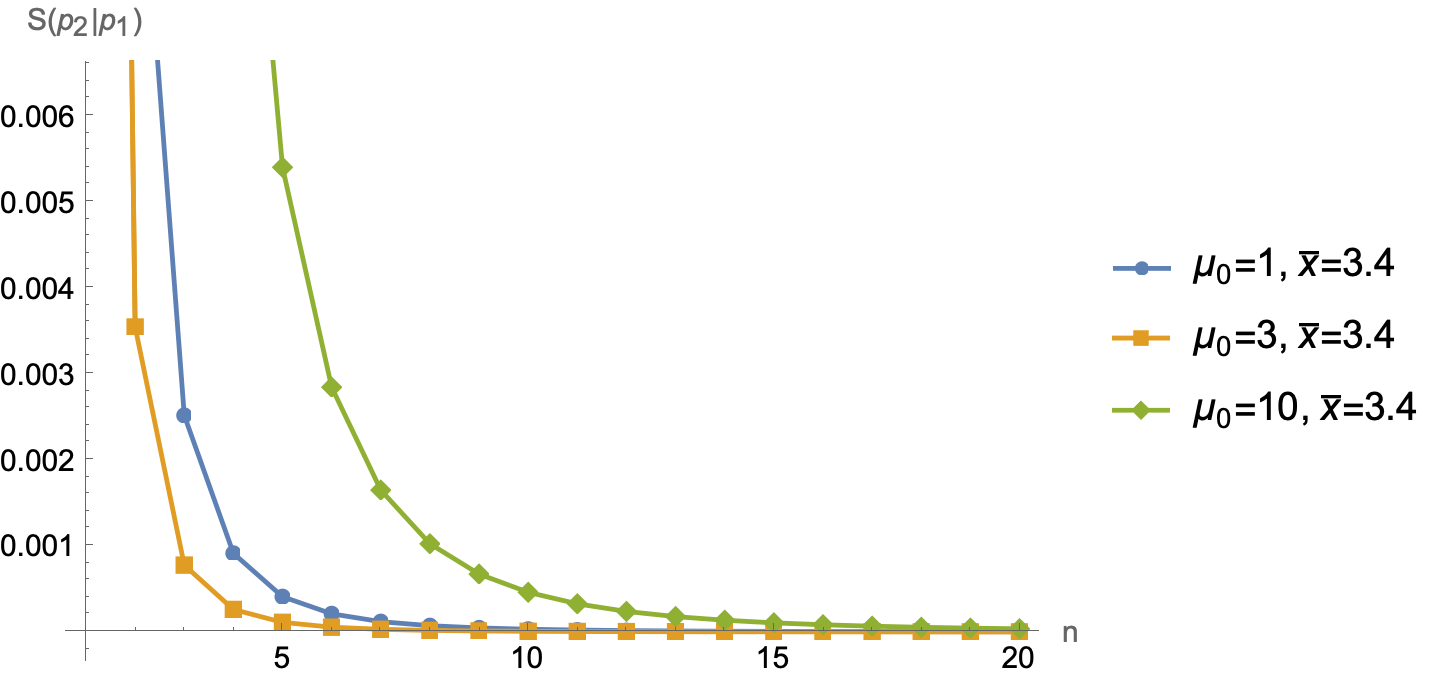}
\caption[test]{\small {The Stein
discrepancy  $S(p_2 \mid p_1)$  comparing the effect of  normal-gamma $\mathrm{NG}(\mu_0, \lambda_0, \alpha_0, \beta_0)$ prior vs uniform prior on the posterior in a normal $\mathcal{N}(\nu, \sigma^2)$ model} for $\mu_0 = 1$ {(blue circles)}, $ 3 $ {(orange squares)}, $ 10$ {(green diamonds)} and {the fixed arbitrary choice}
  $\lambda_0 = 1, \alpha_0 = 2, \beta_0 = 3$, and $s^2 = 1.2$. The Stein discrepancy to the normal
  decreases superlinearly with increasing $n$, and with increasing
  distance $| \bar x - \mu_0|$. 
  } \label{fig:ploppp}
  \end{figure}
\end{example}

\section{Comparing 
distributions in Wasserstein distance}
\label{sec:stein-discrepancies}

The main motivation behind Stein's method is the quantitative
comparison of probability distributions.  Consider {two probability
  measures $P_i$ on the same probability space, with Stein operator
  $\mathcal{A}_i$ and Stein class $\mathcal{F}(\mathcal{A}_i)$, for
  $i=1, 2$}.
Given some class $\mathcal{G}$ of
suitable test functions, Stein's method  {uses}
{$\sup_{g \in \mathcal{G}} \min \left\{ |\E_{P_2} \mathcal{A}_1
  g|, |\E_{P_1} \mathcal{A}_2 g| \right\}$} as a measure of
difference between $P_1$ and $P_2$.
Given this premise, there are a
variety of possible routes, including couplings, exchangeable pairs,
and comparison of operators.  {Here we explore the latter approach; we
focus on $\sup_{g \in \mathcal{G}}  |\E_{P_2}
\mathcal{A}_1|$, the other bound following by exchanging the roles of
$P_1$ and $P_2$.}

\subsection{
{C}omparing Stein operators {with nested support}}
\label{subsec:gencomp}

Recalling that
$ \E_{p_2} \mathcal{A}_2 g = 0$ for all
{$g \in \mathcal{F}(\mathcal{A}_2)$,}
 for any
$\mathcal{G} \subseteq \mathcal{F}(\mathcal{A}_2)$ it holds that
\begin{align*}
  \sup_{g \in \mathcal{G}}
  \bigl| \E_{P_2} \mathcal{A}_1 g \bigr|
  =
  \sup_{g \in \mathcal{G}}
  \bigl| \E_{P_2} ( \mathcal{A}_1 - \mathcal{A}_2) g \bigr|
\end{align*}
so that the difference $\mathcal{A}_1 - \mathcal{A}_2$ controls the
Stein discrepancy and, consequently, any metric controlled by the
latter.  There is much freedom in the choice of operators and
  classes for this purpose.
For transparency of exposition we
  {focus here on using operators \eqref{eq:4} to control  Wasserstein distance
  $\mathcal{W}_1(P_1, P_2)$, recalling \eqref{eq:diststeindis}}.
The main general result of the section {then} follows from previous
developments.

\begin{theorem}\label{theo:comap}
  Let $P_1$ and $P_2$ {be two probability measures} on $\R^d$ with
  respective pdfs $p_1$ and $p_2$ {having nested support $\Omega_{p_2} \subseteq \Omega_{p_1}$.} {Assume that $ \E_{p_1} | h| < \infty$ for every $h  \in \Lip({\Omega_{p_1}}, 1)$.}
  Associate to $p_i$, $i = 1, 2$
  the second order operator \eqref{eq:4}
  \begin{align*}
    \A_i g
    =
    \left\langle
      \mathcal{T}_{\dv , p_i}\mathbf{F}_i, \grad g
    \right\rangle + \left\langle
      \mathbf{F}_i, \grad^2 g \right\rangle_{\mathrm{HS}} ;
    \quad i = 1, 2
  \end{align*}
  for some
  matrix-valued functions
  $\mathbf{F}_i {\null \in \mathcal F_\loc(p_i)} $, $ i = 1, 2$;
  let $\mathcal{F}(\mathcal{A}_i)$, $i = 1, 2$ be the corresponding
  Stein classes. {Consider a collection
  $ \mathcal G_2(\mathbf F_1, \mathbf F_2) \subseteq
  \mathcal F(\mathcal A_1) $ 
  such that}
  \begin{enumerate}[(i)]
  \item {$ \E_{p_2} |\mathcal A_i g| < \infty $ for $ i = 1, 2 $,}
  \item {$ \int_{\Omega_{p_2}} \bigl|
    \dv \bigl( {\mathbf{F}_2^T} \grad g\, p_2 \bigr)
  \bigr| < \infty $}
  \end{enumerate}
  {for all $ g \in \mathcal G_2(\mathbf F_1, \mathbf F_2) $.
  Suppose}
  that the matrix-valued functions $\mathbf{F}_1$
  and $\mathbf{F}_2$ are such that,
  for every $h  \in \Lip({\Omega_{p_1}}, 1)$,
  we can find a solution
  $ g \in {\mathcal G_2(\mathbf F_1, \mathbf F_2)} $
  of the \emph{Stein equation}
  \begin{equation}
  \label{eq:strongsteinequ}
    \left\langle
      \mathcal{T}_{\dv, p_1}\mathbf{F}_1, \grad g
    \right\rangle
      +
    \left\langle
      \mathbf{F}_1, \grad^2 g \right\rangle_{\mathrm{HS}}
    =
    h - \E_{p_1} h
    \, .
  \end{equation}
  Then
  \begin{align*}
    \mathcal{W}_1(P_1, P_2)
    &\leq
    \sup_{g \in \mathcal{G}_2(\mathbf{F}_1,\mathbf{F}_2)}
    \Bigl| \E_{p_2} \bigl[
      \left\langle
        \mathcal{T}_{\dv, p_1} \mathbf{F}_1
          -
        \mathcal{T}_{\dv, p_2} \mathbf{F}_2, \grad g
      \right\rangle
        +
      \left\langle
        \mathbf{F}_1 - \mathbf{F}_2, \grad^2 g
      \right\rangle_{\mathrm{HS}}
    \bigr] \Bigr|
      \\ & \kern 4em \null +
    \kappa_2(\mathbf{F}_1, \mathbf{F}_2)
  \end{align*}
  where
  $ \kappa_2(\mathbf{F}_1, \mathbf{F}_2) =
  \sup_{g \in  \mathcal{G}_2(\mathbf{F}_1,\mathbf{F}_2)}
  {\left|  \int_{\Omega_{p_2}} \dv \bigl(
    \mathbf{F}_2^T \grad g \, p_2
  \bigr) \right|} $.
\end{theorem}

\begin{proof}
  We use \eqref{eq:9} to calculate
  \begin{align*}
    \E_{p_2} \Bigl[
     \left\langle
       \mathcal{T}_{\dv ,  p_2} \mathbf{F}_2, \grad g
     \right\rangle + \left\langle
       \mathbf{F}_2, \grad^2 g
     \right\rangle_{\mathrm{HS}}
    \Bigr]
    &=
    \E_{p_2} \bigl[
      \mathcal{T}_{\dv, p_2} \bigl( \mathbf{F}_2^T \grad g \bigr)
    \bigr]
    =
    \int_{\Omega_{p_2}} \dv \bigl( \mathbf{F}_2^T \grad g \, p_2 \bigr)
    \, .
  \end{align*}
  {If $ h \in \Lip(\Omega, 1) $ and
  $ g \in \mathcal{G}_2(\mathbf{F}_1, \mathbf{F}_2) $
  is a solution of \eqref{eq:strongsteinequ}, then}
  the assumptions of the theorem guarantee that
  {the  $\E_{p_2} h$ is well defined}
  {and}
  \begin{align*}
    {\E_{p_2} h - \E_{p_1} h}
    &=
    \E_{p_2} \Bigl[
      \left\langle
        \mathcal{T}_{\dv, p_1} \mathbf{F}_1
          -
        \mathcal{T}_{\dv, p_2} \mathbf{F}_2,
        \grad g
      \right\rangle
    \Bigr]
      +
    \E_{p_2} \bigl[ \left\langle
      \mathbf{F}_1 - \mathbf{F}_2, \grad^2 g
    \right\rangle_{\mathrm{HS}} \bigr]
      \\ & \qquad \qquad \null +
    \int_{\Omega_{p_2}} \dv \bigl( \mathbf{F}_2^T \grad g \, p_2 \bigr)
    \, .
  \end{align*}
  Taking suprema leads to the claims.
\end{proof}

\begin{rem}
  \leavevmode
  {
  \begin{itemize}
    \item Taking $g$ as in \eqref{eq:strongsteinequ}, it is not
    guaranteed that $g \in \mathcal{F} ( \A_2)$, so that
    $\E_{p_2} \mathcal{A}_2 g$ need not be zero.  It is assured that
    $g$ is in $\dom(\A_2)$, the set of functions for which
    $\A_2$ is defined.
    \item {As we shall see, in} many cases of interest, it will {be easy to verify that}
    the assumptions of Theorem \ref{theo:comap} are satisfied and,
    moreover, the solutions $g$ belong to $\mathcal{F}(\mathcal{A}_2)$,
    i.e.\ $\mathcal{G}_2(\mathbf{F}_1, \mathbf{F}_2) \subseteq
    \mathcal{F}(\mathcal{A}_2)$.
    Then $\kappa_2(\mathbf{F}_1, \mathbf{F}_2)=0$.
  \end{itemize}
}
\end{rem}

There is considerable
 flexibility in the bounds that can be obtained
  from Theorem~\ref{theo:comap};
   two
  particular cases   are illustrated in the next {e}xamples.

\begin{example}[{Wasserstein distance and Fisher information distance}]
\label{rem:steinfish}
  Suppose that the
  assumptions {of Theorem \ref{theo:comap}} are satisfied for some
  $\mathbf{F}_1 = \mathbf{F}_2 $ then
  \begin{align}
    \mathcal{W}_1(P_1, P_2)
    &\leq
    \sup_{g \in \mathcal{G}_2(\mathbf{F}_1, \mathbf{F}_1)}
    \Bigl| \E_{p_2} \Bigl[ \bigl\langle
      \left(
        \mathcal{T}_{\dv, p_1} - \mathcal{T}_{\dv, p_2}
      \right) \mathbf{F}_1, \grad g
    \bigr\rangle \Bigr] \Bigr|
      +
    \kappa_2(\mathbf{F}_1, \mathbf{F}_1)
    \, .
    \label{eq:difffff12}
  \end{align}
  Suppose 
  that $\mathbf{F}_1 = \mathrm{I}_d$ is allowed in
  \eqref{eq:difffff12}, then
  \begin{equation} \label{wassfish}  \mathcal{W}_1(P_1, P_2) \leq \sup_{g \in
    \mathcal{G}_2(\mathrm{I}_d, \mathrm{I}_d)}  {\bigl( \E_{p_2} \| \grad g \|^2 \bigr)^{1/2}} I(p_2 \mid p_1)
    + \kappa_2(\mathrm{I}_d, \mathrm{I}_d) \end{equation}  where
  $ I(p_2 \mid p_1) = (\E_{p_2} [ \norm{\grad \log p_1
      - \grad \log p_2}^2 ])^{1/2} $ is a \emph{Fisher
    information distance} which is well-known to metrize convergence
  in distribution, see e.g.\ \cite{Jo04}.  Thus bounds on
  the Fisher information (which scales well over convolutions, see
  e.g.\ in \cite{BaBaNa03,Jo04,nourdin2013integration}) translate
  immediately into bounds on Wasserstein distance.
  {We will study  inequality
  \eqref{eq:difffff12}  in more detail in {Section
  \ref{subsec:gencomp}}.}
\end{example}

\begin{example}[{Wasserstein distance  and Stein kernel discrepancy}]
\label{rem:steinker}
Suppose that the
assumptions {of Theorem \ref{theo:comap}} are satisfied for some $\mathbf{F}_1, \mathbf{F}_2$ such
that
$\mathcal{T}_{\dv , p_1}\mathbf{F}_1=
\mathcal{T}_{\dv , p_2}\mathbf{F}_2$ then
\begin{align}
 \mathcal{W}_1(P_1, P_2)  &  \leq   \sup_{g \in
    \mathcal{G}_2(\mathbf{F}_1,\mathbf{F}_2)} \left| \E_{p_2} \left[\left\langle
                       \mathbf{F}_1-\mathbf{F}_2,
                       \grad^2g\right\rangle_{\mathrm{HS}}
                              \right]\right| + \kappa_2(\mathbf{F}_1,
                              \mathbf{F}_2).
  \label{eq:difffff2}
  \end{align}
  If 
  $p_1$ and $p_2$ share a common mean
  and {if}  the Stein kernels $\mathbf{F}_1 = \pmb \tau_1$ and
  $\mathbf{F}_2 = \pmb \tau_2$ are allowed in
  \eqref{eq:difffff2}, then
  \begin{align}
 \mathcal{W}_1(P_1, P_2)  &  \leq   \sup_{g \in
    \mathcal{G}_2(\mathbf{F}_1,\mathbf{F}_2)} \left| \E_{p_2} \left[\left\langle
                      \pmb \tau_1-\pmb \tau_2,
                       \grad^2g\right\rangle_{\mathrm{HS}}
                              \right]\right| + \kappa_2(\mathbf{F}_1,
                              \mathbf{F}_2)  \label{eq:aliuehg}
  \end{align}
  which provides a {direct} connection between Wasserstein
  distance, Stein's method and the \emph{Stein kernel discrepancies}
  studied {in Section \ref{sec:stein-kern-discr}}.
 \end{example}

 Applicability of Theorem \ref{theo:comap} (or inequalities
 \eqref{eq:difffff12} and \eqref{eq:difffff2}) rests on a good
 understanding of the properties of solutions
 $g \in \mathcal{G}_2(\mathbf{F}_1, \mathbf{F}_2)$ of Stein equations,
 and more particularly on $\grad g$ and $\grad^2g$. Bounds on these
 quantities are called \emph{Stein
   factors}.
 For ease of use {here} we collect
  relevant estimates  {from the
   literature};
a new result, under the assumption of  Poincar\'e
 constant, will be provided in {Section
 \ref{sec:stein-factors-under}.}
\begin{example}[Gaussian Stein
  factors] \label{ex:steinfac} \label{lem:bouddd} Let $\Sigma$ be a
  $d \times d$ positive-definite matrix.   The Stein equation resulting from  \eqref{eq:stlemma}  for $\mathcal{N}(0, \Sigma)$
   is
\begin{align}\label{eq:stronggausseq}
\left\langle \Sigma,  \grad^2 g(x)  \right\rangle _{\mathrm{HS}}- \left\langle x, \grad g(x) \right\rangle = h(x) -
  \E [h(\Sigma^{1/2}Z) ]; \quad \quad  x \in \R^d,
\end{align}
with
$Z$ a standard normal random vector.  Letting
$Z_{x, t} = e^{-t} x + \sqrt{1-e^{-2t}}\Sigma^{1/2}Z$, a solution of
\eqref{eq:stronggausseq} is identified in \cite{MR1035659} as
$ g(x) = - \int_0^\infty \E[\bar{h}(Z_{x, t})] \,\dl t$
where $\bar{h} =h - E[\Sigma^{1/2}Z]$, see also
\cite{G91}.  
In
\cite{GoRi96, MR2573554} it is shown that if $h$ is $n$ times
differentiable then $g$ is also $n$ times differentiable, and
\begin{align*}
  \left| \frac{\partial^k g(x)}{ \prod_{j=1}^k \partial x_{i_j}} \right| \le
  \frac{1}{k}   \left| \frac{\partial^k h(x)}{ \prod_{j=1}^k \partial x_{i_j}} \right|
\end{align*}
for all $k= 1, \ldots, n$ and all $x \in \R^d$.  {Moreover,} if
$h \in \Lip({\R^d,} 1)$
then  $g \in \mathcal{F}(\mathcal{A}_1)$ and
\begin{align}\label{eq:mexkes}
  &     \sup_{x \in \R^d} {\norm{ \grad g(x)}}   \leq 1 \quad
     \mbox{ and } \quad \sup_{x \in \R^d} {\norm{ \grad^2g(x)
      }}_{\mathrm{HS}} \leq \sqrt{\frac{2}{\pi}}
      \norm{\Sigma^{-1/2}}_{\mathrm{op}}.
\end{align}
{Much more is known on the properties of these solutions, and we refer to
\cite{meckes2009stein} for an overview. }
\end{example}

\begin{example} [Log-concave Stein factors, \cite{mackey2016multivariate}]
  \label{lem:bouddd2}
  Let $P_1$ have pdf $p_1$ with full support
  $\Omega_{p_1} = \mathbb{R}^d$, and consider the  Stein
  equation
    \begin{equation}\label{eq:stronglogconc}
 \Delta g + \left\langle  \grad \log p_1, \grad g \right\rangle
   = h - \E_{p_1}h
 \end{equation}
with $h \in \Lip({\R^d,} 1)$.
{For a function $g\Colon \R^d \rightarrow \R$ introduce $M_j, j=2, 3$ as}
\[
  M_j(g) := \sup_{x, y \in \R^d, x \neq y} \frac{\| \grad^{k-1} g(x) - \grad^{k-1} g(y)\|_{\mathrm{op}}}{\| x - y\|}.
\]
Recall that a function $f \in \Cont^2(\R^d)$ is {\it $k$-strongly concave}
for $k>0$ if for all $x, y \in \R^d$,
$ y^T \grad^2 f(x) y \leq - k \| y\|^2.$ Suppose that
$\log p_1 \in \Cont^4(\R^d)$ is $k$-strongly concave with
$M_3(\log p) \leq L_3$. {Then} \eqref{eq:stronglogconc} has a solution
$g \in \mathcal{F}(\mathcal{A}_1)$ which satisfies
\begin{align*}
  &     \sup_{x \in \R^d} {\norm{ \grad g(x)}}   \leq \frac{2}{k}
     \mbox{ and }\sup_{x \in \R^d} {\norm{ \grad^2g(x)
      }}_{\mathrm{HS}} \leq  \frac{2L_3}{k^2} + \frac{1}{k}M_2(h).
  \end{align*}
     \end{example}

  \begin{example}[More general Stein factors,
    \cite{gorham2016measuring,fang2018multivariate}] \label{ex:mcoahsoin}
    Let $P_1$ with  support $\Omega_{p_1}= \R^d$ be invariant measure of
    an It\^o diffusion as in Example
    \ref{ex:steinfiffufi} and consider the strong Stein equation
    (using the notations from that example)
    \begin{align}\label{eq:gmame}
      \left\langle a(x) + c(x), \grad^2g(x)  \right\rangle_{\mathrm{HS}}
                         - \left\langle \mathcal{T}_{\dv , p_1}a(x) + f(x), \grad g(x)
                         \right\rangle = h(x) - \E_{p_1}h
    \end{align}
    with $h \in \Lip({\R^d,} 1)$. In \cite{gorham2016measuring}, it is assumed  that the transition semigroup of the diffusion
    satisfies a condition of \emph{Wasserstein decay rate} (see their
    Definition 4), while \cite{fang2018multivariate} make more
    analytical assumptions (see their Assumption 2.1); under these
    assumptions, Equation \eqref{eq:gmame} has a solution
    $g \in \mathcal{F}( \mathcal{A}_1)$ which is twice continuously
    differentiable and Stein factors are available (see \cite[Theorem
    5]{gorham2016measuring} and \cite[Theorem
    3.1]{fang2018multivariate}).
  \end{example}

{For the remainder of the section we focus on the case where}
  %
$p_1$ {is} the pdf of $\mathcal{N}(0, \Sigma_1)$
   with positive
   definite
 covariance $\Sigma_1$.  {Then the bounds on $\grad g$ and
   $\grad^2g$ can be read
   from Example \ref{ex:steinfac} and t}he following two observations follow directly.
{First, b}y \eqref{wassfish},
     if  $P_2$ has full support $\R^d$
  then $\kappa_2(\mathrm{I}_d, \mathrm{I}_d) = 0$ {and the  first bound in
   \eqref{eq:mexkes}} gives
  $$ \mathcal{W}_1(P_1, P_2) \leq I(p_2 \mid p_1).$$  {Hence, in
    particular, for $P_2$ with support $\R^d$ then any 
    {central limit theorem (CLT)}
    in Fisher information directly translates to one in Wasserstein
    distance}; {if the support is not $\R^d$ then some adaptations
    are necessary in order to incorporate the integration constants.}
  {Second, if $P_2$ is centred with a Stein kernel $ \pmb \tau_2$
    which is allowed as $\mathbf{F}_2$ in \eqref{eq:difffff2},}
  choosing $\pmb \tau_1 = \Sigma_1$ in \eqref{eq:aliuehg} the
  Cauchy--Schwarz inequality and the second bound in \eqref{eq:mexkes}
  give
\begin{align}
\label{eq:usususu}
  \mathcal{W}_1(p_1, p_2)
  \leq
  \sqrt{\frac{2}{\pi}} \,
  \norm{\Sigma_1^{-1/2}}_{\mathrm{op}}
  \E_{p_2} \norm{\pmb \tau_2 - {\Sigma_1}}_{\mathrm{HS}}
    +
  \kappa_2(\Sigma_1, \pmb \tau_2)
  \, .
\end{align}
Thus,
{any CLT in Stein kernel discrepancy leads to a CLT in
  Wasserstein distance. For instance, using {known properties of}
  Stein kernels, the following result is straightforward.}

\begin{example}[CLT in Wasserstein]
  Suppose that $p_2$ is the pdf of the random variable
  $W_n = (\sum_{i=1}^n X_i)/\sqrt n$ where the $X_i\sim p$ are i.i.d.\
  random vectors with mean 0, variance $\mathrm{I}_d$ and (common)
  Stein kernel $\pmb \tau$. Then, following \cite[Theorem
  3.2]{courtade2017existence}, one shows that
  $$\pmb\tau_2(x) = n^{-1}\sum_{i=1}^n \E \left[ \pmb
    \tau(X_i) \mid W_n =x \right]$$ is a Stein kernel for $p_2$
and $ \E_{p_2} \norm{
  \pmb \tau_2 - \mathrm{I}_d}_{\mathrm{HS}} \leq \frac{1}{\sqrt n}
\E_p \norm{\pmb \tau - \mathrm{I}_d}_{\mathrm{HS}}. $
If furthermore  $p$ has
  support $\R^d$ then $\kappa_2(\Sigma_1, \pmb \tau_2) = 0$ and
  \eqref{eq:usususu} yields
   \begin{equation*}
       \mathcal{W}_1(\gamma, p_2)  \leq  \frac{1}{\sqrt n}\sqrt{\frac{2}{\pi}}
\E_p \norm{\pmb \tau - \mathrm{I}_d}_{\mathrm{HS}}.
\end{equation*}
Under the same conditions as in \cite[Corollary
2.5]{courtade2017existence} 
this leads to a CLT in Wasserstein
distance with correct dependence on the dimension and on the sample
size $n$.
\end{example}

{The next example serves to  compare our approach with standard results.}

\begin{example}[Comparison {between} {Gaussians}]\label{ex:compafonormales}
  For $i=1,2$ let $p_i$ be a centred Gaussian pdf
  with covariance
  $\Sigma_i$. Inequality \eqref{eq:usususu} applies with
  $\kappa_2(\Sigma_1, \Sigma_2) = 0$ so that
  \begin{equation*}
    \mathcal{W}_1(p_1, p_2)
    \leq
    \sqrt{\frac{2}{\pi}} \min \left\{
      \norm{\Sigma_1^{-1/2}}_{\mathrm{op}}, \>
      \norm{\Sigma_2^{-1/2}}_{\mathrm{op}}
    \right\}
    \norm{\Sigma_1 -\Sigma_2}_{\mathrm{HS}}
    \, .
  \end{equation*}
  A general bound for Wasserstein distance between Gaussians is also
  given in \cite[Lemma 2.4]{chafma2010}. For the sake of {illustration,}
  take $d=2$ and $\Sigma_i =
  \begin{pmatrix}
    1 & \rho_i \\ \rho_i & 1
  \end{pmatrix}
  $
  for some $\rho_i \in [0, 1]$. Then our bound reads
  \begin{equation}
  \label{eq:gmsb}
    \mathcal{W}_1(p_1, p_2)
    \leq
    \frac{2}{\sqrt{\pi}} \min \left\{
      \frac{1}{\sqrt{1+\rho_1}}, \frac{1}{\sqrt{1+\rho_2}}
    \right\} |\rho_2 - \rho_1|
  \end{equation}
  whereas that from  \cite[Lemma 2.4]{chafma2010} is
  \begin{equation}\label{eq:chafff}
    \mathcal{W}_1(p_1, p_2) \leq \sqrt{
      4 - 2 \sqrt{1-\rho_1}\sqrt{1-\rho_2} -
      2 \sqrt{1+\rho_1}\sqrt{1+\rho_2}
    } \, .
  \end{equation}
  Finally,
  an efficient coupling can
  be constructed directly (simply write $X_j = \Sigma_j^{1/2} N$ for
  $N$ standard normal),  proving that
  \begin{equation}
  \label{eq:refer}
    \mathcal{W}_1(p_1, p_2)
    \leq
    \sqrt{\frac{{\pi}}{2}} \> |\rho_1 - \rho_2|
    \, .
  \end{equation}
  We provide illustrations of the comparison of these bounds in
  Figure \ref{fig:compachaff}. {In this scenario, the bound \eqref{eq:refer} is outperformed by  the bounds from \eqref{eq:gmsb} and \eqref{eq:chafff}, with our new bound \eqref{eq:gmsb} outperforming \eqref{eq:chafff} when $\rho_1 = 0.5$ and $\rho_2 $ is large.}
  \begin{figure}
    \centering
    \captionsetup{width=0.88\textwidth}
  \includegraphics[width=0.45\textwidth]{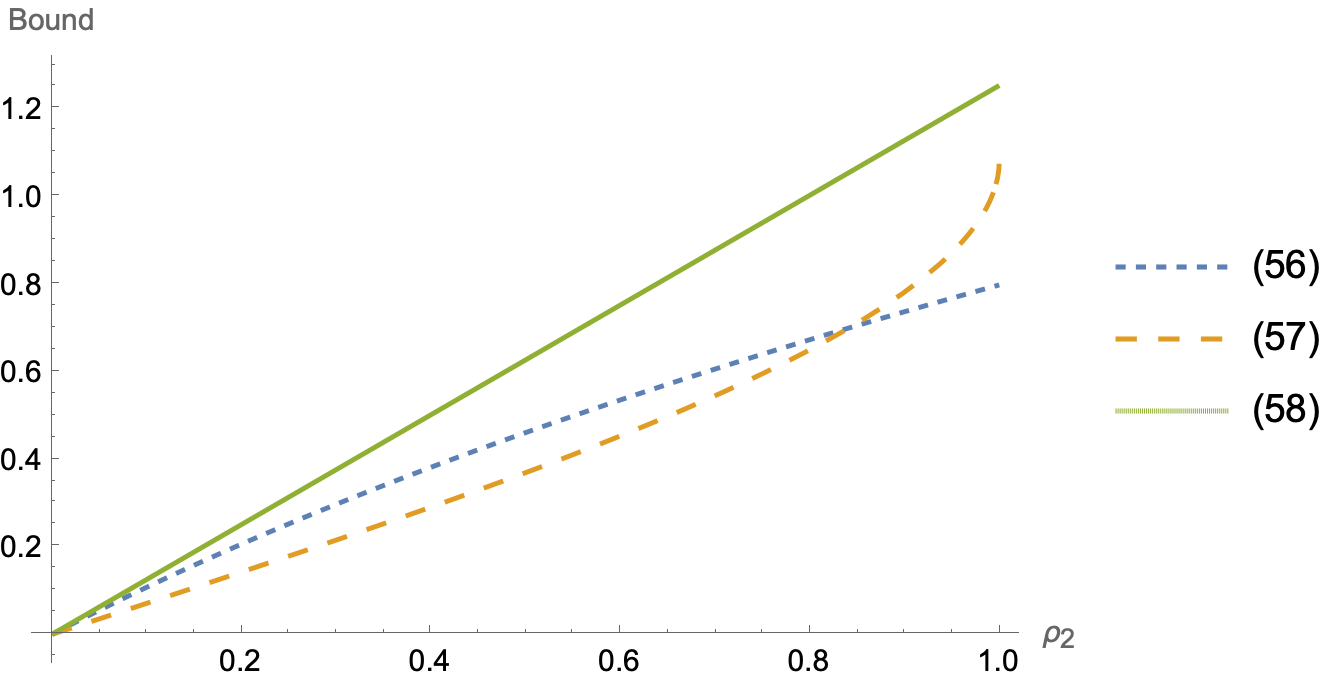} 
  \includegraphics[width=0.45\textwidth]{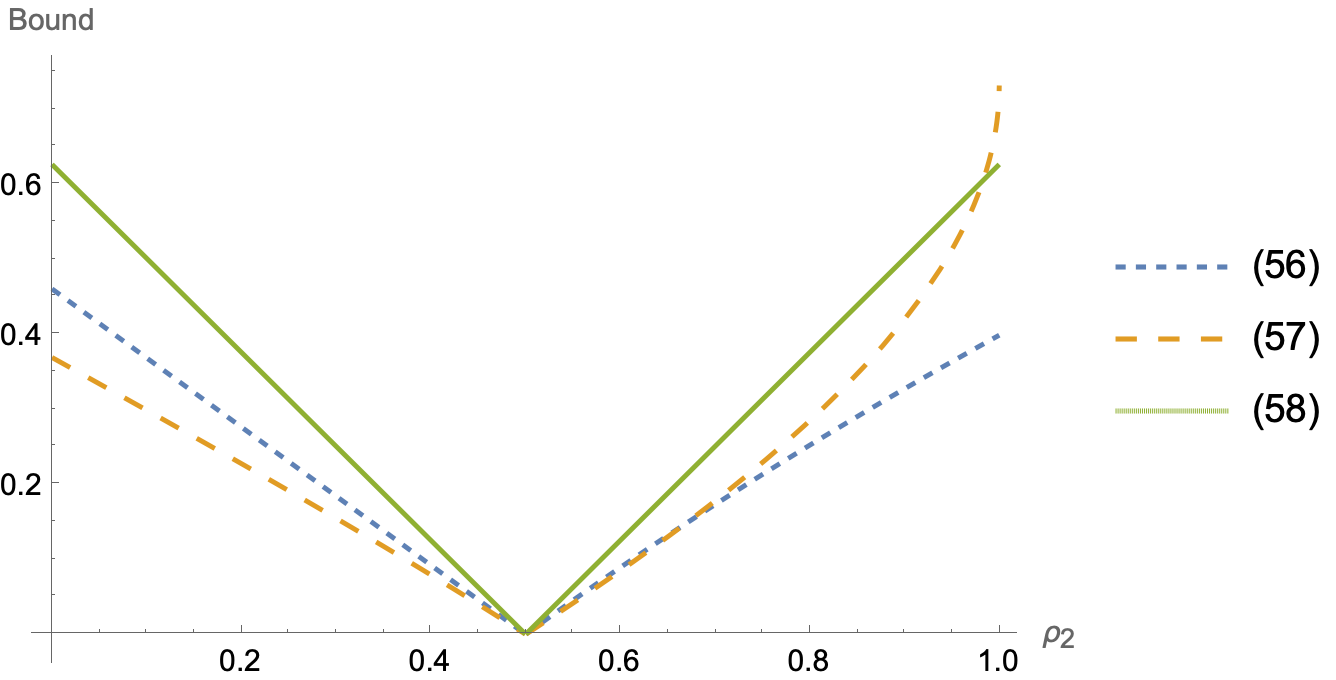}
  \caption[test]{\small  {Comparison between {Gaussians}}. We fix $d =
    2$ and,  $\Sigma_i = \begin{pmatrix}
    1 & \rho_i  \\ \rho_i &  1
  \end{pmatrix}$ for $i=1, 2$.
  Both plots report the bounds \eqref{eq:gmsb} {(blue dots)}, \eqref{eq:chafff} {(orange dashes)},
  and \eqref{eq:refer} {(a green line)} for $\rho_2 \in (0, 1)$. {For the} 
  left plot, 
  $\rho_1 = 0$; {for}  the right plot, 
  $\rho_1 = 0.5$.  {The new bound \eqref{eq:gmsb} outperforms both \eqref{eq:refer}  and  \eqref{eq:chafff} when $\rho_1 = 0.5$ and $\rho_2 $ is large. 
  }
  \label{fig:compachaff}}
  \end{figure}
 \end{example}



We now concentrate on \eqref{eq:difffff12} {in the special case $\mathbf{F}_2 = \mathbf{F}_1$ in Theorem \ref{theo:comap}.} 
Similar
arguments as {for Theorem \ref{theo:comap}}  
lead to the following
result.

\begin{prop} \label{cor:wass-dist-betw-1} Suppose that $P_1$ and $P_2$
  have {pdfs $p_1$ and $p_2$} with nested supports
{$\Omega_{p_2} \subseteq \Omega_{p_1}$}.  Instate the notations and
  assumptions from Theorem \ref{theo:comap}, but with
  $\mathbf{F}_2 = \mathbf{F}_1$.  Then, letting $\pi_0 = p_2/p_1$, it holds that
\begin{align} \label{bound1}
  \mathcal{W}_1(P_1, P_2) \leq      \sup_{g \in
   \mathcal{G}_2(\mathbf{F}_1, \mathbf{F}_1)} \left| \E_{p_1}   \langle \mathbf{F}_1\grad \pi_0,
      \grad g \rangle   \right| +\kappa_2(\mathbf{F}_1, \mathbf{F}_1).
\end{align}
In particular,
  \begin{align}
                                \label{prop:nesteddensitie}
                                \mathcal{W}_1({P_1, P_2}) &  \leq \sup_{g \in
    \mathcal{G}_2(\mathbf{F}_1, \mathbf{F}_1)}\sup_{x \in \R^d}\norm{ \grad g(x)
 }  \, \E_{p_1} \norm{
                                \mathbf{F}_1 \grad\pi_0} +
                                \kappa_2(\mathbf{F}_1, \mathbf{F}_{1}) \\
                      \label{eq:nestedforpoinca}
 \mathcal{W}_1({P_1, P_2})   &   \leq \sup_{g \in
    \mathcal{G}_2(\mathbf{F}_1, \mathbf{F}_1)} (\E_{p_1}
 \norm{\grad g}^2)^{1/2} \,\sqrt{\E_{p_1} \norm{\mathbf{F}_1
                               \grad\pi_0}^2}+\kappa_2(\mathbf{F}_1,
                             \mathbf{F}_{1}) .
  \end{align}
  If, furthermore, $\pmb \tau_1$ is a Stein kernel for $p_1$ then
  \begin{align} \label{bothbounds} \mathcal{W}_1 (P_1,P_2) \geq
    \norm{\E{_{p_1}} [ \pmb \tau_1 \grad \pi_0 ]}.
\end{align}
\end{prop}

\begin{proof}
Since $p_1$ and $p_2$
have nested supports, it
{follows}
{from \eqref{eq:42}} that
  \begin{align} \label{eq:nested} \mathcal{T}_{\dv , p_2} \mathbf{F} =
\mathcal{T}_{\dv , p_1} \mathbf{F} + \mathbf{F} \grad
\log(p_2/p_1)
\end{align}
{(over $\Omega_{p_2}$)} for any $\mathbf{F}$.
  {The bound {\eqref{bound1} now}  follows directly from Theorem
    \ref{theo:comap}}.
        {Inequalities \eqref{prop:nesteddensitie}
    and \eqref{eq:nestedforpoinca} follow from different applications
    of the Cauchy--Schwarz inequality}. {For  \eqref{prop:nesteddensitie}, one bounds
    $\sum_{i=1}^d | \partial_i g |$ by $\norm{ \grad g}$ and for \eqref{eq:nestedforpoinca}
    one uses the  Cauchy--Schwarz inequality inside the expectation.}

{{For \eqref{bothbounds},} let $e$ be a unit
  vector. Let $\nu_1 = \E_{p_1} \mathrm{Id}$ be the mean of
  $p_1$, recalling that $\mathrm{Id}$ is the identity function and that
  $\mathrm{{Id}}_e = \left\langle \mathrm{Id}, e
\right\rangle$   {denotes}
{the marginal projection} in direction $e$. Noting that  $\mathrm{Id}_e$ is 1-Lipschitz and
  using the nested structure we have {with}
  $\nu_e = \E_{{p_1}} [\mathrm{Id}_e]$:
\begin{align*}
  \mathcal{W}_1 ({P}_1,P_2)& \geq \left| \E_{{p_1}}  [ \mathrm{Id}_e]  -
                               \E_{{p_2}} [  \mathrm{Id}_e]  \right| =
                             \left| \nu_{e, 1} -  \E_{{p_1}}  [
                             \mathrm{Id}_e\pi_0] \right|\\
  &
                             = \left| \E_{{p_1}}  [ (\nu_{e, 1} -  \mathrm{Id}_e)\pi_0)]  \right|
                               =
                               \left| \E_{{p_1}} [ \langle (\nu -
    {\mathrm{Id}) \pi_0} ,e\rangle] \right|     = \left| \langle \E_{{p_1}} [  \pmb \tau_1 \grad
    \pi_0 ] , e \rangle\right|
\end{align*}
where the last identity follows from \eqref{eq:defstek2}.} Taking
$e = \E_{{p_1}} [ \pmb \tau_1 \grad {\pi_0} ] / \norm{\E _{{p_1}}[
  \pmb \tau_1 \grad \pi_0 ]}$ gives the claim
.
\end{proof}

For evaluating expectations it may be useful to note that the
expectations  $\E_{p_1}$  in
  Proposition \ref{cor:wass-dist-betw-1} can be expressed in terms of
  $\E_{p_2}$, because
  \begin{align*}
  \E_{p_1}\left[
    \langle \mathbf{F}_1\grad  \pi_0,
    \grad g \rangle \right]  =\E_{p_2}\left[
    \langle \mathbf{F}_1\grad \log \pi_0,
    \grad g \rangle \right].
\end{align*}

\begin{rem}
  \label{rem:boundpi}
  Let $\pmb \tau_1$ be a Stein kernel for $p_1$ and suppose, for
  simplicity, that
  $\mathcal{G}_2(\pmb \tau_1, \pmb \tau_1) \subseteq
  \mathcal{F}(\mathcal{A}_2)$ so that
  $\kappa_2(\pmb \tau_1, \pmb \tau_1) =0$. When $d=1$, in
  \cite[Equation (4.2)]{ley2015distances}, the following bound is
  provided:
\begin{align*}
|\E{_{p_1}} [\pmb \tau_1   \pi_0']| \leq \mathcal{W}_1({P}_1,{P}_2) \leq
\E{_{p_1}} [\pmb \tau_1  | \pi_0'|].
\end{align*}
{If $\mathbf{F} = \pmb \tau_1$ satisfies the assumptions of Theorem
    \ref{theo:comap}, then,}
 combining \eqref{bothbounds} and \eqref{prop:nesteddensitie},
    \begin{align}
\label{bothboundspr}     \norm{ \E{_{p_1}} [ \pmb \tau_1 \grad \pi_0
  ]}   & \leq   \mathcal{W}_1 (P_1,P_2)   \leq    C  \E_{p_1}\left[  \norm{
         \pmb \tau_1 \grad \pi_0 } \right]  . 
    \end{align}
    with
    $C =\sup_{g \in \mathcal{G}_2(\pmb \tau_1, \pmb \tau_1)}\sup_{x
      \in \R^d}\norm{ \grad g(x)}$, {yielding a similar bound as in
      the case $d=1$.} {In particular, if $p_1$ is the Gaussian
      distribution then $C = 1$ (recall Example \ref{lem:bouddd})
      which leads to the same bound as in the univariate case, whereas
      if $p_1$ is only assumed to be  $k$-strongly log-concave with
      $\log p_1 \in \Cont^4(\R^d)$ then
      the best available bound is, to the best of our knowledge, $C
      \leq 2/k$
      (recall Example \ref{lem:bouddd2}). Bounds in more general cases
      are also available (recall Example \ref{ex:mcoahsoin}). }
\end{rem}

\begin{example}[Azzalini--Dalla Valle skew-normal distributions vs
  multivariate normal]\label{ex:azzal}
The pdf of the centred Azzalini--{Dalla Valle} type r.v.\ $X \in \R^d$ is given by
$$p_\alpha(x)=2\omega_d(x;\Sigma) \Phi(\alpha^T x),$$
where $\omega_d(x;\Sigma)$ is the pdf of {the $d$-dimensional normal distribution} $\mathcal N(0,\Sigma)$,
$\Phi$ the c.d.f.\ of the standard normal on $\R$, and
$\alpha \in \R^d$ is a skew{ness}  parameter, {see
  \cite{azzalini1996multivariate}}.  {In \cite{LRS16}, an exact
  expression for the Wasserstein distance between $p_2 = p_{\alpha}$
  and $p_1 = \omega_d$ is given for $d=1$; here we
  extend this result to general $d$.}  
We aim to apply Proposition \ref{cor:wass-dist-betw-1} with
$\mathbf{F}_1 = \Sigma$ the covariance (a Stein kernel for the
Gaussian) and $\pi_0 (x) = 2 \Phi (\alpha^T x)$. Solutions of the
Gaussian Stein equations exist, recall {Example} \ref{lem:bouddd}; these
solutions also belong to $\mathcal{F}(\mathcal{A}_2)$ because $\pi_0$
is bounded, hence $\kappa_2(\Sigma, \Sigma) = 0$.  {It is easy to see that
 $2 \Phi (\alpha^T x)  - \E_{\mathcal N(0,\Sigma)} [2 \Phi (\alpha^T x) ] \in W_0^{1,2}(\mathcal N(0,\Sigma))$.
} Then
\eqref{bothboundspr} can be applied with $C=\sup_{g \in
    \mathcal{G}_2(\Sigma, \Sigma)}\sup_{x \in \R^d}\norm{ \grad g(x)
 } =1$ (recall {Example} \ref{lem:bouddd}). Since
$\grad \pi_0 (x) = 2\alpha \gamma (\alpha^T x)$  with $\gamma$ the
one-dimensional standard normal pdf,
the upper and lower bounds
coincide in \eqref{bothboundspr}, leading to
\begin{align*}
  \mathcal{W}_1(p_{\alpha}, \omega_d) &=   2   \norm{
  \Sigma \alpha }   \E_{\omega_d} \left[  \gamma (\alpha^T  {{\rm Id}})
   \right] \\
 &= 2   \norm{\Sigma \alpha }   (2\pi)^{-(d+1)/2}
                          \frac{1}{\sqrt{\mathrm{det}(\Sigma)}}\int  \exp\left( -
                          \frac{1}{2}\left(   x^T  (\alpha \alpha^T+\Sigma^{-1}) x \right)
                           \right)\,\dl x
                           \\
& =\frac{2   \norm{
  \Sigma \alpha }}{\sqrt{2\pi(1+\alpha^{T} \Sigma \alpha)}}.
\end{align*}
where the second equation follows from the matrix determinant lemma.
In particular we recover the univariate result {from \cite{LRS16}}
when $d = 1$.
\end{example}

{The nestedness assumption is naturally satisfied in a Bayesian
  setting when comparing the effect of a specific prior against the
  uniform prior on the posterior distribution. We use this context to
  illustrate our bounds on two last examples.}

\begin{example}[{Normal versus uniform prior in normal model}]\label{ex:normvsunifom}
  Consider a normal ${\mathcal{N}}(\theta , \Sigma)$ model with mean
  $\theta \in \R^d$ and positive definite covariance matrix $\Sigma$.
  Let
$ P_1$, with pdf $p_1$, denote the posterior distribution
of $\theta$ with uniform prior and $P_2$, with pf $p_2$,
the posterior distribution with {prior ${\mathcal{N}}(\nu,
  \Sigma_2)$;}
$\Sigma_2$ is assumed positive definite.  Suppose that a sample
$(x_1,\ldots,x_n)$ with $x_i \in \R^d$ for all $i$ is
observed.  
Inequality \eqref{bothboundspr} applies with $P_1 = \mathcal N(\bar x,
n^{-1}\Sigma)$, with $\bar x = \frac{1}{n}\sum_{i=1}^n x_i$, and $P_2
= \mathcal N(\tilde \mu_n, \tilde \Sigma_n)$ with
\begin{align*}
& \tilde \mu_n
= \nu+n \tilde \Sigma_n \Sigma^{-1} (\bar x - \nu) \quad  \mbox{ and } \quad
 \tilde \Sigma_n = (\Sigma_2^{-1}+n\Sigma^{-1})^{-1} \notag
\end{align*}
and $C_2(\pmb \tau_1, p_1) = 1$. Since
$p_2(\theta) \propto p_1(\theta) \exp \left( (\theta-\nu)^T \Sigma_2^{-1}
(\theta-\nu) \right)$, it holds that
$\grad \log \pi_0(\theta) =  -\Sigma_2^{-1} (\theta-\nu).$
Since $\pmb \tau_1 = n^{-1}\Sigma$ is a Stein kernel for $p_1$,
inequality  \eqref{bothboundspr} (re-expressed in terms of $p_2$) becomes
\begin{align*}
\norm{n^{-1}\Sigma \Sigma_2^{-1}\E_{p_2}  \left[{\rm Id} - \nu \right]} \leq \mathcal{W}_1(P_1, P_2) \leq \E_{p_2} \left[ \norm{
  n^{-1}\Sigma  \Sigma_2^{-1}  ({\rm Id}-\nu)} \right].
\end{align*}
{These expectations can be evaluated and after some calculations we
  find}
  \begin{align*}
    \norm{  (\mathrm{I}_d+n\Sigma^{-1}\Sigma_2)^{-1}
                                                                   (\bar
    x - \nu) } &  \leq \mathcal{W}_1(P_1, P_2)   \leq \norm{  (\mathrm{I}_d+n\Sigma^{-1}\Sigma_2)^{-1}
                                                                   (\bar
                       x - \nu) }   \\
    & \quad + n^{-1}\norm{\Sigma
  \Sigma_2^{-1} (\Sigma_2^{-1}+n\Sigma^{-1})^{-1/2}}_{\mathrm{op}}  \frac{\sqrt
    2\Gamma(d/2+1/2)}{\Gamma(d/2)}.
  \end{align*}
{The} multivariate bound indeed simplifies to
the univariate {bound from \cite{ley2015distances}} when $d = 1$.

\end{example}

\begin{example}[Bayesian logistic regression with Gaussian prior] \label{ex:bayeslog}
  We follow \cite[Example 1]{gorham2016measuring}.  Consider the
  log pdf of a Bayesian logistic regression posterior based on a
  dataset of $L$ observations $\mathbf{x}_{\ell} = (v_{\ell}, y_{\ell})$,
  $\ell = 1, \ldots, L$, {with $v_\ell \in \R^d$ a vector of covariates and $y_\ell \in \{0,1\}$} and a $d-$dimensional $\mathcal{N}(\nu, \Sigma)$ prior {on the parameter $\beta \in \R^d$ of the logistic regression}:
\begin{equation*}
\log p_2(\beta) =  \kappa(\mathbf{x}) - \frac{1}{2} \norm{
  \Sigma^{-1/2}(\beta
 - \nu) }^2 - \sum_{\ell = 1}^L \log \left( 1 + \mathrm{exp}\left(
   -y_{\ell} \left\langle v_{\ell}, \beta \right\rangle \right)  \right)
\end{equation*}
where $ \kappa(\mathbf{x})$ is an irrelevant normalizing constant, the
first summand is the multivariate Gaussian prior {on $\beta$} and the second term
is the logistic regression likelihood. Treating the Gaussian prior as
the target $p_1$ (and therefore the logistic likelihood as $\pi_0$),
it follows from   \eqref{bothboundspr} {and  $\E_{p_1} ( g \grad \pi_0) = \E_{p_2} ( g \grad \log \pi_0)$,} with  $C_2(\Sigma, p_1) \le
1$,  $\pmb \tau_1 = \Sigma$ and
\begin{align*}
 \grad \log \pi_0 =  \sum_{\ell = 1}^L \frac{-y_{\ell}
}{1+\mathrm{exp}\left(
   -y_{\ell} \left\langle v_{\ell}, \beta \right\rangle  \right) } v_{\ell} {\mathrm{exp}\left(
   -y_{\ell} \left\langle v_{\ell}, \beta \right\rangle  \right) } ,
\end{align*}
that \begin{align*}
       \norm{  \E_{p_2} \left[ \sum_{\ell = 1}^L \frac{-y_{\ell}   \Sigma v_{\ell} {\, \mathrm{exp}\left(
   -y_{\ell} \left\langle v_{\ell}, \beta \right\rangle  \right) } }{1+\mathrm{exp}\left(
   -y_{\ell} \left\langle v_{\ell}, \beta \right\rangle \right)}
       \right]  }&  \leq  \mathcal{W}_1(P_1, P_2) \\
        & \leq   \E_{p_2} \left[ \norm{ \sum_{\ell = 1}^L \frac{-y_{\ell}   \Sigma v_{\ell} {\, \mathrm{exp}\left(
   -y_{\ell} \left\langle v_{\ell}, \beta \right\rangle  \right) }}{1+\mathrm{exp}\left(
   -y_{\ell} \left\langle v_{\ell}, \beta \right\rangle \right)} }
  \right]
     \end{align*}
     (expectation is over $\beta \sim p_2$).  Similarly the roles can
     be reversed, with $p_1$ now the logistic regression likelihood
     and $\pi_0$ the Gaussian prior; here a Stein kernel is available
     as well, see \cite[Example 1]{gorham2016measuring}.
\end{example}

\subsection{{Weak Stein equations and a bound under a Poincar\'e condition}}
\label{sec:stein-factors-under}

Unlike in the 1-dimensional case, bounds such as those in Examples
\ref{lem:bouddd}, \ref{lem:bouddd2} or \ref{ex:mcoahsoin} are not
available for general target $p_1$, and even existence, let alone
regularity, of a solution of the Stein equation is {often} out of
reach. One way to bypass the necessity for solving Stein equations is
to work directly with Fisher information distance $I(p_2 \mid p_1)$
or Stein kernel discrepancies $S(p_2 \mid p_1)$ {and} then
use non-Stein's method related connections with classical
discrepancies such as Total Variation, 1-Wasserstein or
2-Wasserstein. This approach was used e.g.\ in
\cite{nourdin2013integration,ledoux2015stein,
  courtade2017existence,1i2018stein}. {We conclude the paper by
  introducing an alternative route via the notion of \emph{weak Stein
    equations} (and corresponding \emph{weak Stein factors}), as
  follows.}

  {{In this subsection we i}nstate the notation and assumptions of
Proposition~\ref{cor:wass-dist-betw-1}.}
{In the proof of Theorem~\ref{theo:comap}, the Stein equation \eqref{eq:strongsteinequ} is only used in a weak form, namely
\begin{equation}
\label{eq:weaksteinequ}
  {\E_{p_2} \bigl[ \mathcal{A}_{p_1} g \bigr]}
  =
  \E_{p_2} \left[
    \left\langle
      \mathcal{T}_{\dv, p_1} \mathbf{F}_1, \grad g
    \right\rangle + \left\langle
       \mathbf{F}_1, \grad^2 g
    \right\rangle_{\mathrm{HS}}
  \right]
  =
  \E_{p_2} \bigl[ h - \E_{p_1} h \bigr]
  \, .
\end{equation}
{
Now recall that $ p_2 = \pi_0 p_1 $ and that
the Stein operator has actually been obtained by
\eqref{eq:4}, leading to
\[
  \E_{p_2} \bigl[ \mathcal{A}_{p_1} g \bigr]
  =
  \E_{p_2} \bigl[
    \mathcal{T}_{\dv, p_1} \bigl( \mathbf{F}_1^T \grad g \bigr)
  \bigr]
  =
  \E_{p_1} \bigl[
    \mathcal{T}_{\dv, p_1} \bigl( \mathbf{F}_1^T \grad g \bigr) \pi_0
  \bigr]
  \, .
\]
Applying \eqref{eq:37f}, we find that
\[
  \E_{p_2} \bigl[ \mathcal{A}_{p_1} g \bigr]
  =
  - \E_{p_1} \bigl[ \bigl\langle
    \mathbf{F}_1^T \grad g, \grad \pi_0
  \bigr\rangle \bigr]
  =
  - \E_{p_1} \bigl[ \bigl\langle
    \mathbf{F}_1 \grad \pi_0, \grad g
  \bigr\rangle \bigr]
  \, .
\]
Thus \eqref{eq:weaksteinequ} is equivalent to
\begin{equation}
\label{eq:weaksteinequ2}
  - \E_{p_1} \bigl[ \bigl\langle
    \mathbf{F}_1 \grad \pi_0, \grad g
  \bigr\rangle \bigr]
  =
  \E_{p_1} \bigl[ (h - \E_{p_1} h) \pi_0 \bigr]
  \, .
\end{equation}}%
{Classical results from functional analysis provide handles on
  expressions such as \eqref{eq:weaksteinequ2}. {We follow \cite{courtade2017existence} and
{introduce  $W^{1,2}(p)$, the natural (weighted) Sobolev space of
{weakly differentiable}
functions $ u \Colon \Omega \rightarrow \R $
{with finite}
(squared) Sobolev norm
\begin{align}\label{def:W.norm.ker}
 \| u \|_{W^{1,2}(p)}^2
 :=
 \| u \|_{L^2(p)}^2 + \| {\grad u} \|_{L^2(p)}^2
 \, ,
\end{align}
  see also \cite{fathi2022relaxing}.
  {\emph{In the rest of the paper, we shall assume that $ p $ is
  continuous on $ \Omega_p $.}}
  Then $W^{1,2}(p)$ is complete under $ \| \cdot \|_{W^{1,2}(p)} $.
  We also introduce the space  $W^{1,2}_0(p)$
  of all functions $ u \in  W^{1,2}(p) $ such that
  $ \int_{\Omega} \partial_e u = 0 $ for all $ e \in S^{d-1} $.}}

{Next,
{motivated by \eqref{eq:weaksteinequ2}},
following intuition from \cite{courtade2017existence},
under
the condition that $P$ admits a Poincar\'e constant, in the sense of definition \ref{def:poincareconstant},
the following holds.}

 \begin{prop}[Weak Stein equation and factor]
\label{prop:l2bound}
Let $P_1$ be as in Theorem \ref{theo:comap}. {Suppose furthermore that} $P_1$
has finite variance and Poincar\'e constant $C_{P_1}$. 
Let  $h \in \Lip({\Omega_{p_1},} 1)$.
Then there exists a
{function $ g \in W^{1, 2}_0(p_1)$, which solves}
\begin{equation}
\label{eq:wkstek1}
  \E_{p_1} \bigl[
    \left\langle  {\grad} v , \grad g \right\rangle
  \bigr]
  =
  \E_{p_1} \bigl[ (h - \E_{p_1} (h)) v \bigr]
\end{equation}
{for all $ v \in W^{1,2}_0(p_1) $ and is}
such that
\begin{equation}
\label{eq:stekfakwkstek1}
  \sqrt{\E_{p_1} \norm{\grad g}^2 } \leq C_{P_1} \, .
\end{equation}
\end{prop}

\begin{proof}
Let $h \in \Lip({\Omega_{p_1},} 1)$.
The set  $\mathcal{H}_1 = \{ v \in L^2(P_1) \sth \E_{p_1} v = 0\} $ with
inner product $\langle u,v \rangle_{p_1} = \E_{p_1}[uv]$ is a Hilbert
space and  $\bar h = h - \E_{p_1}h \in \mathcal{H}_1$ as $P_1$
admits a variance. Moreover,
$$\norm{h}_{p_1} = \sqrt{ \E_{p_1} h^2 } \leq \sqrt{ C_{P_1}  \E_{p_1}
  \norm{\grad h}^2 } \leq \sqrt{ C_{P_1}}.$$
In particular we have that $ W^{1,2}_0(p_1) \subseteq \mathcal{H}_1.$
We endow the space $ W^{1,2}_0(p_1)$ with the symmetric bilinear form
$\mathcal{E}_{p_1}(u,v) = \E_{p_1} [\left\langle  \grad u,\grad v  \right\rangle].$
This bilinear form is coercive as due to the Poincar\'{e} inequality
\begin{equation} \label{eq:coercive} \mathcal{E}_{p_1}(u,u) = \E_{p_1}
[\left\langle \grad u,\grad u \right\rangle] \geq \frac{1}{C_{P_1}}
\E_{p_1} [u^2]= \frac{1}{C_{P_1}} \norm{ u }_{p_1}^2 .
\end{equation}
{From $\eqref{eq:coercive}$ it follows that}
{
$ \langle \cdot, \cdot \rangle_{\mathcal E_p} $ is also an inner product {on $ W^{1,2}_0(p) $}
and
$ \| \cdot \|_{\mathcal E_p} $ is a norm which is uniformly equivalent to
$ \| \cdot \|_{W^{1,2}(p)} $. Therefore, $ W^{1,2}_0(p) $ is complete
under $ \| \cdot \|_{\mathcal E_p} $, so that
$ \langle u, v \rangle_{\mathcal E_p} $ makes it a Hilbert space.

Now let $ \phi_{p_1}(v) := \E_{p_1} \bigl[ (h - \E_{p_1} h) v \bigr] $.
By the Cauchy--Schwarz {inequality} and {the} Poincar\'e inequality
and using 
that $ h \in \Lip({\Omega_{p_1},} 1)$ 
we can {bound}
\[
 |\phi_{p_1}(v)|
 \le
 \| h - \E_{p_1} h \|_{L^2(p_1)} \| v \|_{L^2(p_1)}
 \le
 C_{P_1} \| h \|_{\mathcal E_{p_1}} \| v \|_{\mathcal E_{p_1}}
 \le
 C_{P_1} \| v \|_{\mathcal E_{p_1}}
 \, .
\]
{Hence, b}y the Riesz representation theorem, there exists $ g \in W^{1,2}_0(p_1) $, such that
$ \phi_{p_1}(v) = \langle v, g \rangle_{\mathcal E_{p_1}} $ for all $ v \in W^{1,2}_0(p_1) $, {and} 
$ \| g \|_{\mathcal E_{p_1}} \le C_{P_1} $.
This completes the proof.}

\end{proof}

Obviously if $g$ is a solution of the strong Stein equation satisfying
all the assumptions from Proposition \ref{cor:wass-dist-betw-1}, then
it will also satisfy \eqref{eq:wkstek1}. However, in Corollary
\ref{cor:referee} we need not solve the equation, nor require any form
of regularity. Equations {\eqref{eq:weaksteinequ2} and
  \eqref{eq:wkstek1} are} akin to \emph{``weak Stein equations''} (and
thus would encourage us to consider \emph{``weak Stein operators''})
and inequality \eqref{eq:stekfakwkstek1} is a form of \emph{``weak
  Stein factor''}. This approach seems promising for tackling Stein's
method in more general multivariate settings, without needing to solve
Stein equations explicitly. {For instance, the next result follows easily.}
                                                             %
\begin{cor} \label{cor:referee} Let $P_1, P_2$ be as in Theorem
  \ref{theo:comap} and suppose furthermore that $P_1$ satisfies the
  conditions of Proposition \ref{prop:l2bound}. Also assume that
  $p_2 = \pi_0 \, p_1$ with
  $\pi_0 - \E_{p_1}\pi_0 \in W^{1,2}_0(p_1)$. 
Then
\begin{equation}
    \label{eq:bouindwaspoinca}
    \mathcal{W}_1(P_1, P_2) \leq C_{P_1} \sqrt{\E_{p_1}
       \norm{ \grad \pi_0}^2 }.
  \end{equation}
\end{cor}

\begin{proof}
  Let $h \in \Lip({\Omega_{p_1},} 1)$.
  From
  Proposition \ref{prop:l2bound} applied with $v = \pi_0$, it follows
  that there
  exists $g \in W^{1,2}_0(p_1)$ which is solution of the weak
  Stein equation
  $$
    \E_{p_2}[h] - \E_{p_1}[h]
    =
    \E_{p_1} \bigl[ \langle \grad  \pi_0 , \grad g  \rangle \bigr]
    \, .
  $$
  Since, furthermore, $g$ satisfies  \eqref{eq:stekfakwkstek1} which
  does not depend on $h$, the
  claim follows by the Cauchy--Schwarz inequality.
\end{proof}

There is a {large} literature regarding Poincar\'e inequalities and
their optimal associated Poincar\'e constant. For example, when $P$
has $k$-log-concave pdf $p$ such that $\log p \in \Cont^4(\R^d)$ then
$C_P \leq \, 2/k$; for $P$ the uniform distribution on
$[0, 1]^2$ it is known that $C_P = 2/\pi^2$ is an optimal
Poincar\'e constant, see \cite{payne1960optimal}. These bounds allow
to obtain bounds on 1-Wasserstein distance via Stein's method even in
cases where Stein factors are not available.}

\begin{example}[Comparing copulas]\label{ex:comparingcopulas}
  Let $V = (V_1,V_2)$ be a 2-dimensional random vector  such that the
  marginals $V_1$ and $V_2$ have a uniform distribution on {$(0,1)$}.
  {This distribution is described by the copula of $V$} defined
  as
  $C(x_1,x_2) = \P[V_1\leq x_1, V_2 \leq x_2],$
  $(x_1,x_2)\in {(}0,1{)}^2.$
  We
  want to bound the Wasserstein distance between $P^V$, the law of $V$
  and $P^{U}$ the 
  law of
  $U = (U_1, U_2)$, {where} $U_1$ and $U_2$
  are uniform {on $(0,1)$} and independent. First we  recall that an
  optimal Poincar\'e constant for the uniform distribution on $(0,1)^2$ is
  $C_p = 2/\pi^2 $.
Assume
that $V$ has a {pdf}
$c = \partial^2_{x_1x_2} C$. Then the densities of $V$ and $U$ have
nested supports, with $\pi_0 = c$.   {If $c  - \E_{U(0,1)^2} c \in
  W_0^{1,2}(U(0,1)^2)$}
then, a simple application of
\eqref{eq:bouindwaspoinca} with $C_P = 2/\pi^2$ yields
$$
  \mathcal{W}_1({P^V, P^U} ) \leq  \frac{2}{\pi^2} \sqrt{ \E_{ U(0,1)^2}  \norm{\grad c}^2}.$$
In some cases, one can compute the gradient of $c$ in a closed
form. For instance,  the Ali--Mikhail--Haq
copula \cite{ali1978class} has pdf
$$ c(x_1, x_2) = \frac{(1-\theta) \{ 1  - \theta(1-x_1)(1 - x_2)\} + 2 \theta x_1 x_2}{\{  1  - \theta(1-x_1)(1 - x_2)\}^3}.$$
{Then  $\E_{ U(0,1)^2}[ c ] = 1$ and
\begin{align*}
\E_{ U(0,1)^2} \norm{  \grad c}^2
 &= \frac{2}{105} \frac{\theta}{(1-\theta)^2} \left( \theta(2 -
    \theta)(52(1-\theta)+17\theta^2) - 36 \log(1-\theta) \right)
    \end{align*}
so that  $c  - \E_{U(0,1)^2} c \in W_0^{1,2}(U(0,1)^2)$.}
Here $\theta \in (-1,1)$ is a measure of association between the two components $V_1$ and $V_2$ of the vector $(V_1,V_2)$ with uniform marginals each. If $\theta=0$ then the uniform copula $(x_1,x_2) \mapsto x_1x_2$ is recovered. Using
Corollary \ref{cor:referee}
 we can assess the Wasserstein distance between the
Ali--Mikhail--Haq copula and the uniform copula in terms of $\theta$.
For $-1 < \theta < 1$ {we bound}
$
{\E_{{ U(0,1)^2}} \norm{ \grad c}^2 }
  \leq  \frac{8}{3} \frac{\theta^2}{(1-|\theta|)^{4}}
$ and
$$  \mathcal{W}_1({P^V, P^U} )  \leq {\frac{2 \sqrt{8} }{\pi^2 \sqrt{3}} } \, |\theta|  \{  1  - |\theta| \}^{-2}.$$
This bound {decreases to 0 for $\theta \rightarrow 0$, indicating
  agreement with the uniform copula for $\theta=0$}. {To our
  knowledge there is no explicit formula for $\mathcal{W}_1({P^V, P^U}
  )$ available; \cite{hallin2020multivariate}
 simulate the Wasserstein distance under Gaussian marginals and discuss simulation strategies.}

 \end{example}

\bigskip
\noindent {\bf{Acknowledgements.}} {We are indebted to Andreas Kyprianou for his patient and considerate handling of this paper.} GM gratefully acknowledge support by the Fonds
de la Recherche Scientifique -- FNRS under Grant MIS F.4539.16. YS
acknowledges support by the Fonds de la Recherche Scientifique -- FNRS
under Grant CDR/OL J.0197.20, as well as the ULB ARC Consolidator
grant.  GR acknowledges partial support from EPSRC grants EP/T018445/1, EP/R018472/1, and EP/X0021951,}
and the Alan Turing Institute.
{MR would like to thank Oliver Dragi\v{c}evi\'c for his help with
Sobolev spaces.} We also thank Max Fathi, Christophe Ley, Gilles Mordant
and Guillaume Poly for interesting discussions, as
well as Lester Mackey and Steven Vanduffel for suggesting some
references which we had overlooked.

 \addcontentsline{toc}{section}{Bibliography}

   \bibliographystyle{abbrv}

\appendix

\section{Some more remarks about our function spaces}
\label{sn:somemore}

{As indicated in Remark \ref{rem:rem33}, here is a more detailed
discussion about our choices of ACL-type function classes, with 
$ \ACL $ standing for `absolutely continous on lines'.}
{In particular, we discuss 
relationships to the Sobolev spaces.}

\begin{prop}
\label{prop:D0}
  Let $ \Omega \subseteq \R^d $ be an open set and let
  $ u \in \ACL(\Omega) $ or
  $ u \in W^{1,1}_\loc(\Omega) $. Suppose that $ u $ is constant on some
  set $ A \subseteq \Omega $. Then for each $ e \in S^{d-1} $,
  $ \partial_e u(x) = 0 $ for Lebesgue-almost all $ x \in A $.
\end{prop}

\begin{proof}
  Let $ u(x) = c $ for all $ x \in A $.
  Taking a line $ L $ parallel to $ e $, observe that almost all
  $ x \in A \cap L $ (with respect to the one-dimensional Lebesgue
  measure) are points of Lebesgue density $ 1 $ for the set
  $ A \cap L $ and consequently for the set
  $ \{ x \in \Omega \cap L \sth u(x) = c \} $,
  considering the one-dimensional Lebesgue density in direction $ e $
  (see, e.~g., Corollary~B.121 of \cite{Leo} or Corollary~3 in
  Section~1.7.1 of \cite{EG}). However, if $ x $ is such a point and $ u $
  is differentiable at $ x $ in direction $ e $, we have
  $ \partial_e u(x) = 0 $. The proof is now completed by Fubini's theorem.
\end{proof}


\noindent
{The next result, which is a {counterpart of Proposition \ref{prop:D0},} 
shows that the ACL functions `almost' preserve a fundamental property of
the $ \Cont^1 $ functions. This property is crucial 
{to the proof of} 
Proposition \ref{sec:diff-stein-oper-directional}.}

\begin{prop}
\label{prop:D0Connected}
  Let $ \Omega \subseteq \R^d $ be a connected open set and let $ u \in \ACL(\Omega) $ be such that $ \partial_i u = 0 $ Lebesgue-almost
  everywhere on $ \Omega $ for all
  $ i = 1, 2, \ldots, n $\spacefactor=3000{}
  Then $ u $ is Lebesgue-almost everywhere constant on $ \Omega $.
\end{prop}%

\begin{proof}
  First, consider the case where $ \Omega $ is a cube parallel to the
  coordinate directions. Then the result can be proved by induction.
  The case $ d = 1 $ is immediate. For the induction step from $ d - 1 $
  to $ d $, observe that by Fubini's theorem, there exists a hyperplane
  $ H $ perpendicular to $ e_d $ and intersecting $ \Omega $, such that
  for all $ i = 1, 2, \ldots, d - 1 $, $ \partial_i u = 0 $
  almost everywhere on $ H $ with respect to the $ (d - 1)$-dimensional
  Hausdorff measure. By the induction hypothesis, $ u $ is constant Lebesgue-almost
  everywhere on $ H $.
  Next, for almost all
  lines $ L $ parallel to $ e_d $ and intersecting $ \Omega $, $ u $ is
  absolutely continuous and therefore constant on $ L $. The proof is now
  completed by Fubini's theorem.

  For the general case, observe that $ \Omega $ can be covered by
  a countable family $ \mathcal Q $ of open cubes contained in
  $ \Omega $. By the above, there exists a Lebesgue-null
  set $ N $, such that for any $ Q \in \mathcal Q $, $ u $ is
  constant on $ Q \setminus N $. Since $ \Omega $ is connected, any
  two points $ x, y \in \Omega \setminus N $ are linked by a finite
  sequence of overlapping  cubes belonging to $ \mathcal Q $. Clearly,
  if two open cubes overlap, they also contain a common point which is not
  in $ N $. Therefore, $ u(x) = u(y) $.
\end{proof}%
}


\begin{prop}
\label{prop:W11c0}
  Let $ \Omega \subseteq \R^d $ be an open set and let
  $ u \in W^{1,1}_\loc(\Omega) $ be compactly supported. Then we have
  $ u \in W^{1,1}(\Omega) $ and $ \int_\Omega \partial_e u = 0 $ for all
  $ e \in S^{d-1} $.
\end{prop}

\begin{proof}
  Since $ u $ has a compact support, so does $ \partial_e u $. Therefore,
  $ u \in W^{1,1}(\Omega) $.
  By Exercise~C.23 in \cite{Leo}, there exists a cut-off function
  $ v \in \Cont^\infty_\comp(\Omega) $, which equals $ 1 $ on the supports
  of $ u $ and $ \partial_e u $. By the very definition of the weak
  derivative, we then have
  \[
    \int_\Omega \partial_e u
    =
    \int_\Omega (\partial_e u) v
    =
    - \int_\Omega u (\partial_e v)
    =
    0 \, .
  \]
\end{proof}

\noindent
{In the case that $\Omega = \R^d$ we have the following result.} 

\begin{prop}
\label{prop:W110Entire}
  For all $ u \in W^{1,1}(\R^d) $ and all $ e \in S^{d-1} $, we have
  $ \int \partial_e u = 0 $.
\end{prop}

\begin{proof}
  By Proposition~\ref{prop:W11c0}, the assertion is true for all
  $ u \in \Cont^\infty_\comp(\R^d) $. However, by Theorem~11.35 of
  \cite{Leo}, $ \Cont^\infty_\comp(\R^d) $ is dense in $ W^{1,1}(\R^d) $
  with respect to the appropriate Sobolev norm. This completes the proof.
\end{proof}

\noindent
{The next group of results clarifies relationships between the ACL-type spaces and the Sobolev spaces.}

\begin{prop}
\label{prop:DL1loc}
  For any open set $ \Omega \subseteq \R^d $, we have
  $ \ACL^1_\loc(\Omega) \subseteq L^1_\loc(\Omega) $.
\end{prop}

\begin{proof}
  {We need to prove that each $ x \in \Omega $ has an open neighbourhood
  $ U $, such that the restriction of $ u $ to $ U $ belongs to
  $ L^1(U) $. Now choose $ U $ to be a bounded rectangle and apply
  the basic Poincar\'e inequality for rectangles (see, e.~g.,
  Exercise~13.34 in [Leo]), noting that the restriction of $ u $ to
  $ U $ belongs to $ \ACL^1(U) $.}
\end{proof}

\begin{prop}
\label{prop:ACL1HomW}
  For an open set $ \Omega \subseteq \R^d $, $ \dot W^{1,1}(\Omega) $ is
  precisely the set of all Borel measurable functions
  $ u \Colon \Omega \to \R $ which have a version
  $ \bar u \in \ACL^1(\Omega) $. Moreover, the classical directional
  derivatives of $ \bar u $ are also its corresponding weak derivatives.
\end{prop}

\begin{proof}
  {The result is essentially Theorem~11.45 of \cite{Leo}, but
  we need to modify it from the usual Sobolev space $ W^{1,1} $
  to the homogeneous Sobolev space $ \dot W^{1,1} $ and, most
  important, from functions which are absolutely continuous on
  almost all lines \emph{parallel to the coordinate axes} to ACL
  functions.

  Firstly, suppose that $ u $ has a version $ \bar u \in \ACL^1(\Omega) $.
  By Proposition \ref{prop:DL1loc}, $ \bar u \in L^1_\loc(\Omega) $,
  so that each $ x \in \Omega $ has an open neighbourhood $ U $
  such that the restriction of $ \bar u $ to $ U $ belongs to $ L^1(U) $.
  By Theorem~11.45 of \cite{Leo}, the classical directional derivatives of
  $ \bar u $ are then also its corresponding weak derivatives considered on
  $ U $. However, as the weak derivative is a local concept,
  the weak derivatives can be considered on the whole $ \Omega $.
  By assumption, all directional derivatives belong to $ L^1(\Omega) $.
  Therefore, $ \bar u $, and therefore $ u $, belongs to
  $ \dot W^{1,1}(\Omega) $.}

  {Now we turn to the non-trivial part, the converse.}
  Thus, take $ u \in \dot W^{1,1}(\Omega) $. As in Step~1 of the proof of
  Theorem~11.45 of \cite{Leo}, the function $ \bar u $ is constructed
  by means of the \emph{standard mollifier}:
  \[
    \eta(z)
    :=
    \left\{ \begin{array}{c@{\quad}l}
      a_d \exp \left( - \frac{1}{1 - \| z \|^2} \right) &
        \text{if $ \| z \| \le 1 $,}
    \\ \noalign{\vskip 0.5ex}
      0                                                 &
        \text{if $ \| z \| \ge 1 $,}
   \end{array} \right.
  \]
  where $ c_d $ is chosen so that $ \int_{\R^d} \eta(z) \,\dl z = 1 $.
  As usual, the mollified functions are then defined as
  $ u_\e(x) := \int_{\R^d} u(x + \e z) \, \eta(z) \,\dl z $.

  In Step~1 of the proof of Theorem~11.45 of \cite{Leo}, a function
  $ \bar u $, which is locally absolutely continuous on Lebesgue-almost
  all lines parallel to the coordinate directions, is constructed as the
  limit of the sequence $ u_{\e_n} $ for some sequence $ \{ \e_n \}_n $.
  However, we claim that the function
  \[
    \bar u(x)
    :=
    \left\{ \begin{array}{c@{\quad}l}
      \lim\limits_{\e \downarrow 0} u_\e(x) &
        \text{if this limit exists,}
    \\
      0
        & \text{otherwise}
    \end{array} \right.
  \]
  is already locally absolutely continuous on Lebesgue-almost all lines
  parallel to $ e $ for all $ e \in S^{d-1} $. Before turning to the proof
  of this claim, notice that $ \bar u $ is Borel measurable, as
  the limit can be taken only over rational $ \e $ (the map
  $ \e \mapsto \eta_\e $ is continuous in the supremum norm).

  Now fix $ e \in S^{d-1} $. Without loss of generality, we may assume
  that $ e = e_d $. The crucial reason why the function $ \bar u $ works
  for all directions is that the set
  $ \{ x \in \Omega \sth \bar u(x) \ne u(x) \} $ has
  $ (d - 1) $-dimensional Hausdorff measure zero. This essentially follows
  from Theorem~1 in Section~4.8 of \cite{EG} and Theorem~3 in
  Section~5.6.3 ibidem. However, the first one cannot be applied directly
  and some additional thought is needed:
  \begin{list}{\textbullet}%
    {\labelwidth 1em \topsep 0ex \parsep 0ex \itemsep 0.5ex}
    \item Theorem~1 in Section~4.8 of \cite{EG} concerns convergence of
    the functions
    \[
      (u)_{x,r} := b_d \int_B u(x + r y) \,\dl y
    \]
    rather than the functions $ u_\e(x) $ (here, $ B $ denotes the unit
    ball in $ \R^d $ and $ b_d $ its reciprocal volume). However, a
    routine application of Fubini's theorem yields
    \[
      u_\e(x)
      =
      \frac{2 a_d}{b_d} \int_0^1
        (u)_{x, t \e} \, \frac{t^{d+1}}{(1 - t^2)^2}
        \exp \left( - \frac{1}{1 - t^2} \right)
      \dl t \, .
    \]
    Now if $ \lim_{r \downarrow 0} (u)_{x,r} = u(x) $ for a fixed $ x $,
    the values $ (u)_{x,r} $ must be bounded for sufficiently small $ r $,
    so that the values $ (u)_{x, t \e} $ are uniformly bounded in
    $ t \in [0, 1] $ for sufficiently small $ \e $. Since
    $ \int_0^1 \frac{t^{d+1}}{(1 - t^2)^2}
    \exp \left( - \frac{1}{1 - t^2} \right) \dl t < \infty $,
    the desired convergence follows from the dominated convergence theorem.
    \item Theorem~1 in Section~4.8 of \cite{EG} assumes that
    $ u \in W^{1,1}(\R^d) $, but here, we only have
    $ u \in \dot W^{1,1}(\Omega) $. Although the convergence of $ u_\e $ 
    towards $ u $ is an entirely local property, we still need to
    construct appropriate functions which are, along with their weak
    first-order partial derivatives, globally integrable. Thus, expressing
    $ \Omega $ as a countable union of open balls $ U_n $, it suffices to
    prove that for each $ n $, the $ (d - 1) $-dimensional Hausdorff
    measure of the set $ \{ x \in U_n \sth \bar u(x) \ne u(x) \} $ equals
    zero. Clearly, the restriction of $ u $ to $ U_n $ belongs to
    $ W^{1,1}(U_n) $. By Theorem~1 in Section~4.4 of \cite{EG} or
    Theorem~13.17 of \cite{Leo}, the latter restriction can be extended to
    a function in $ W^{1,1}(\R^d) $, so that, finally, by Theorem~1 in
    Section~4.8 of \cite{EG} and Theorem~3 in Section~5.6.3 ibidem, the
    $ (d - 1) $-dimensional Hausdorff measure of the set
    $ \{ x \in U_n \sth \bar u(x) \ne u(x) \} $ vanishes.
  \end{list}

  Once we know that $ (d - 1) $-dimensional Hausdorff measure of the set
  $ \{ x \in \Omega \sth \bar u(x) \ne u(x) \} $ equals zero, Corollary~1
  in Section~2.4.1 of \cite{EG} yields that there exists a Borel set
  $ E_1 \subseteq \R^{d-1} $, such that the Lebesgue measure of
  $ \R^{d-1} \setminus E_1 $ vanishes and that
  $ \bar u(x', x_d) = u(x', x_d) $ for all $ x' \in E_1 $ and all
  $ x_d \in \R $ with $ (x', x_d) \in \Omega $.

  Following Step~1 of the proof of Theorem~11.45 of \cite{Leo}, we find
  that there exist a sequence $ \{ \e_n \}_n $ converging to zero
  and a Borel set $ E_2 \subseteq \R^{d-1} $, such that the Lebesgue
  measure of $ \R^{d-1} \setminus E_2 $ vanishes and such that
  \[
    \int_{(x^\prime, x_d) \in \Omega} \bigl\| \grad u(x', x_d) \bigr\|
    \,\dl x_d
    <
    \infty
  \]
  and
  \[
    \lim_{n \to \infty}
    \int_{(x^\prime, x_d) \in \Omega_{\e_n}} \bigl\|
    \grad u_{\e_n}(x', x_d) - \grad u(x', x_d) \bigr\| \,\dl x_d
    =
    0
  \]
  for all $ x' \in E_2 $; here, $ \Omega_{\e_n} $ is the set of all points
  $ \Omega $ with the distance from $ \R^d \setminus \Omega $ strictly
  greater than $ \e $ (for $ \Omega = \R^d $, we take
  $ \Omega_{\e_n} = \R^d $).

  Now take $ x' \in E_1 \cap E_2 $ and claim that the function
  $ x_d \mapsto \bar u(x', x_d) $ is absolutely continuous on all compact
  intervals $ I $ with $ \{ (x', x_d) \in \Omega \} $ for all
  $ x_d \in I $. Take $ a, b \in I $. Since the functions $ u_{\e_n} $ are
  in $ \Cont^\infty(\Omega_{\e_n}) $, we have, by the fundamental theorem
  of calculus,
  \[
    u_{\e_n}(x', b) - u_{\e_n}(x', a)
    =
    \int_a^b \partial_d u_{\e_n}(x', t) \,\dl t
  \]
  for all $ n $ with $ \{ x' \} \times [a, b] \subseteq \Omega_{\e_n} $
  (which is true for all sufficiently large $ n $). Taking the limit, we
  find that
  \[
    \bar u(x', b) - \bar u(x', a)
    =
    \int_a^b \partial_d u(x', t) \,\dl t
    \, .
  \]
  As $ \bigl| \partial_d u(x', t) \bigr| \,\dl t < \infty $,
  the function $ x_d \mapsto \bar u(x', x_d) $ is absolutely continuous on
  $ I $ and its classical derivative (extended arbitrarily to the whole
  set $ \{ x_d \sth (x', x_d) \in \Omega \} $ and possibly altered on a
  Lebesgue-null set) are also its corresponding weak derivatives.
  This proves the result.
\end{proof}

\begin{cor}
\label{cor:ACL1HomWloc}
  For an open set $ \Omega \subseteq \R^d $, $ W^{1,1}_\loc(\Omega) $ is
  precisely the set of all Borel measurable functions $ \Omega \to \R $
  which have a version in $ \ACL^1_\loc(\Omega) $.\qed
\end{cor}

\begin{cor}
\label{cor:ACL1HomW0}
  For an open set $ \Omega \subseteq \R^d $, $ \dot W^{1,1}_0(\Omega) $
  is precisely the set of all Borel measurable functions
  $ \Omega \to \R $ which have a version in $ \ACL^1_0(\Omega) $.\qed
\end{cor}


\end{document}